\documentclass[graybox]{svmult}
\usepackage{type1cm}        % activate if the above 3 fonts are
\usepackage{makeidx}         % allows index generation
\usepackage{graphicx}        % standard LaTeX graphics tool
\usepackage{tikz}
\usepackage{multicol}        % used for the two-column index
\usepackage[bottom]{footmisc}% places footnotes at page bottom
\usepackage{amsmath,amssymb,wasysym}

\usepackage{newtxtext}       % 
\usepackage[varvw]{newtxmath}       % selects Times Roman as basic font

\newcommand\FD{\mathrm{FD}}
\newcommand\FT{\mathrm{FT}}
\newcommand\FQ{\mathrm{FQ}}
\newcommand\FH{\mathrm{FH}}
\newcommand\FP{\mathrm{FP}}

\makeindex   

\begin{document}

\title*{Farey graphs and geodesic expansions of complex continued fractions}
\titlerunning{Farey graphs and geodesic expansions of complex continued fractions} 
\author{Hitoshi Nakada, Rie Natsui, and J\"org Thuswaldner}
% Use \authorrunning{Short Title} for an abbreviated version of
% your contribution title if the original one is too long
\institute{Hitoshi Nakada \at Department of Mathematics, Keio University, Kohoku-ku, Yokohama, Japan\\ 
\email{nakada@math.keio.ac.jp}
\and Rie Natsui \at Department of Mathematics, Physics and Computer Science, 
Japan Women's University, Bunkyo-ku, Tokyo, Japan\\ 
\email{natsui@fc.jwu.ac.jp}
\and J\"org Thuswaldner \at Chair of Mathematics, Statistics, and Geometry, Montanuniversit\"at Leoben, A-8700 Leoben, Austria \\ \email{joerg.thuswaldner@unileoben.ac.at}}
%
% Use the package "url.sty" to avoid
% problems with special characters
% used in your e-mail or web address
%
\maketitle

%\begin{center}
    \vskip -2.5cm
    {\it Dedicated to Vitaly Bergelson on the occasion of his 75th birthday}
%\end{center}
\vspace{2cm}

\abstract{We discuss complex Farey graphs for the Euclidean imaginary quadratic number fields $\mathbb Q(\sqrt{-d})$, $d\in\{1, 2, 3, 7, 11\}$. We  study hyperbolic versions of A.~Schmidt's Farey polygons living in $3$-dimensional hyperbolic space~$\mathbb{H}^3$. Using these Farey polygons we recover tessellations of the hyperbolic plane $\mathbb{H}^2$ that are defined by the action of the Hecke groups $H_4$ and $H_6$ and have been studied earlier by I.~Short and M.~Walker. Moreover, hyperbolic Farey polygons allow us to define polyhedra that induce Farey tessellations of $\mathbb{H}^3$ by the action of certain Bianchi groups. Using complex Farey graphs we consider geodesic complex continued fraction expansions. Our method provides a different and more general approach as the one from the discussion by M.~Hockman. 
%200 Words at most
}

\section*{Introduction}
\label{sec:introduction}

The main purpose of this paper is to give a brief explanation of some relations between complex versions of the Farey graph, the Farey tessellation, and complex continued fractions for  Euclidean imaginary quadratic number fields. In the case of the Gaussian field, Hockman \cite{H-1,H-2} already discussed this topic. We consider all Euclidean imaginary quadratic number fields by a different approach.  

The notion of Farey sequence, Farey graph, and Farey tessellation, as well as their connection with continued fraction expansions is well-studied in the field of the rational numbers~$\mathbb{Q}$, see {\em e.g.}, \cite{B-H-S,J-S-W,M}.  We give a brief sketch of this classical situation. Usually, one starts with $0 = \tfrac{0}{1}$ and $1=\tfrac{1}{1}$, which constitute the {\em Farey sequence} of level~$1$.  Then $\tfrac{0}{1}, \,\, \tfrac{1}{2} = \frac{0 + 1}{1+1}, \,\, \tfrac{1}{1}$ is 
of level~$2$, and $\tfrac{0}{1}, \,\, \tfrac{1}{3} = \tfrac{0+1}{1+2}, \,\, \tfrac{1}{2}, \,\, \tfrac{2}{3} = \tfrac{1+1}{2+ 1}, \,\, \tfrac{1}{1}$ is of level~$3$.  We define the Farey sequence of level~$k$ recursively from that of level~$k-1$ by inserting the {\em mediant fraction} $\tfrac{a + c}{b + d}$ between consecutive rational numbers $\tfrac{a}{b}$ and $\tfrac{c}{d}$ in the sequence of level~$k-1$.  Every rational number in $[0, 1]$ appears in the Farey sequence of level~$k$ for some $k \ge 1$, and each appearing fraction $\tfrac ab$ satisfies $\gcd(a,b)=1$.  It is easy to extend Farey sequences to all rational numbers by taking $\mathbb Z$ as the Farey sequence of level~$1$ and iterating in the same way.  One can check inductively that for fractions $\tfrac{a}{b}$ and $\tfrac{c}{d}$ that are ``Farey neighbors'' in some level~$k$ we always have $|bc-ad|=1$.

We now discuss the idea of Farey neighbors from a different point of view. Setting $i=\sqrt{-1}$ let $\mathbb H^{2}=\{ z = x + yi: x, y \in \mathbb R \,\,\mbox{and} \,\,y> 0\}$ be the hyperbolic plane with the Poincar\'{e} metric $ds^2 = \tfrac{dx^2+dy^2}{y}$.  
The boundary of $\mathbb H^{2}$, denoted by $\partial \mathbb H^{2}$, is $\mathbb R \cup \{\infty\}$.  A geodesic curve $S$ is either a half-circle or a straight half-line, perpendicular to the $\partial \mathbb H^{2}$, {\em i.e.},
\[
\begin{split}
S &= \big\{x+ yi\in \mathbb{C}\,: \, y>0\text{ and }  |x+yi - a|^{2} = r^{2}\big\} \;  \text{for some $a\in \mathbb{R}$ and $r>0$, or} \\
S &= \{x + yi \in \mathbb{C} \,:\, y>0\} \;  \text{for some $x\in \mathbb{R}$}.  
\end{split}
\]
We denote $S$ by $(a - r, a + r)$ in the first case and $(\infty,  x)$ in the second case. We interpret 
\[
{\rm PSL}(2, \mathbb Z) = \left\{\begin{pmatrix} a & b \\c & d \end{pmatrix} : 
a, b, c, d \in \mathbb Z, ad-bc = 1 \right\} /\left\{ \begin{pmatrix}\pm 1 & 0 \\ 0 &\pm 1 
\end{pmatrix} \right\}
\] 
as a set of M\"obius transformations defined by  $A(z) = \tfrac{a z + b}{cz + d}$ for $A = \begin{pmatrix}a & b \\ c & d \end{pmatrix}$ 
(note that $A$ and $-A$ define the same M\"obius transformation and and we may choose either of these representatives to define the fractional transformation). 

The {\em Farey graph} is the graph with vertices in $\mathbb Q^{*} = \mathbb Q \cup \{\infty\}$ and set of edges $\mathcal E = \{A(\infty) \to A(0) : A  \in {\rm PSL}(2, \mathbb Z)\}$. Because $\det A = 1$ one can show that $\frac ab, \frac cd \in \mathbb{Q}$ are connected by an edge if and only if $|bc-ad|=1$. Thus  $\frac ab$ and $\frac cd$ are connected by an edge in $\mathcal E$ if and only if $\frac ab$ and $\frac cd$ are Farey neighbors in a Farey sequence of some level $k$. In this paper, we regard the Farey graph just as the set of edges $\mathcal E$ and often (if the direction of the edge plays no role) interpret $\alpha\to\beta\in\mathcal{E}$ as the geodesic $(\alpha,\beta)$ in $\mathbb{H}^2$. Since the edges $\infty \to 0$ and $0 \to \infty$ correspond to the identity matrix and to   $\begin{pmatrix} 0 & -1 \\ 1 & 0 \end{pmatrix}$, respectively, $z_{1} \to  z_{2} \in \mathcal E$ is equivalent to $z_{2} \to  z_{1} \in \mathcal E$.\footnote{Although this means that $\mathcal{E}$ is an undirected graph it is convenient for us to use directed edges.}  Consider the hyperbolic triangle $\mathcal C = \left\{z = x + yi : 0 < x <1, \,\, y > 0, \,\, (x -\tfrac{1}{2})^{2} + y^{2} > \tfrac{1}{2} \right\}$. It is well known 
that 
\[
\bigcup_{A \in {\rm PSL}(2, \mathbb Z)} A \cdot {\mathcal C}^{\mathrm{cl}}
= \bigcup_{A \in {\rm PSL}(2, \mathbb Z)} \{ A(z) : z \in {\mathcal C}^{\mathrm{cl}} \} = 
\mathbb H^{2}, 
\] 
where ${B}^{\mathrm{cl}}$ is the closure of a given set $B$.  Also we see that $\mathcal C \cap A \mathcal C$ is either equal to $\mathcal{C}$ or to the empty set. In other words, $\mathcal C$ {\em tessellates} $\mathbb{H}^2$ by the action of the group $ {\rm PSL}(2, \mathbb Z)$. Because the edges of $\mathcal{E}$ are equal to the edges of the hyperbolic triangles $A\cdot C$, $A\in {\rm PSL}(2, \mathbb Z)$, this shows that for any $A_{1}, \,A_{2} \in {\rm PSL}(2, \mathbb Z)$ the intersection $(A_{1}(\infty), A_{1}(0)) \cap (A_{2}(\infty),A_{2}(0))$ is either equal to $(A_{1}(\infty), A_{1}(0))$ or empty,\footnote{Note that $\partial\mathbb{H}^2$ is not part of the geodesic curves $(\alpha,\beta)$ forming the edges $\alpha\to \beta$ of $\mathcal{E}$.} thus the Farey graph is a planar graph. An illustration of a part of the Farey graph, which also shows
\begin{figure}[h]
\includegraphics[width=0.9\textwidth]{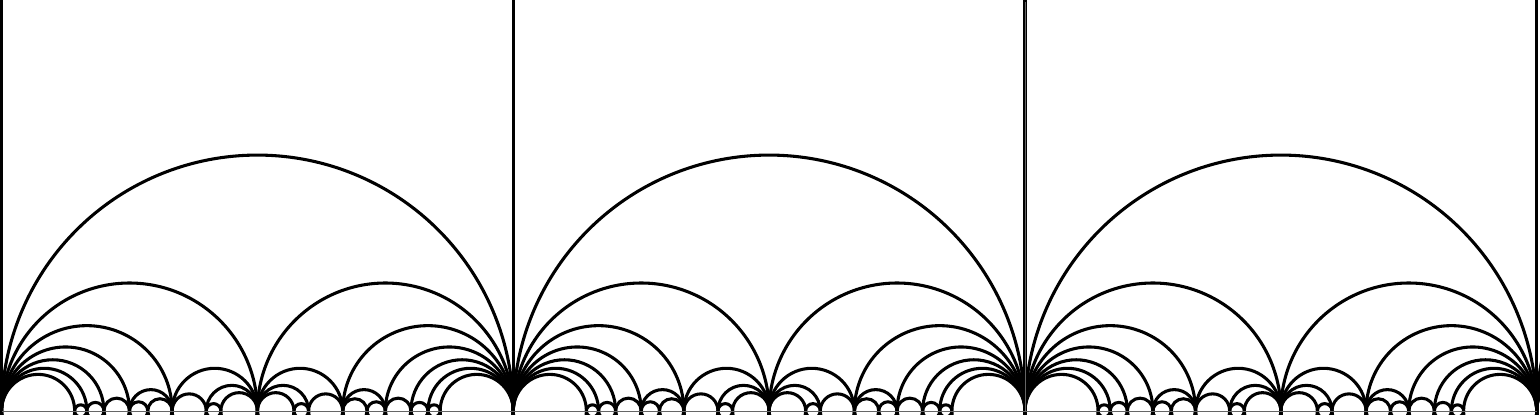} \qquad
\caption{The classical Farey tessellation of $\mathbb{H}^2$.
\label{fig:ClassicalFarey}}
\end{figure}
the Farey tessellation $\mathcal{T}=\{A\cdot C\,:\, A\in {\rm PSL}(2, \mathbb Z)\}$ of $\mathbb{H}^2$ into hyperbolic triangles, is provided in Figure~\ref{fig:ClassicalFarey}. It is well known that this tessellation is strongly related to the tessellaton of $\mathbb{H}^{2}$ by the orbit of a fundamental domain of the modular group (see {\em e.g.} Series~\cite{Series}).

For a rational number $r = \tfrac{p}{q}\in \mathbb{Q}$ with $\gcd(p,q)=1$, the {\em Ford circle} $\mathcal F(r)$ associated with $r$ is the Euclidean circle defined by 
\[
\mathcal F(r) = \Bigg\{x+yi \in \mathbb{C} :  \,\,y \ge 0, \,\, 
\bigg|x+yi - \bigg(r + \frac{i}{2 q^{2}}\bigg)
\bigg| = \frac{1}{2 q^{2}} \Bigg\},
\]
\begin{figure}[h]
\includegraphics[width=0.9\textwidth]{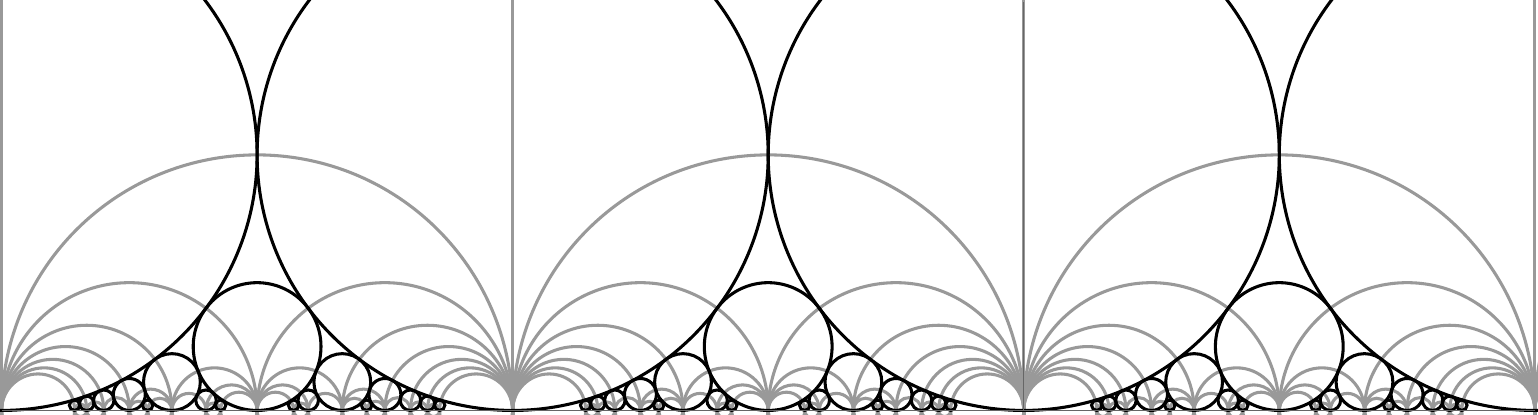} \qquad
\caption{Ford circles (black) and their interplay with the Farey graph (shaded).
\label{fig:ClassicalFord}}
\end{figure}
and $\mathcal F(\infty) = \{ z = x +i : x \in \mathbb R \} \cup \{\infty\}$.   
We refer to Ford \cite{F} for the original idea of the Ford circle.
One can show that the set of Ford circles is equal to $\{A\cdot \mathcal F(\infty) :  A \in {\rm PSL}(2, \mathbb Z) \}$. The (degenerate) Ford circle $\mathcal{F}(\infty)=\mathcal{F}\big(\frac10\big)$ intersects $\mathcal F\big(\frac01\big)$ in~$i$, which is contained in the geodesic $\big(\frac10,\frac01\big)$.  Because for each edge $\tfrac ab \to \tfrac cd\in\mathcal{E}$ there is $A = \begin{pmatrix}a & c \\ b & d \end{pmatrix} \in {\rm PSL}(2, \mathbb Z)$ such that $\tfrac ab=A(\infty)$ and $\tfrac cd=A(0)$, we see that $\tfrac ab \to \tfrac cd\in\mathcal{E}$ implies that  $A\cdot\mathcal{F}\big(\frac10\big)=\mathcal F\big(\tfrac{a}{b}\big)$ and  $A\cdot\mathcal{F}\big(\frac01\big)=\mathcal F\big(\tfrac{c}{d}\big)$ are tangent and intersect in a point contained in the geodesic $\big(\tfrac{a}{b}, \tfrac{c}{d}\big)$. Moreover,  if 
$\tfrac ab \to \tfrac cd\not\in\mathcal{E}$, then one can check by simple calculation that $\mathcal F\big(\tfrac{a}{b}\big)$ and $\mathcal F\big(\tfrac{c}{d}\big)$ are disjoint. See Figure~\ref{fig:ClassicalFord} for an illustration of this interplay between Ford circles and the Farey graph.

It is well known that the Farey graph, the Farey tessellation, and Ford circles have tight relations to continued fraction algorithms. In particular, one can prove that a continued fraction expansion
\begin{equation}\label{eq:cf}
[a_0;a_1,\ldots,a_n]=
\displaystyle a_0+ \frac1{a_1+ 
 \displaystyle \frac1{\ddots + 
 \displaystyle \frac1{a_n}}}
\end{equation}
with $a_0,a_1,\ldots, a_n \in \mathbb {Z}$ corresponds to a walk\footnote{Recall that a {\em walk} in a graph is a consecutive sequence of edges. A {\em path} is a walk in which no vertex occurs more than one time. The shortest walk from a vertex $a$ to a vertex $b$ (which has to be a path) is called a {\em geodesic path}. Keep in mind  the ambigous use of the word {\em geodesic} throughout the paper.}
of length $n+1$ in the Farey graph leading from $\infty$ to $x=[a_0;a_1,\ldots, a_n]$. This is explored in  Beardon {\em et al.}~\cite{B-H-S}. It is natural to ask which continued fraction expansion is the shortest in the sense that it corresponds to the shortest walk from $\infty$ to $x$ in the Faray graph. Such continued fraction expansions are called {\em geodesic continued fraction expansions}. In \cite{B-H-S}, rational geodesic continued fraction expansions are characterized.

In the present paper we extend these ideas to Euclidean imaginary quadratic number fields. This will lead to new versions of the Farey graph that live in the 3-dimensional hyperbolic upper half space $\mathbb{H}^3$. Moreover, Ford circles will become Ford spheres and the associated Farey tessellation is a tessellation of $\mathbb{H}^3$ by hyperbolic polyhedra under the action of Bianchi groups of Euclidean imaginary quadratic number fields (see \cite{H-1,H-2}, where aspects of the Gaussian field are treated; a different generalization of Farey graphs is studied in~\cite{FKST}). This Farey tessellation is related to certain ``modular'' tessellations of $\mathbb{H}^3$ by Bianchi groups as studied for instance by Swan~\cite{Swan:71}. The faces of the hyperbolic polyhedra are intimately related to the Farey polygons studied by Schmidt~\cite{SchmidtAsmus:69,SchmidtAsmus:78}. Indeed, as discussed by Nakada~\cite{N-Monatshefte,N}, the ``subdivision'' of Schmidt's Farey polygons becomes more natural when viewed in the hyperbolic upper half space $\mathbb{H}^3$.  
We also refer to \cite{C, hatcher1983hyperbolic, Marklof:10, S-vS-Z, Stange:18, Stange2014} for some related topics. 

Following \cite{B-H-S,H-1,H-2}, we will use Farey graphs and Ford spheres to establish some results on geodesic complex continued fraction expansions in Euclidean imaginary quadratic number fields. In particular, we are able to show that certain continued fraction algorithms that go back to Hurwitz~\cite{H:1887} and Lakein~\cite{lakein1973approximation} always produce geodesic continued fraction expansions in the sense mentioned above.

Nakada~\cite{N2} applied the results on geodesic continued fraction expansions of \cite{B-H-S} in order to solve a conjecture of Kraaikamp {\em et al.}~\cite{KSS} on the entropy of $\alpha$-continued fraction algorithms. In particular, it is shown in \cite[Theorem~1]{N2} that the expansions produced by the $\alpha$-continued continued fraction map $T_\alpha(x) = \big| \frac1x\big| - \big\lfloor \big| \frac1x\big| + 1 - \alpha \big\rfloor$ are always geodesic if $\big(\tfrac{\sqrt 5 - 1}{2}\big)^2 \le \alpha \le \tfrac{\sqrt 5 - 1}{2}$. This shows that our results have the potential to be applied to related questions in the complex case. In a forthcoming work we want to come back to such applications.

\section{The generalized setting in the imaginary quadratic case}\label{sec:2}

We now want to set up the analogs of the Farey graph and Ford spheres in the case of imaginary quadratic number fields.
Let $d\ge 1$ be a squarefree integer and consider the number field $\mathbb{Q}(\sqrt{-d})$. If we put
\[
\omega=\omega_d=\begin{cases}
\sqrt{-d}, & \hbox{for } d \not\equiv 3\bmod 4, \\
\frac{1+\sqrt{-d}}2, & \hbox{for } d \equiv 3\bmod 4,
\end{cases}
\] 
the ring of algebraic integers in $\mathbb Q(\sqrt{-d})$ is given by 
$\mathbb{Z}[\omega_d]$. The $3$-dimensional hyperbolic upper half space is defined by $\mathbb H^{3} = \{x_{1} + x_{2}{i} + x_{3}{j} : 
x_{1}, x_{2}, x_{3} \in \mathbb R, \,\, x_{3}> 0 \}$, where $i$ and $j$ are elements of the usual basis $\{1, i, j, k\}$ of the quaternion field.   Note that $\mathbb Q(\sqrt{-d})$ is Euclidean if and only if $d\in \{1, 2, 3, 7, 11\}$.   We define the Farey graph for $\mathbb Q(\sqrt{-d})$ for an arbitrary squarefree integer $d\in\mathbb{N}$, however, the Euclidean case will be of particular interest.

${\rm PSL}(2, \mathbb{Z}[\omega_d])$ denotes the {\em Bianchi group} of M\"obius transformations with coefficients in $\mathbb{Z}[\omega_d]$ defined on  $\mathbb H^{3}$, whose matrix representation is given by 
\[
{\rm PSL}(2, \mathbb{Z}[\omega_d]) = {\rm SL}(2,  \mathbb{Z}[\omega_d])/ 
\left\{ \begin{pmatrix} \pm 1 & 0 \\ 0 & \pm 1 \end{pmatrix}\right\}, 
\]
where ${\rm SL}(2, \mathbb{Z}[\omega_d])$ is the group of $2\times 2$-matrices of determinant $1$ with elements in $ \mathbb{Z}[\omega_d]$.
Thus, for $A= \begin{pmatrix} \alpha_{1 1} & \alpha_{1 2} \\ \alpha_{2 1}& \alpha_{2 2} \end{pmatrix}  \in 
{\rm SL}(2, \mathbb Z[\omega_{d}])$ we get the M\"obius transformation 
\[
A(w) = (\alpha_{1 1}w + \alpha_{1 2})(\alpha_{2 1} w + \alpha_{2 2})^{-1} \quad \mbox{for} \quad w \in \mathbb H^{3}.
\] 
This can also be regarded as a M\"obius transformation on $\mathbb{C}$. We note that if we consider a linear fractional transformation for complex numbers, there is no difference between $A$ and $iA$ as M\"obius transformations on $\mathbb{C}$ but for quaternion numbers, $A$ and $iA$ are not the same.  We refer to Ahlfors~\cite{A} for a detailed discussion of $\mathbb{H}^3$. For $d \in\{1,3\}$, we have to be careful concerning the definition of the M\"obius transformation because $\pm1$ are not the only units of $\mathbb{Z}[\omega_d]$ in these cases.  
For the Gaussian case $d=1$, if $A \in {\rm SL}(2, \mathbb{Z}[\omega_1])$, the M\"obius transformation $A$ maps $\mathbb H^{3}$ onto itself, however, $iA$ has $\det (iA) = -1$ and maps $\mathbb H^{3}$ to the lower half space $\{ x_{1} + x_{2} i + x_{3} j : x_{1}, x_{2}, x_{3} \in \mathbb R,\, x_{3} <0 \}$. This shows that if $\det A = -1$,  then $iA \in {\rm SL}(2, \mathbb{Z}[\omega_1])$ and $iA$ maps $\mathbb H^{3}$ onto itself.  For example, if $A = \begin{pmatrix} -1 & 0 \\ 0 &1 \end{pmatrix}$, then  $A(j) = -j$ but $iA(j)= \begin{pmatrix} -i & 0 \\ 0 &i \end{pmatrix}(j) = j$. 
In the Eisenstein case $d=3$ the set of units is  $\{ \omega_{3}^{\ell}\,:\, 0 \le \ell \le 5\}$. Thus, if  $A = \begin{pmatrix} \alpha_{11} &\alpha_{12} \\ \alpha_{21} & \alpha_{22} \end{pmatrix}$ has elements in $\mathbb{Z}[\omega_3]$ and $\det A \in\{ \omega_{3}^{2}, \omega_{3}^{4}\}$, then we have either $\omega_{3}^{2} A\in {\rm SL}(2,\mathbb{Z}[\omega_3])$ or $\omega_{3} A \in {\rm SL}(2,\mathbb{Z}[\omega_3])$.   For example, if $A = \begin{pmatrix} \omega_{3}^{2} & 0 \\ 0 &1 \end{pmatrix}$, then  $A(j) = \tfrac{-j + \sqrt{-3}j}{2}$ and $\omega_{3}^{2}A(j)= \begin{pmatrix} \omega_{3}^{4}  & 0 \\ 0 &\omega_{3}^{2} \end{pmatrix}(j) = j$.   

\begin{remark} 
For $A\in \mathrm{PSL}(2, \mathbb{Z}[\omega_d])$ depending on the context, we regard $A$ as a matrix and 
also a linear fractional transformation.  For instance, in Definition~\ref{def:fareytriangle} below  $A$ is first used as a matrix and then a linear fractional transformation. Since the role of $A$ will always be clear from the context, there will be no risk of confusion.

Moreover, in the sequel we will use the notation $A\cdot M = \{ A(z) : z \in M\}$ for $M\subset \mathbb{H}^3$. 
\end{remark}

\begin{definition}[Generalization of Farey Graph]
For each squarefree $d\in\mathbb{N}$, the {\em Farey graph}  $\mathcal E_{d}$ is the graph with vertices in ${\rm PSL}(2,  \mathbb{Z}[\omega_d]) (\infty)$  whose set of edges is $\big\{A(\infty) \to A(0)\,: \, A \in {\rm PSL}(2, \mathbb{Z}[\omega_d])\big\}$. 
\end{definition}
%%%%%%%%%%%%%%%%%%%%%%
%%%%%%%%%%%%%%%%%%%%%%
\begin{remark}  In general, if $d$ is a  squarefree integer, then 
by a classical result of Bianchi and Hurwitz, ${\rm PSL}(2, \mathbb{Z}[\omega_d]) (\infty) = \mathbb Q(\sqrt{-d}) \cup \{\infty\} $ if and 
only if the class number of $\mathbb Q(\sqrt{-d})$ is one (see for instance Swan~\cite[Proposition~3.10 and the remark after it]{Swan:71}). Therefore, in the case of the Euclidean imaginary quadratic number fields that are of interest to us, the set of vertices of the Farey graph  $\mathcal E_{d}$ equals $\mathbb Q(\sqrt{-d}) \cup \{\infty\}$.
\end{remark}

The Farey graph $\mathcal E_{d}$ is a directed graph, the direction being defined by $A(\infty)\to A(0)$. However, usually the direction plays no role and   we interpret $\alpha \to \beta$ as the geodesic line $(\alpha,\beta)$ connecting $\alpha$  and $\beta$ in $\mathbb H^{3}$. Note that $(\alpha,\beta)$ is the (possibly degenerate) semi-circle from $\alpha$ to $\beta$ perpendicular to the complex plane $\partial\mathbb H^{3}$. The following result shows that in the Euclidean case the edges of $\mathcal E_{d}$ satisfy the analog of the ``Farey neighbor relation'' from the rational case.

%%%%%%%%%%%%%%%%%%%%%%%%
%%%%%%%%%%%%%%%%%%%%%%%%
\begin{proposition}\label{prop:EdCrit}
For $d\in\{1,2,3,7,11\}$ let $\tfrac{p_{k}}{q_{k}} \in \mathbb Q(\sqrt{-d}) \cup \{\infty\}$ with $\gcd(p_{k}, q_{k})=1$ for $k\in\{1,2\}$.  Then 
$\tfrac{p_{1}}{q_{1}}\to \tfrac{p_{2}}{q_{2}}$ is an edge of $\mathcal E_{d}$ if and only if $|p_{1}q_{2} - p_{2}q_{1}| = 1$ 
\end{proposition}
%%%%%%%%%%%%%%%%%%%%%%%
%%%%%%%%%%%%%%%%%%%%%%%
\begin{proof} 
If $\tfrac{p_{1}}{q_{1}}\to \tfrac{p_{2}}{q_{2}}$ is an edge of the Farey graph $\mathcal{E}_d$, then there exists 
$A \in {\rm SL}(2, \mathbb{Z}[\omega_d])$ such that $\tfrac{p_{1}}{q_{1}}\to \tfrac{p_{2}}{q_{2}}$ equals $A(\infty) \to  A(0)$, 
which implies $|p_{1}q_{2} - p_{2}q_{1}| = 1$.  On the other hand, 
if $|p_{1}q_{2} - p_{2}q_{1}| = 1$, then $A = \begin{pmatrix} p_{1} & q_{2} \\ q_{1} & p_{2} \end{pmatrix}$ maps $\infty$ to 
$\tfrac{p_{1}}{q_{1}}$ and $0$ to $\tfrac{p_{2}}{q_{2}}$.  If $\det A = 1$, then we see $\tfrac{p_{1}}{q_{1}}\to \tfrac{p_{2}}{q_{2}}$ equals $A(\infty) \to  A(0)$.  If not, then we can find a unit $\xi$ of $\mathbb{Z}[\omega_d]$ so that 
${\rm det} \begin{pmatrix}\xi p_{1}, & p_{2} \\ \xi q_{1} & q_{2} \end{pmatrix} = 1$ and 
$\begin{pmatrix}\xi p_{1}, & p_{2} \\ \xi q_{1} & q_{2} \end{pmatrix}$ maps $\infty$ to $\tfrac{p_{1}}{q_{1}}$ and $0$ to $\tfrac{p_{2}}{q_{2}}$. 
\end{proof}     

{\em Ford spheres} (horospheres) in $\mathbb{H}^3$ are the analogs of the Ford circles in $\mathbb{H}^2$ and play a role in our proofs. They are used for instance by Beardon and Short~\cite{Beardon-Short:14} in the context of complex continued fraction algorithms. We slightly generalize their definition.

\begin{definition}[Ford sphere]
For $d\in\{1,2,3,7,11\}$ let $\beta>0$ and $\tfrac pq \in \mathbb{Q}(\sqrt{-d})\cup\{\infty\}$ with $\gcd(p,q)=1$. If $\tfrac pq\not=\infty$, the (Euclidean) sphere $S$ with midpoint $(\frac pq,\frac{\beta}{2|q|^2})$ and radius $\frac{\beta}{2|q|^2}$ is called the \emph{generalized Ford sphere of size $\beta$ for $\frac pq$}. If $\tfrac pq=\infty$, the plane $\{w = x_{1} + x_{2} i + \beta j \}$  is called the \emph{generalized Ford sphere of size $\beta$ for $\infty$}. 

If $\beta =1$ then $S$ is called the \emph{Ford sphere for $\frac pq$}  and $S$ will be denoted by $\mathcal F\big(\tfrac pq\big)$. 
\end{definition}

 The Ford spheres $\mathcal F\big(\tfrac10\big)$ and $\mathcal F\big(\tfrac01\big)$ are tangent to each other and the touching point $0 + 0i + 1j$ is located on the geodesic $\big(\tfrac10,\tfrac01\big)$. It is easy to see that $P=A(\infty)$ implies that $\mathcal F(P) = A\cdot \mathcal F(\infty)$ for $A \in {\rm PSL}(2,\mathbb{Z}[\omega_d])$. Thus the geodesic $\big(\tfrac ab,\tfrac cd\big)=A\big(\tfrac10,\tfrac01\big)\in\mathcal{E}_d$ passes through the unique tangent point of $\mathcal F(\tfrac ab)$ and $\mathcal F(\tfrac cd)$. It follows that the structure of the Farey graph $\mathcal{E}_d$ is equivalent to that of Ford spheres. This is illustrated in Figure~\ref{fig:sqrt2Ed}.  Moreover, one can show by direct calculation that the generalized Ford spheres for $\tfrac ab$ and $\tfrac cd$ of size $\beta$ touch each other if and only if $|bc-ad|=\beta$. In this case the geodesic connecting $\tfrac ab$ and $\tfrac cd$ passes through this touching point. We refer to \cite{N} for an application of generalized Ford spheres. 

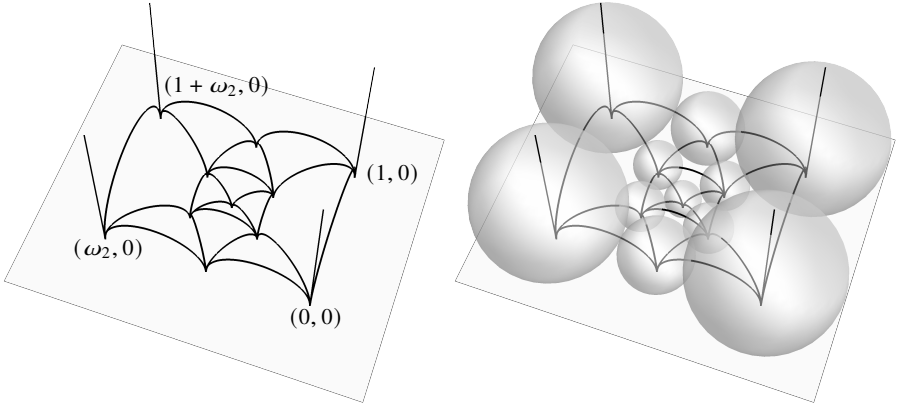
\begin{figure}[h]
\begin{tikzpicture}[scale=0.3]
\node[draw=none,fill=none] at (0,0){\includegraphics[trim=50 0 50 90,width=0.48\textwidth]{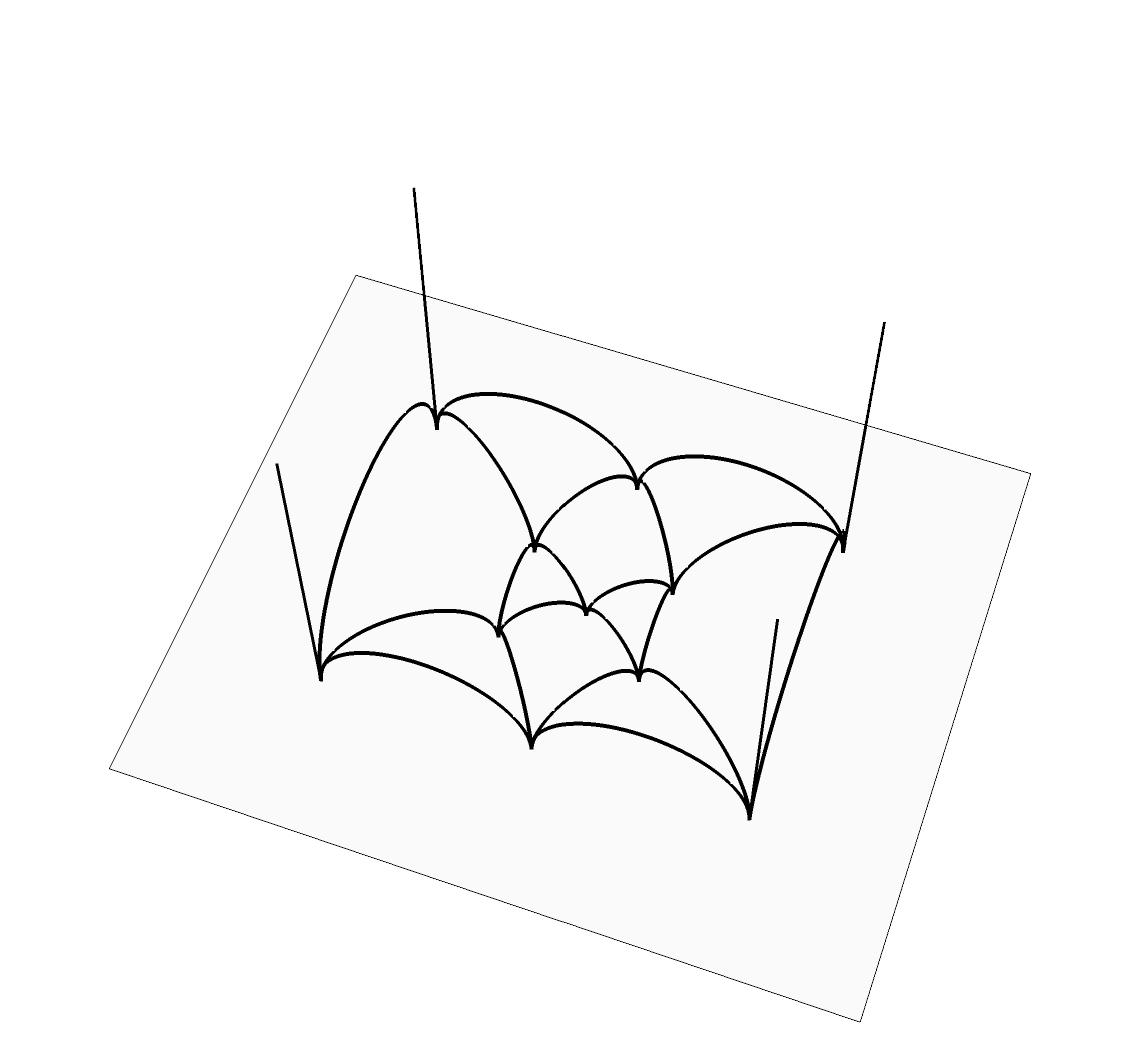}};
\node[draw=none,fill=none] at (20,0){\includegraphics[trim=50 0 50 90,width=0.48\textwidth]{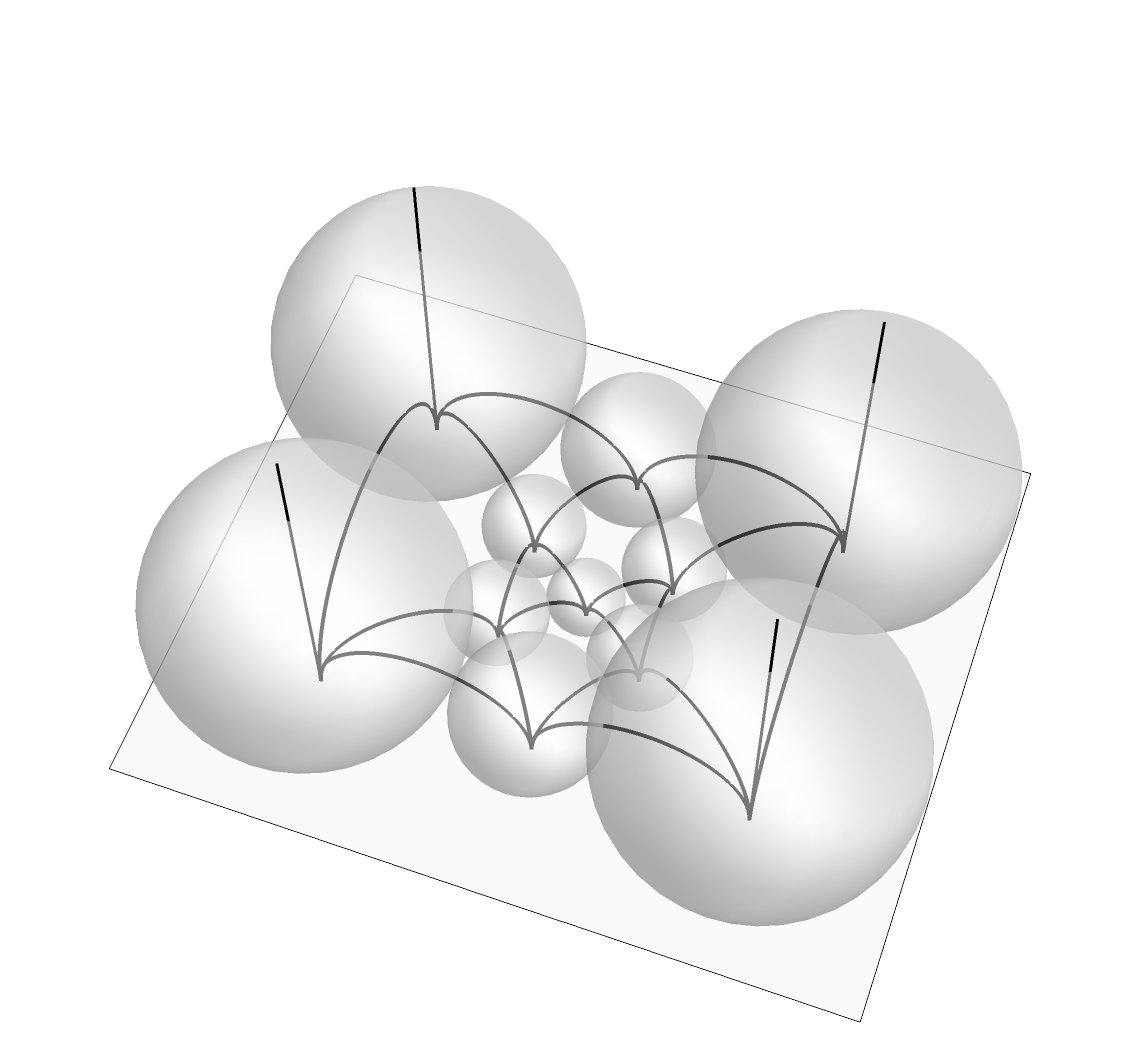}};
\node at (-5,-1.8) {$(\omega_2,0)$};
\node at (4.2,-4.9) {$(0,0)$};
\node at (7.6,1.5) {$(1,0)$};
\node at (-0.2,5.4) {$(1+\omega_2,0)$};
\end{tikzpicture}
\caption{A part of the Farey graph $\mathcal{E}_2$ with its edges drawn as geodesics in $\mathbb{H}^3$ (left) and its interplay with Ford spheres (right). \label{fig:sqrt2Ed}}
\end{figure}

In order to set up analogs in $\mathbb{H}^3$ for the Farey tessellation of $\mathbb{H}^2$ we need some more terminology.  Inspired by the work of Schmidt~\cite{SchmidtAsmus:69,SchmidtAsmus:78} we define certain hyperbolic polygons in $\mathbb{H}^3\cong\mathbb{C}\times\mathbb{R}_+$.

\begin{definition}[Farey triangle]\label{def:fareytriangle}
Let $d\in\{1,2,3,7,11\}$ and let $\triangle$ be the closed hyperbolic triangle in $\mathbb{H}^3$ with vertices $\frac01$, $\frac10=\infty$, and $\frac11$. For $A \in \mathrm{PSL}(2, \mathbb{Z}[\omega_d])$ set 
\[
\begin{pmatrix} p_1& p_2&p_3 \\ q_1&q_2&q_3\end{pmatrix}=A \cdot \begin{pmatrix} 0&1&1 \\ 1&0&1\end{pmatrix}.
\]
Then $\FT\big(\frac{p_1}{q_1},\frac{p_2}{q_2},\frac{p_3}{q_3}\big)=A\cdot \triangle$ is called a \emph{Farey triangle}.
\end{definition}

Since $\triangle$ is contained in a half-plane orthogonal to $\mathbb{C}\times\{0\}$, for $A \in \mathrm{PSL}(2, \mathbb{Z}[\omega_d])$ the Farey triangle $A\cdot \triangle=\FT\big(\frac{p_1}{q_1},\frac{p_2}{q_2},\frac{p_3}{q_3}\big)\subset\mathbb{H}^3$ is contained in a half-sphere (possibly degenerated to a half-plane) orthogonal to $\mathbb{C}\times\{0\}$. Its boundary consists of three geodesics. The Farey triangle $\FT\big(\frac{p_1}{q_1},\frac{p_2}{q_2},\frac{p_3}{q_3}\big)$ has the set of {\em cusps} $\big\{\frac{p_1}{q_1},\frac{p_2}{q_2},\frac{p_3}{q_3}\big\}$.

The notion of Farey triangle in Schmidt~\cite{SchmidtAsmus:69} coincides with the orthogonal projection of our Farey triangles to $\mathbb{C}\times\{0\}$ (this projection is a possibly degenerate Euclidean triangle). 

For some values of $d$ we need Farey quadrangles and Farey hexagons.

\begin{definition}[Farey quadrangle]
Let $d\in\{2, 7\}$ and let $\square$ be the closed hyperbolic quadrangle in $\mathbb{H}^3$ with vertices $\frac01$, $\frac10$, $\frac{\omega_d}1$, and $\frac1{\overline{\omega}_d}$. For $A \in \mathrm{PSL}(2, \mathbb{Z}[\omega_d])$  set 
\[
\begin{pmatrix} p_1& p_2&p_3&p_4 \\ q_1&q_2&q_3&q_4 \end{pmatrix}=A \cdot \begin{pmatrix} 0&1&\omega_d&1 \\ 1&0&1&\overline{\omega}_d\end{pmatrix}.
\]
Then $\FQ\big(\frac{p_1}{q_1},\frac{p_2}{q_2},\frac{p_3}{q_3},\frac{p_4}{q_4}\big)=A\cdot \square$ is called a \emph{Farey quadrangle}.
\end{definition}

Again, for $A \in \mathrm{PSL}(2, \mathbb{Z}[\omega_d])$ the quadrangle $\FQ(\frac{p_1}{q_1},\frac{p_2}{q_2},\frac{p_3}{q_3},\frac{p_4}{q_4})=A\cdot \square$ is contained in a half sphere which is orthogonal to the base plane. Its boundary consists of four geodesics which connect pairs of the cusps $\big\{\frac{p_1}{q_1},\frac{p_2}{q_2},\frac{p_3}{q_3},\frac{p_4}{q_4}\big\}$. 

\begin{definition}[Farey hexagon]
Let $d=11$ and let $\hexagon$ be the closed hyperbolic hexagon in $\mathbb{H}^3$ with vertices $\frac01$, $\frac10$, $\frac{\omega_d}1$,  $\frac2{\overline{\omega}_d}$, $\frac{\omega_d}2$,  and $\frac1{\overline{\omega}_d}$. For $A \in \mathrm{PSL}(2,\mathbb{Z}[\omega_d])$ set 
\[
\begin{pmatrix} p_1&p_2 & p_3&p_4&p_5&p_6 \\ q_1&q_2&q_3&q_4&q_5&q_6 \end{pmatrix}=A \cdot \begin{pmatrix} 0&1&\omega_d&2&\omega_d&1 \\ 1&0&1&\overline{\omega}_d&2&\overline{\omega}_d\end{pmatrix}.
\]
Then $\FH\big(\frac{p_1}{q_1},\frac{p_2}{q_2},\frac{p_3}{q_3},\frac{p_4}{q_4},\frac{p_5}{q_5},\frac{p_6}{q_6}\big)=A\cdot \hexagon$ is called a \emph{Farey hexagon}.
\end{definition}

To facilitate notation we will sometimes regard a geodesic as a {\em Farey digon} $\FD\big(\frac{p_1}{q_1},\frac{p_2}{q_2}\big)$. Moreover, a {\em Farey polygon} $\FP\big(\frac{p_1}{q_1},\ldots,\frac{p_n}{q_n}\big)$ can be a Farey digon, a Farey triangle, a Farey quadrangle, or a Farey hexagon (depending on $n\in\{2,3,4,6\}$, which may be unspecified).  Note that $\FP\big(\frac{p_1}{q_1},\ldots,\frac{p_n}{q_n}\big)$ has $n$ cusps. For $n\ge 3$ each Farey polygon is a subset of a unique possibly degenerate half-sphere whose equator lies in $\mathbb{C}\times\{0\}$. 

We are interested in the $\mathrm{PSL}(2,\mathbb{Z}[\omega_d])$-action on $\mathbb{Q}(\sqrt{-d})\cup\{\infty\}$. This action will induce a tessellation of the upper half space $\mathbb{H}^3\cong \mathbb{C}\times \mathbb{R}_+$ whose cells are bounded by Farey polygons. 

If $d\not\equiv 3 \pmod 4$ then $W_d\subset \mathbb{C}$ is the rectangle with corner points $0,1,\omega_d,\omega_d+1$ and if $d\equiv 3 \pmod 4$ then $W_d\subset \mathbb{C}$ is the triangle with corner points $0,1,\omega_d$. 

\begin{definition}[Wall and standard cell]\label{def:standard}
Let $Y\subset \mathbb{C}$ be an edge of the polygon $W_d$. A half-plane $X\subset \mathbb{H}^3 \cup \partial \mathbb{H}^3$ containing $Y \times \mathbb{R}_+$  is called a \emph{wall}.

The \emph{standard cell} $\mathcal{C}_d$ is the closure of the set of all points in $W_d\times \mathbb{R}_+$ that lie above the union of all non-degenerate spheres of the form $A \cdot X$, where $A\in \mathrm{PSL}(2,\mathbb{Z}[\omega_d])$ and $X$ is a wall.
\end{definition}

Later, in Section~\ref{sec:polyhedron}, we will see that $\mathcal{C}_d$ is a hyperbolic polyhedron with finitely many faces for each $d\in\{1,2,3,7,11\}$. Moreover, we will show that $\mathcal{C}_d$ will serve as tile in a Farey tessellation of $\mathbb{H}^3$, which forms an analog of the classical Farey tessellation of~$\mathbb{H}^2$.

Strongly related to Bianchi groups and Farey graphs are {\em complex continued fraction expansions}. We use the abbreviations 
\begin{equation*} 
[a_0; a_1,a_2\ldots]=
\displaystyle  a_0+\frac1{a_1+ 
 \displaystyle  \frac1{a_2 + 
 \displaystyle \frac1{\ddots }}}, 
 \;
 [a_0; a_1,\ldots,a_n]=\displaystyle  a_0+\frac1{a_1+ 
 \displaystyle \frac1{\ddots + 
 \displaystyle \frac1{a_n}}}, 
\end{equation*}
for continued fraction expansions and call $\frac{p_n}{q_n} = [a_0; a_1,\ldots,a_n]$ the {\em $n$-th convergent} of the continued fraction expansion $[a_0; a_1,a_2\ldots]$; in particular $\tfrac{p_{0}}{q_{0}} = \tfrac{a_{0}}{1}$. If $a_0=0$ we just write $[a_1,a_2,\ldots]=[0;a_1,a_2,\ldots]$ and $[a_1,\ldots,a_n]=[0;a_1,\ldots,a_n]$. 
It is well-known that $p_{n-1}q_{n}- p_{n}q_{n-1} = (-1)^{n}$ for $n \ge 1$.  Hence, $\tfrac{p_{n-1}}{q_{n-1}} \to  \tfrac{p_{n}}{q_{n}}$ is an edge of $\mathcal E_{d}$.  We can immediately generalize \cite[Theorem~3.1]{B-H-S} to the complex setting. In particular, in the following theorem we show that for  $\infty, R_{1}, R_{2},  \ldots, R_{n} \in \mathbb Q(\sqrt{-d})$ such that $R_{1}\to R_{2}\to  \cdots \to R_{n}$ is a walk in $\mathcal E_{d}$ we can construct a continued fraction expansion of $R_{n}$ satisfying $\frac{p_k}{q_k}=R_k$ for $k\in\{1,\ldots,n\}$.

\begin{theorem}\label{th:correspondencePathCF}
Fix $d\in\{1,2,3,7,11\}$ and let $x\in \mathbb{Q}(\sqrt{-d})$. Then $R_1,\ldots, R_n$ with $R_n=x$, are the consecutive convergents of some complex continued fraction expansion of $x$ if and only if $\infty \to  R_1 \to \cdots \to  R_n$ is a walk in $\mathcal{E}_{d}$ from 
$\infty$ to $x$.
\end{theorem}

For the sake of completeness we give the proof, which follows the classical case.

\begin{proof}
We first start with consecutive convergents $R_1,\ldots, R_n$, $R_n=x$, of a continued fraction expansion of $x \in \mathbb{Q}(\sqrt{-d})$. For convenience, set 
$S_{\alpha}= \begin{pmatrix}
\alpha &1\\
1&0
\end{pmatrix}$.
Then, by definition, 
$R_k=S_{\alpha_1}\cdots S_{\alpha_k}(\infty)$ holds for each $k\in\{1,\ldots,n\}$. For $k=1$ we  see that $\infty \to  R_1$ is an edge in $\mathcal{E}_{d}$, because $R_1=S_{\alpha_1}(\infty)=\alpha_1\in \mathbb{Z}[\omega_{d}]$. Now, we argue by induction. Since $\infty$ and $\alpha_{k+1}$ are Farey neighbors, the same is true for  
\[
R_k=S_{\alpha_1} \cdots S_{ \alpha_k}(\infty) \quad\hbox{and}\quad
 R_{k+1}=S_{\alpha_1} \cdots S_{\alpha_{k+1}}(\infty)=R_k=S_{\alpha_1} \cdots 
 S_{\alpha_k}(\alpha_{k+1}).
\]
Thus there is an edge from $R_k\to R_{k+1} \in \mathcal{E}_d$ and, hence, by induction $\infty \to  R_1 \to \cdots \to  R_{k+1}$ is a walk in $\mathcal{E}_{d}$.

For the converse direction we assume that $\infty \to  v_1 \to \cdots \to  v_n$ with $v_n=x$ is a walk in $\mathcal{E}_{d}$ from $\infty$ to $x$. We need to construct 
$\alpha_1,\ldots \alpha_n \in \mathbb{Z}[\omega_d]$ satisfying $v_k=S_{\alpha_1} \cdots 
S_{\alpha_k}(\infty) $. 
Then clearly $v_k$ is a convergent of  $[\alpha_1,\ldots, \alpha_n]=x$ for $k\in\{1,\ldots,n\}$. For $k=1$ we see that $v_1\in \mathbb{Z}[\omega_d]$ and set $\alpha_1=v_1$, hence, $v_1=S_{\alpha_1}(\infty)$. Assume now that we already constructed $\alpha_1,\ldots,\alpha_k\in \mathbb{Z}[\omega_d]$ in a way that 
$v_k=S_{\alpha_1} \cdots S_{\alpha_k}(\infty)$.  Then set $\alpha_{k+1}=S_{\alpha_k}^{-1}\cdots S_{\alpha_1}^{-1}(v_{k+1})$. This immediately yields $v_{k+1}=S_{\alpha_1} \cdots 
S_{\alpha_{k+1}}(\infty)$. Moreover, because $v_k =S_{\alpha_1} \cdots S_{\alpha_{k}}(\infty)
\sim S_{\alpha_1} \cdots S_{\alpha_{k}}(\alpha_{k+1})$ and $S_{\alpha_i}^{-1}$ preserves the Farey neighbor relation ``$\sim$'', we gain $\infty \sim \alpha_{k+1}$ and, hence, 
$\alpha_{k+1}\in \mathbb{Z}[\omega_d]$. The proof again follows by induction.
\end{proof}

Theorem~\ref{th:correspondencePathCF} allows to define  geodesic continued fraction expansions as follows.
 
\begin{definition}[Geodesic continued fraction expansion]
Let $d\in\{1,2,3,7,11\}$. The continued fraction expansion $z=[a_0;a_1,a_2,\ldots]$ with $a_0,a_1,\ldots\in\mathbb{Z}[\omega_d]$ of $z\in\mathbb{C}$ is said to be a {\em geodesic continued fraction expansion} if $\infty \to \tfrac{p_{0}}{q_{0}} \to \tfrac{p_{1}}{q_{1}} \to \cdots \to \tfrac{p_{n}}{q_{n}}$ is a geodesic path in $\mathcal E_{d}$ for any $n \ge 0$. 
\end{definition} 

Following Lakein~\cite{lakein1973approximation} we define the {\em nearest integer continued fraction algorithms} for $d\in\{1,2,3\}$ as follows.

\begin{definition}[Nearest integer continued fraction algorithm]\label{def:nicf}
Let
\[
\begin{split}
U_{1}  &= \left\{z= x + yi : -\tfrac{1}{2} \le x < \tfrac{1}{2}, \,\,  -\tfrac{1}{2} \le y < 
\tfrac{1}{2} \right\}, \\
U_{2}  &= \left\{z= x + yi : -\tfrac{1}{2} \le x < \tfrac{1}{2}, \,\,  -\tfrac{\sqrt{2}}{2} \le y < 
\tfrac{\sqrt{2}}{2} \right\}, 
\end{split}
\]
and let $U_3$ be the hexagon with set of vertices 
$\big\{\big(\pm \tfrac12, \pm \tfrac1{2\sqrt3}\big),\big(0,\pm\tfrac{1}{\sqrt{3}}\big)\big\}$ including its left and two bottom sides with two vertices (see Figure~\ref{fig:regionU3}). 
\begin{figure}[h]
\begin{center}
\includegraphics[width=0.2\textwidth]{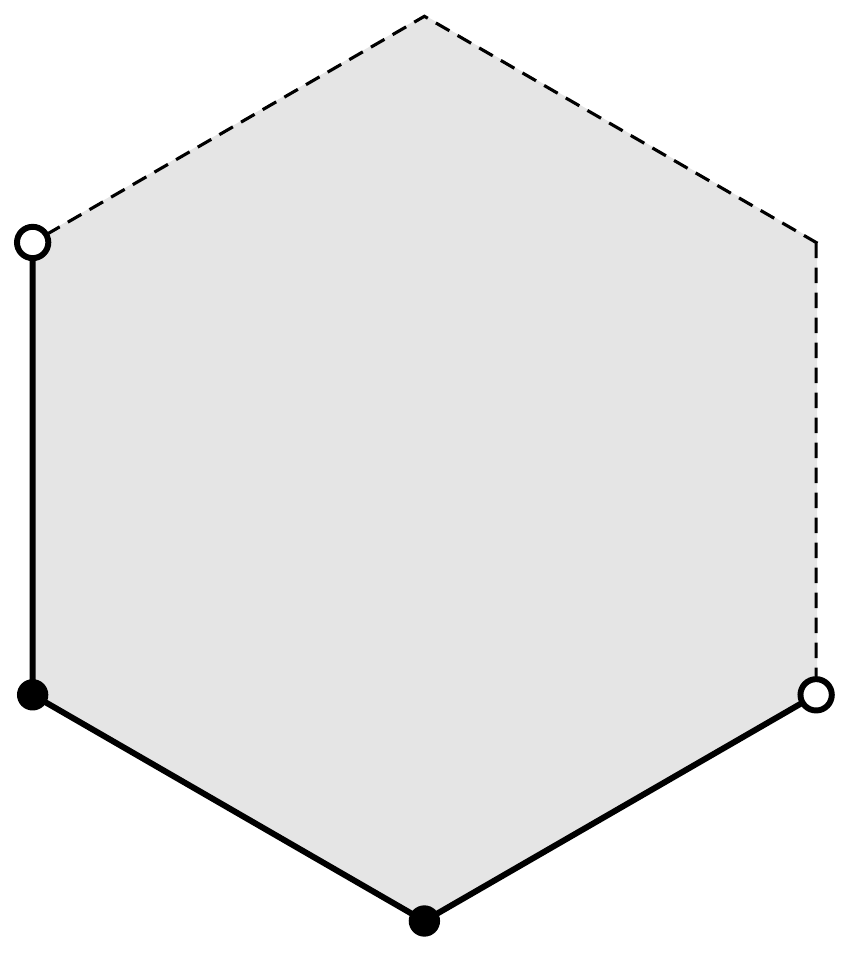} \qquad
\end{center}
\caption{The region $U_3$.
\label{fig:regionU3}}
\end{figure}
For $d\in \{1, 2, 3\}$, the {\em nearest integer continued fraction mapping} is given by 
\[
T_{d}(z) = \frac{1}{z} - a_{d}(z) \,\, \mbox{if} \,\, z \ne 0 \,\,\mbox{and}\,\, T_{d}(0) = 0,
\]
where $a_{d}(z) \in \mathbb{Z}[\omega_d]$ is chosen in a way that $\tfrac{1}{z} - a_{d}(z) \in U_{d}$ if $z\ne 0$ and $a_{d}(0) = \infty$. The the {\em nearest integer continued fraction expansion} of $z \in U_{d}$ is given by
$z = [a_{d,1}(z), a_{d,2}(z), \ldots, a_{d,n}(z)]$ with $a_{d, n}(z) = a_{d}(T_{d}^{n-1}(z))$. 
\end{definition}

One can define versions of the nearest integer continued fraction algorithm also for $d\in\{7,11\}$. However, we will only deal with the cases $d\in\{1,2,3\}$ in the sequel. In particular, we are able to prove that these algorithms always produce geodesic continued fraction expansions.

\section{Properties of the Farey graph}

In this section we provide some results on the Farey graphs $\mathcal E_{d}$ for the Euclidean cases $d\in\{1,2,3,7,11\}$. We start with the following connectivity result.

%%%%%%%%%%%%%%%
\begin{proposition}
If $d\in\{1,2,3,7,11\}$ then the Farey graph $\mathcal E_{d}$ is connected.  
\end{proposition}

\begin{proof}  Because $R_1\to R_2$ is an edge in $\mathcal{E}_{d}$ if and only if $R_2\to R_1$ is an edge in $\mathcal{E}_{d}$ it suffices to show that for any $R \in \mathbb Q(\sqrt{-d})$, there exists a walk in $\mathcal E_{d}$ leading from $\infty$ to~$R$. Let $\alpha\in \mathbb{Z}[\omega_d]$ be arbitrary. Because $\infty \to 0$ is an edge in $\mathcal{E}_{d}$ and for $A = \begin{pmatrix} 1 & \alpha \\ 0 & 1\end{pmatrix}$ we have $A(\infty) = \infty$ and  $A(0) = \alpha$, 
also $\infty \to \alpha$ is an edge in $\mathcal{E}_{d}$.  Assume that $\mathcal E_{d}$ is not connected. Then there is $\frac{\alpha_{0}}{\alpha_{1}} \in \mathbb Q(\sqrt{-d})\setminus \mathbb{Z}[\omega_d]$ with $\alpha_0,\alpha_1\in\mathbb{Z}[\omega_d]$ coprime such that there is no walk from $\infty$ to $\frac{\alpha_{0}}{\alpha_{1}}$ in $\mathcal{E}_d$. Assume that $|\alpha_1|>1$ is minimal with this property. Since $\alpha_0,\alpha_1$ are coprime, there are $r,s\in \mathbb{Z}[\omega_d]$ such that $\alpha_0s - \alpha_1 r=1$. By the Euclidean algorithm there exist $a,\alpha_2\in \mathbb{Z}[\omega_d]$ with $|\alpha_2| < |\alpha_1|$ such that $s=a\alpha_1+\alpha_2$. Then $\alpha_0 \alpha_2 - \alpha_1(r-a \alpha_0)=\alpha_0 (s-a\alpha_1) - \alpha_1(r-a \alpha_0)=\alpha_0s + \alpha_1 r=1$ and, hence,  by Proposition~\ref{prop:EdCrit}, $\frac{r-a\alpha_0}{\alpha_2} \to \frac{\alpha_{0}}{\alpha_{1}}$. But by the minimality of $|\alpha_1|$,  there is a walk from $\infty$ to $\frac{r-a\alpha_0}{\alpha_2}$ in $\mathcal{E}_d$ and, hence, a walk from $\infty$ to $\frac{\alpha_{0}}{\alpha_{1}}$, a contradiction.
\end{proof}

If $d \in\{1,3\}$, the restriction of  $\mathcal E_{d}$ to the set of vertices $\mathbb{Z}[\omega_d]$ forms a connected subgraph of 
$\mathcal E_{d}$, because $|\omega_1| = |\omega_3| = 1$. For $d \in\{ 2, 7,11\}$ this is not the case,  however, we have the following result (see Figure~\ref{fund3} for an illustration).
%%%%%%%%%%%%%%%%%%
%%%%%%%%%%%%%%%%%%
\begin{proposition}
Define
$G_2 = \mathbb{Z}[\omega_2] \cup \big(\mathbb{Z}[\omega_2] + \tfrac{\omega_2}{2}\big)$,
$G_7 = \mathbb{Z}[\omega_7] \cup \big(\mathbb{Z}[\omega_7] + 
\tfrac{\omega_7}{2}\big)$, and
$G_{11} = \mathbb{Z}[\omega_{11}] \cup \big(\mathbb{Z}[\omega_{11}] + 
\tfrac{\omega_{11}}{3} \big) \cup \big(\mathbb{Z}[\omega_{11}] + 
\tfrac{\omega_{11}}{2} \big) \cup \big(\mathbb{Z}[\omega_{11}] + 
\tfrac{2\omega_{11}}{3} \big)$.
Then, for $d\in\{2,7,11\}$, the restriction of  $\mathcal E_{d}$ to the set of vertices $G_d$ 
forms a connected subgraph of $\mathcal E_{d}$.
\end{proposition}
%%%%%%%%%%%%%%%%%
%%%%%%%%%%%%%%%%%

\begin{proof} We use Proposition~\ref{prop:EdCrit}.  Since 
$1\cdot1 - 0\cdot1=1$, $0\cdot \omega_2 - (-1)\cdot 1=1$, and $(-1)\cdot1 - \omega_2 \cdot \omega_2=1$ we conclude that $0\to 1$ and
$0 \to \tfrac{\omega_2}{2}=\tfrac{-1}{\omega_2} \to \omega_2$ are walks in $\mathcal E_{2}$.  Therefore, $\alpha \to \alpha+1$ and $ \alpha \to \alpha + \tfrac{\omega_2}{2} \to \alpha+\omega_2$  are walks in $\mathcal E_{2}$ for any $\alpha \in \mathbb{Z}[\omega_2]$. This shows  
assertion for $d=2$.  The case $d=7$ follows by the same reasoning because $\tfrac01\to\tfrac{1}{\overline \omega_7}\to\tfrac{\omega_7}{1}$ is a walk in $\mathcal{E}_7$.   
Finally, the case $d=11$ follows because
$0 \to \frac{1}{\overline \omega_{11}}\to\frac{\omega_{11}}{2} \to  
\frac{2}{\overline{\omega}_{11}} \to \omega_{11}$ is a walk in $\mathcal E_{11}$. 
\end{proof}
%%%%%%%%%%%%%%%%%%%%%%%%%%%
%%%%%%%%%%%%%%%%%%%%%%%%%%%

We now show that an edge of $\mathcal{E}_d$ cannot cross a wall (see Definition~\ref{def:standard}).  

\begin{theorem}\label{GeodesicIntersectW}
Fix $d\in\{1,2,3,7,11\}$. Let $p_1,q_1,p_2,q_2\in\mathbb{Z}[\omega_d]$ with $|p_1q_2-p_2q_1|< \frac{2\Im(\omega_d)}{|\omega_d|}$ and a wall $X$ be given. If $X_1,X_2$  denote the two complementary components of $\mathbb{H}^3\setminus X$ then there is $i\in\{1,2\}$ such that the geodesic $\big(\frac{p_1}{q_1}, \frac {p_2}{q_2}\big)$ is contained in ${X_i}^{\mathrm{cl}}$. Thus $\big(\frac{p_1}{q_1}, \frac {p_2}{q_2}\big)$ cannot cross~$X$. 
\end{theorem}

\begin{proof}
Let $\delta=|p_1q_2-p_2q_1| \in\mathbb{N}$. Note that $\delta<\frac{2\Im(\omega_d)}{|\omega_d|}$ by assumption. Let $i\in\{1,2\}$. Assume that $\frac{p_i}{q_i} \in X_i$ for $i\in\{1,2\}$.   
We first show that the Ford sphere $F_i$ of size $\delta$ for $\frac {p_i}{q_i}$ satisfies $F_i\cap X=\emptyset$. 

If  $X=(\mathbb{R}+0\cdot i)\times\mathbb{R}_+$ then the Euclidean distance between $\frac{p_i}{q_i} $ and $X$ is equal to $\big|\Im(\frac{p_i}{q_i})\big|$, which is nonzero by assumption. However, 
\[
\bigg|\Im\Big(\frac {p_i}{q_i} \Big)\bigg| =  \frac{|\Im(p_i \overline{q}_i)| }{|q_i^2|} \ge \frac{\Im(\omega_d)}{|q_i|^2}=\frac{2 \Im(\omega_d)}{2|q_i|^2}   \ge \frac{2\Im(\omega_d)}{|\omega_d|}\frac{1}{2|q_i|^2},
\]
and since the radius of $F_i$ is $\frac{\delta}{2|q_i|^2}$ with $\delta<\frac{2\Im(\omega_d)}{|\omega_d|}$ this implies that $F_i\cap X=\emptyset$.

If $X=\omega_d (\mathbb{R}+0\cdot i)\times\mathbb{R}_+$ then the Euclidean distance between $\frac{p_i}{q_i}$ and $X$ is equal to $|\omega_d| \big|\Im(\omega_d^{-1} \frac{p_i}{q_i})\big|$. In this case
\[
|\omega_d|\cdot\Big|\Im\Big(\omega_d^{-1} \frac{p_i}{q_i}\Big) \Big|
= \frac1{|\omega_d q_i^2|} |\Im(\overline{\omega}_d p_i \overline{q}_i)|
\ge 
\frac1{|\omega_d q_i^2|} \Im(\omega_d)
=
\frac{2\Im(\omega_d)}{|\omega_d|}\frac{1}{2|q_i|^2}
\]
and since the radius of $F_i$ is $\frac \delta{2|q_i|^2}$ with $\delta< \frac{2\Im(\omega_d)}{|\omega_d|}$ this implies that $F_i\cap X=\emptyset$.

For $d\in \{1,2\}$ the remaining two walls are just translations of the two walls we treated above. For $d\equiv 3\pmod{4}$ it remains to deal with the wall $L'\times\mathbb{R}_+$, where $L'$ is the line in $\mathbb{C}$ passing through $1$ and $\omega_d$. This is done in the same way as the wall $\omega_d(\mathbb{R}+0\cdot i) \times\mathbb{R}_+$. Summing up this yields that $F_i \cap X = \emptyset$ for each wall $X$ and $i\in\{1,2\}$

Thus $F_1 \cap F_2=\emptyset$. Since a geodesic from $\frac pq$ to $\frac rs$ with $|ps-rq|=\delta$ has to pass through a point in $F_1 \cap F_2$, the theorem is proved.
\end{proof}

The following corollary is an immediate consequence of this theorem.

\begin{corollary}
Suppose that $R_{1} \to R_{2}$ is an edge of $\mathcal E_{d}$ and $R_{1} \in 
(W_{d} \times \mathbb{R}_+)^{\circ}$, then $R_{2} \in (W_{d} \times \mathbb{R}_+)^{\mathrm{cl}}$.  
\end{corollary}

In the introduction we mentioned that the Farey graph is planar because the geodesics defining its edges do not intersect.  The following corollary provides the disjointness of the geodesic edges of $\mathcal{E}_d$ also for the complex case.

\begin{corollary}\label{cor:geocor}
Let $\big(\frac {p_1}{q_1}, \frac {p_2}{q_2}\big)$ and $\big(\frac {r_1}{s_1}, \frac {r_2}{s_2}\big)$ be two geodesics with $|p_1q_2-p_2q_1|< \frac{2\Im(\omega_d)}{|\omega_d|}$, and $|r_1s_2-r_2s_1| < \frac{2\Im(\omega_d)}{|\omega_d|}$. Then they are either disjoint with zero or one common cusp, or they are equal.
\end{corollary}

\begin{proof}
We may assume without loss of generality that $\frac{p_1}{q_1}=\frac01$ and $\frac{p_2}{q_2} = \frac 10$. Then the result follows immediately from Theorem~\ref{GeodesicIntersectW}.
\end{proof}

In the complex case, not only the geodesics are disjoint but even the Farey polygons (whose boundary lines are edges of $\mathcal{E}_d$) are disjoint. This is the content of the next result that will also be important when we deal with a generalization of the Farey tessellation in Section~\ref{sec:polyhedron}.

\begin{proposition}\label{lem:GeodesicIntersectGeneral}
Let $d\in\{1,2,3,7,11\}$ and let $\Pi_1$ and $\Pi_2$ be distinct Farey polygons. Then their intersection $\Pi_1\cap \Pi_2$ is empty with zero or one common cusp, or $\Pi_1\cap \Pi_2$ is a common edge of $\Pi_1$ and $\Pi_2$ (which is a geodesic connecting two cusps).
\end{proposition}

\begin{proof}
Without loss of generality we may assume that $\Pi_1$ is contained in a wall $X$ of $W_d\times\mathbb{R}$. Moreover, $\Pi_2$ is an $n$-gon with $n\in\{2,3,4,6\}$. It is clear that the number $k$ of cusps of $\Pi_2$  contained in $X$ satisfies $k\in\{0,1,2,n\}$. We subdivide the proof according to the value of~$k$.

Case 1: $k=0$. We claim that the cusps of $\Pi_2$ all lie on the same side of $X$. Suppose that this is wrong. Then there exist two cusps $\frac pq$ and $\frac rs$ of $\Pi_2$ lying on different sides of $X$. By the definition of a Farey polygon we may assume that the $\frac pq$ and $\frac rs$ are chosen in a way that $|ps - qr|=1$. But this contradicts Theorem~\ref{GeodesicIntersectW} (since $1<\frac{2\Im(\omega_d)}{|\omega_d|}$) and the claim follows. Since the claim implies that $\Pi_1 \cap \Pi_2 \subset X \cap \Pi_2  = \emptyset$ with no common cusp, this case is finished.

Case 2: $k=1$. Let $v$ be the cusp of $\Pi_2$ contained in $X$. As in case~1 we see that all the other cusps of $\Pi_2$ lie on the same side of $X$. Thus $\Pi_1 \cap \Pi_2 \subset X \cap \Pi_2  \subset \{v\}$. This implies that $\Pi_1 \cap \Pi_2=\emptyset$ and $\Pi_1$ and $\Pi_2$ have zero or one common cusp.

Case 3: $k=2$. Let $\frac pq$ and $\frac rs$ be the two cusps of $\Pi_2$ contained in $X$. We claim that $|ps-qr|=1$. Assume that this is wrong. Then $\Pi_2$ cannot be a Farey digon or triangle. If $\Pi_2$ is a Farey quadrangle then $d\in\{2,7\}$ and $|ps-qr|=|\omega_{d}|$ (see Figure~\ref{fig:fareyhex}). In this case the diagonal of the Farey quadrangle is contained in $X$. Thus the other two cusps $\frac{p'}{q'}$ and $\frac{r'}{s'}$ of $\Pi_2$ lie on opposite sides of $X$. As in this case we have $|p's'-q'r'|=|\omega_{d}|$, Theorem~\ref{GeodesicIntersectW} (for $d\in\{2,7\}$ we have $|\omega_{d}|<\frac{2\Im(\omega_{d})}{|\omega_{d}|}$) gives a contradiction. Thus the claim is true also for Farey quadrangles. If $\Pi_2$ is a Farey hexagon then $d=11$ and $|ps-qr|\in \{\sqrt{3},2\}$. According to Figure~\ref{fig:fareyhex} there exist two other cusps $\frac{p'}{q'}$ and $\frac{r'}{s'}$ of $\Pi_2$ that lie on opposite sides of $X$. Moreover, we can choose these cusps in a way that $|ps-qr|=\sqrt{3}$. Since this contradicts Theorem~\ref{GeodesicIntersectW} (for $d=11$ we have $\sqrt{3}<\frac{2\Im(\omega_{11})}{|\omega_{11}|}$) the claim also holds for Farey hexagons.
\begin{figure}
\begin{tikzpicture}[scale=0.3]

\node at (-8.5,8.5) {$\frac{1}{\overline{\omega}_d}$};
\node at (8.5,8.5) {$\frac{\omega_d}{1}$};
\node at (-8.5,-8.5) {$\frac{0}{1}$};
\node at (8.5,-8.5) {$\frac{1}{0}$};

\node at (4.7,3) {$|\omega_d|$};
\node at (-4.7,3) {$|\omega_d|$};
\node at (-4.7,-3) {$|\omega_d|$};
\node at (4.7,-3) {$|\omega_d|$};

\node at (15,8.5) {$\frac{\omega_d}{2}$};
\node at (25,8.5) {$\frac{1}{\overline{\omega}_d}$};
\node at (10,0) {$\frac{2}{\overline{\omega}_d}$};
\node at (29.5,0) {$\frac{0}{1}$};
\node at (15,-8.5) {$\frac{\omega_d}{1}$};
\node at (25,-8.5) {$\frac{1}{0}$};

\node at (20,0.7) {$2$};
\node at (24,4.2) {$|\omega_d|$};
\node at (16,4.2) {$|\omega_d|$};
\node at (16,-4.2) {$|\omega_d|$};
\node at (24,-4.2) {$|\omega_d|$};

\node[draw=none,fill=none] at (0,0){\includegraphics[width=.4\textwidth]{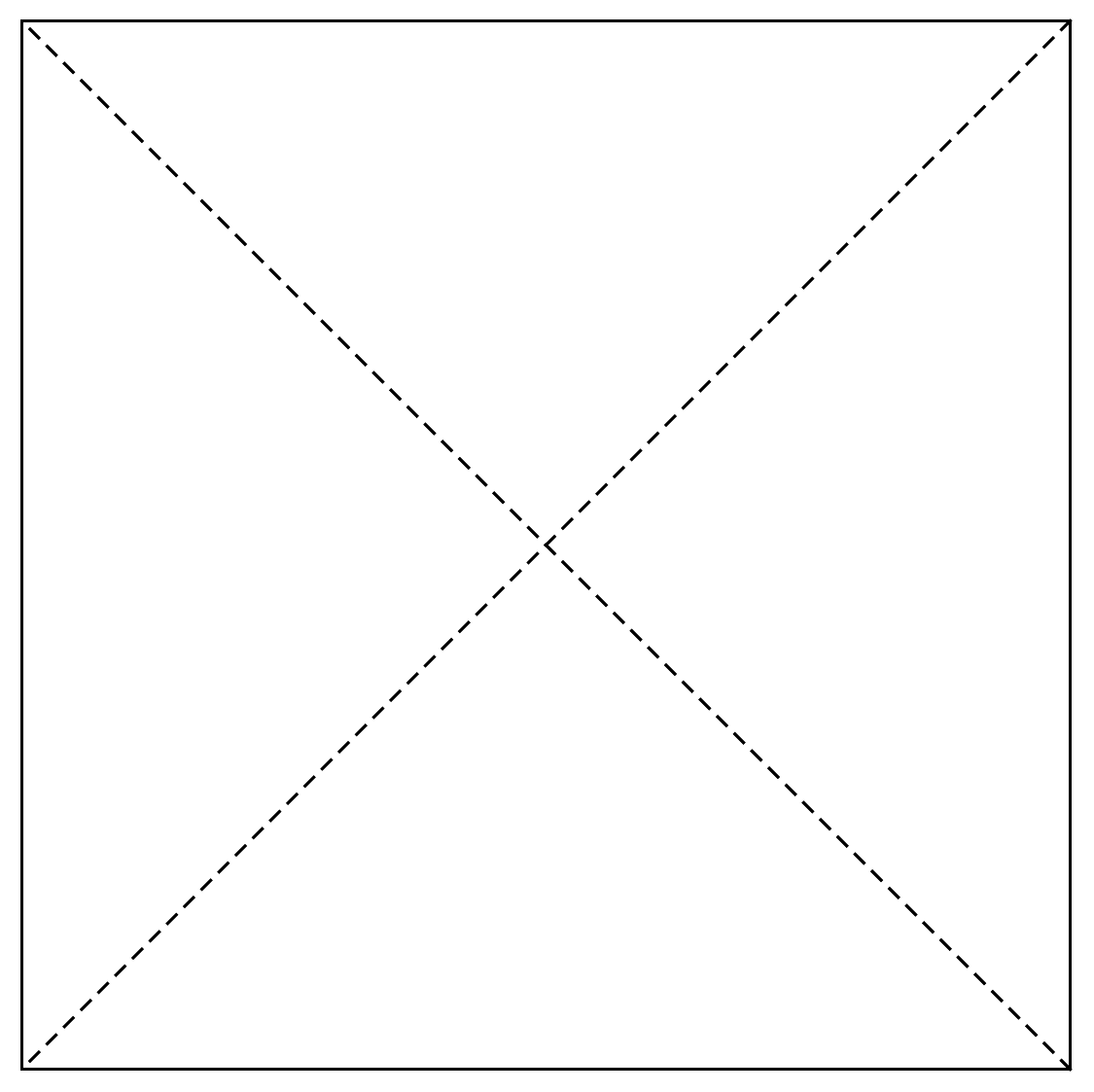}};
\node[draw=none,fill=none] at (20,0){\includegraphics[width=.46\textwidth]{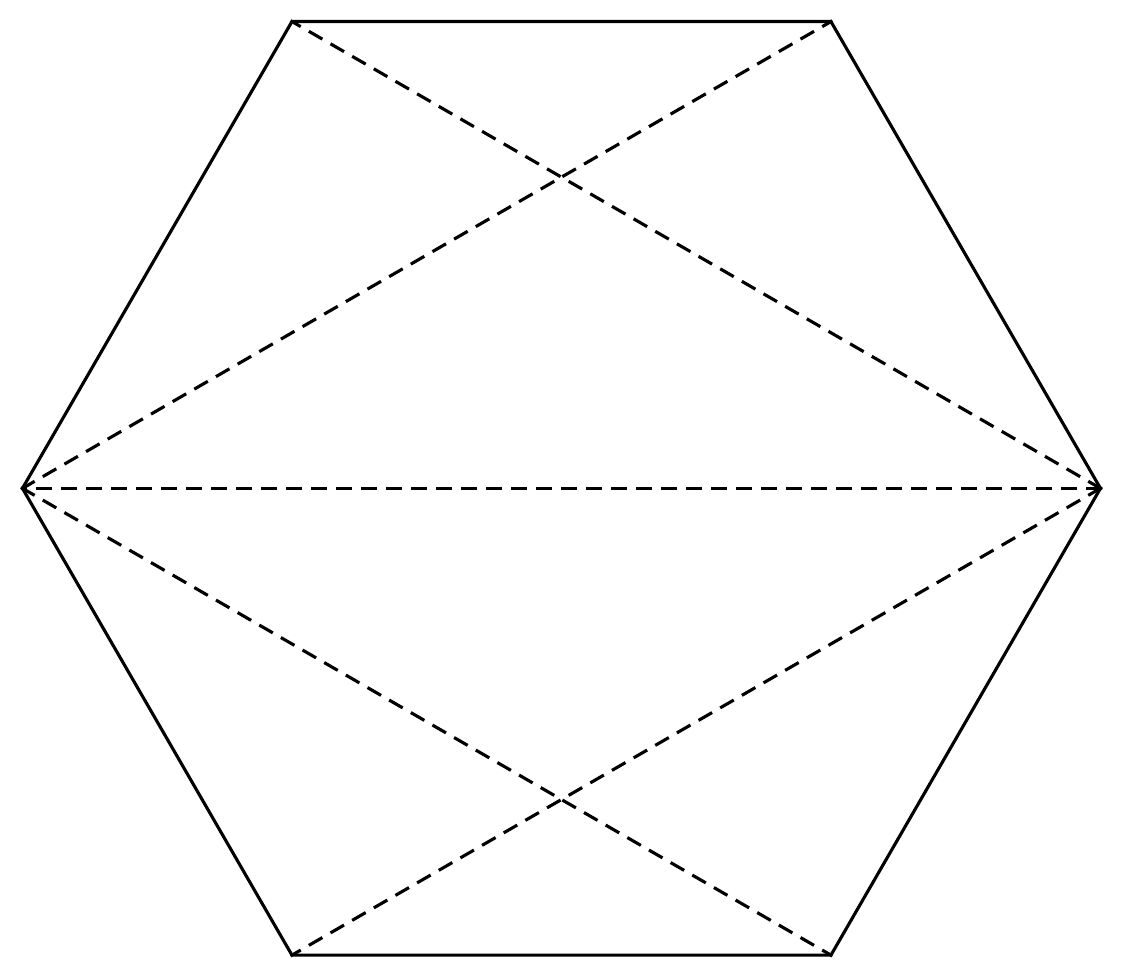}};
\end{tikzpicture}
\caption{The Farey quadrangle $\square$ and the Farey hexagon $\hexagon$. The numbers on the dashed lines between cusps $\frac pq$ and $\frac rs$ are the values of $|ps-qr|$ (for the hexagon, not all connecting lines are drawn, however, the corresponding values of $|ps-qr|$ are symmetric as can be confirmed by a short direct calculation). \label{fig:fareyhex}}
\end{figure}
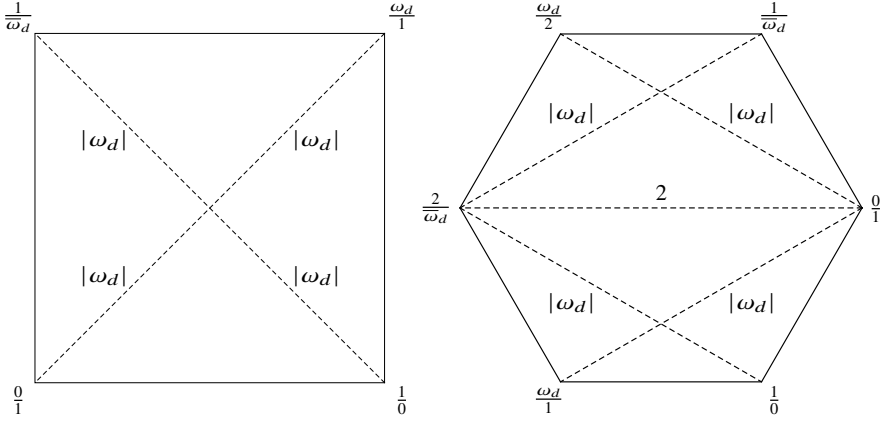

The claim now immediately yields that $\Pi_1 \cap \Pi_2 \subset X\cap \Pi_2 = \big(\frac pq,\frac rs\big)$. If $\Pi_1 \cap \Pi_2$ is equal to this geodesic we are done. Thus we may assume that at least one of the cusps $\frac pq,\frac rs$ is not a cusp of $\Pi_1$. Without loss of generality assume that this is $\frac pq$. Suppose that $\Pi_1 \cap \big(\frac pq,\frac rs\big) \not=\emptyset$. Then only a part of $\big(\frac pq,\frac rs\big)$ is contained in $\Pi_1$. Let $x$ be the first point of $\big(\frac pq,\frac rs\big)$ (starting from $\frac pq$) that is contained in $\Pi_1$. Then $x$ is contained in some geodesic $\big(\frac{p'}{q'},\frac{r'}{s'}\big)$ bounding $\Pi_1$. Thus $x\in \big(\frac pq,\frac rs\big)\cap \big(\frac{p'}{q'},\frac{r'}{s'}\big)$, but Corollary~\ref{cor:geocor} implies that such an $x$ cannot exist. This shows that $\Pi_1 \cap \Pi_2=\emptyset$ where $\Pi_1$ and $\Pi_2$ have the common cusp $\frac rs$.

Case 4: $k=n$ ($n\ge 3$). Let $\big(\frac pq, \frac rs\big)\in\mathcal{E}_d$ be an edge of $\Pi_2$. As in case~3 we can prove that $\big(\frac pq, \frac rs\big)$  is either a subset of $\Pi_1$ or disjoint from $\Pi_1$. If $\big(\frac pq, \frac rs\big)$ is a subset of $\Pi_1$ its cusps have to be cusps of $\Pi_1$. Because $\big(\frac pq, \frac rs\big)$ is an edge of $\Pi_2$, we have $|ps-qr|=1$. Thus $\big(\frac pq, \frac rs\big)$ therefore has to coincide with an edge of $\Pi_1$ (see Figure~\ref{fig:fareyhex}). Thus $\Pi_1\cap\Pi_2$ consists of a union of edges. If $\Pi_1\cap\Pi_2$ has more than two cusps, it contains all edges connecting these cusps. Moreover, if it contains more than two edges then $\Pi_1= \Pi_2$ ($\Pi_1 \subsetneq \Pi_2$ and $\Pi_1 \supsetneq \Pi_2$ are impossible, see Figure~\ref{fig:fareyhex}) which is excluded. Thus, $\Pi_1\cap\Pi_2$ is as specified in the statement.
\end{proof}

\section{The subgraph of $\mathcal{E}_d$ on the walls and Hecke groups} \label{sec:hecke}

In the present section we consider subgraphs of the Farey graph $\mathcal{E}_d$, $d\in\{1,2,3,7,11\}$, whose edges lie in the walls of $W_d\times\mathbb{R}_+$. These graphs are of interest in their own right because they are related to certain Hecke groups. Moreover, we will need these graphs in order to establish our results on Farey tessellations of $\mathbb{H}^3$.

Obviously, the subgraph of $\mathcal E_{d}$ whose edges lie in the wall  $(\mathbb{R}+ 0\cdot i)\times \mathbb{R}_+$ is isomorphic to the classical Farey graph depicted in Figure~\ref{fig:ClassicalFarey}. For $d\in \{1,3\}$ the same is true for all the other walls of $W_d\times\mathbb{R}_+$, and for $d=2$ it is also true for the wall $(\mathbb{R}+ \omega_2)\times \mathbb{R}_+$.
Thus we only have to consider the remaining walls of $W_d\times\mathbb{R}_+$ for $d \in\{2, 7, 11\}$. It turns out that the the subgraphs of $\mathcal E_{d}$ lying in one of these walls are related to the Hecke groups of indexes $4$ and $6$ defined by 
\[ 
H_{4} =\left\langle \begin{pmatrix} 0 & 1 \\ -1 & 0 \end{pmatrix}, \, 
\begin{pmatrix} 1 & \sqrt{2} \\ 0 & 1  \end{pmatrix} \right\rangle \qquad\text{and}\qquad
H_{6} = \left\langle \begin{pmatrix} 0 & 1 \\ -1 & 0 \end{pmatrix}, \, 
\begin{pmatrix} 1 & \sqrt{3} \\ 0 & 1  \end{pmatrix} \right\rangle,
\]
respectively. This relation  will be established via the following real continued fraction algorithms.
Write $\mathbb I_{4} = [0, \sqrt{2})$ and $\mathbb I_{6} = [0, \sqrt{3})$ and define
\[
\lceil x \rceil_{4} = k \sqrt{2} \,\, \mbox{if} \,\,x \in \big( (k-1) \sqrt{2}, 
k\sqrt{2} \big] , \quad\text{and}\quad  \lceil x \rceil_{6} = k \sqrt{3} \,\, \mbox{if} \,\, x \in \big( (k-1) \sqrt{3}, k\sqrt{3} \big]
\]
for $x \in \mathbb R$.  For $\ell\in\{4,6\}$  we define the continued fraction map $F_{\ell}:\mathbb I_{\ell}\to\mathbb I_{\ell}$ by 
\[
F_{\ell}(x) = \left\{\begin{array}{lll}
\big\lceil \tfrac{1}{x} \big\rceil_{\ell} - \tfrac{1}{x} & \mbox{if} & x \ne 0, \\
0 & \mbox{if} & x = 0. \end{array}\right.
\]
As usual, the partial quotients of $x\in \mathbb I_{\ell}$ are given by $b_{\ell, n}(x) = \big\lceil \tfrac{1}{F_{\ell}^{n-1}(x)}\big\rceil$ provided that $F_{\ell}^{n-1}(x) \not\in\{0,\infty\}$, and we have 
\[
x=
\displaystyle \frac1{b_{\ell, 1}(x)- 
 \displaystyle  \frac1{b_{\ell, 2}(x) - 
 \displaystyle \frac1{\ddots}}}.
\]
The convergents of the algorithm $F_\ell$ are defined by 
\begin{equation}\label{eq:Fconv}
\begin{pmatrix} -p_{\ell, n-1}(x) & p_{\ell, n}(x) \\  -q_{\ell, n-1}(x) & q_{\ell, n}(x) 
\end{pmatrix} 
=
\begin{pmatrix} 0 & 1 \\ -1 & b_{\ell, 1}(x) \end{pmatrix} 
\begin{pmatrix} 0 & 1 \\ -1 & b_{\ell, 2}(x) \end{pmatrix}  \,\cdots \,
\begin{pmatrix} 0 & 1 \\ -1 & b_{\ell, n}(x) \end{pmatrix} \quad (n\ge1), 
\end{equation}
with $p_{\ell, 0}(x) = 0$ and $q_{\ell, 0}(x) = 1$. 
One can easily check by induction that for $k \ge 0$ we have $p_{\ell, 2k + 1}(x)$, $q_{\ell, 2k}(x) 
\in \mathbb N$ and $p_{\ell, 2k}(x)$, $q_{\ell, 2k+1}(x)\in  \sqrt{2}\cdot \mathbb N$.   
%%%%%%%%%%%%%%%%%%%%%%%
%%%%%%%%%%%%%%%%%%%%%%%
\begin{remark} We can define $F_{\ell}$ for any $\ell \ge 3$ in the same way.  In particular, the map $F_3$ is the so-called backward continued fraction map defined on $[0, 1)$. However, in this paper we only need the cases $\ell =4$ and $6$.  The map $F_{\ell}$ and the induced continued fraction is essentially the same as the one studied in Gr\"ochenig and Haas~\cite{G-H} for each $\ell \ge 3$.   
\end{remark}
%%%%%%%%%%%%%%%%%%%%%%%%%
%%%%%%%%%%%%%%%%%%%%%%%%%
Let $\ell\in\{4,6\}$. We can define the Farey graph associated with the Hecke group $H_{\ell}$ as the graph whose edges are given by $A(\infty)\to A(0)$ with $A \in H_{\ell}$. As in the classical case, the edge $A(\infty) \to  A(0)$ can be interpreted as the geodesic $(A(\infty), A(0))$ in $\mathbb{H}^2$. Moreover, we can associate  the Ford circles $\mathcal{F}(A(\infty)) = A(\{x + i\,:\, x\in\mathbb{R} \})$, $A\in H_\ell$, with $H_{\ell}$. Then $(x_{1}, x_{2})$ is an edge of the Farey graph associated with $H_{\ell}$ if and only if $\mathcal{F}(x_{1})$ and $\mathcal{F}(x_{2})$ are tangent to each other.  Moreover, the geodesic path from $x_{1}$ to $x_{2}$ passes through this tangent point. This is illustrated for $\ell=4$ in Figures~\ref{fig:Farey7wall} and~\ref{fig:Farey7wall2}.  
\begin{figure}[h]
\includegraphics[width=0.9\textwidth]{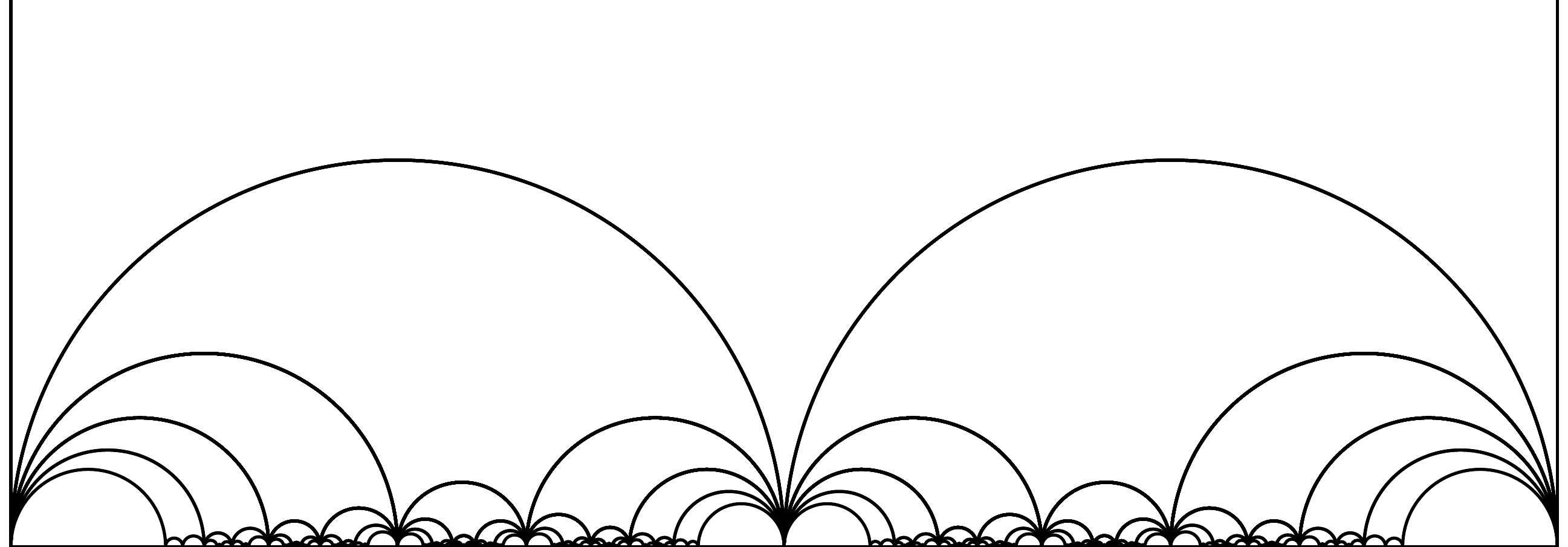} \\
\caption{The Farey graph associated with the Hecke group $H_4$. \label{fig:Farey7wall}}
\end{figure}
\begin{figure}[h]
\includegraphics[width=0.9\textwidth]{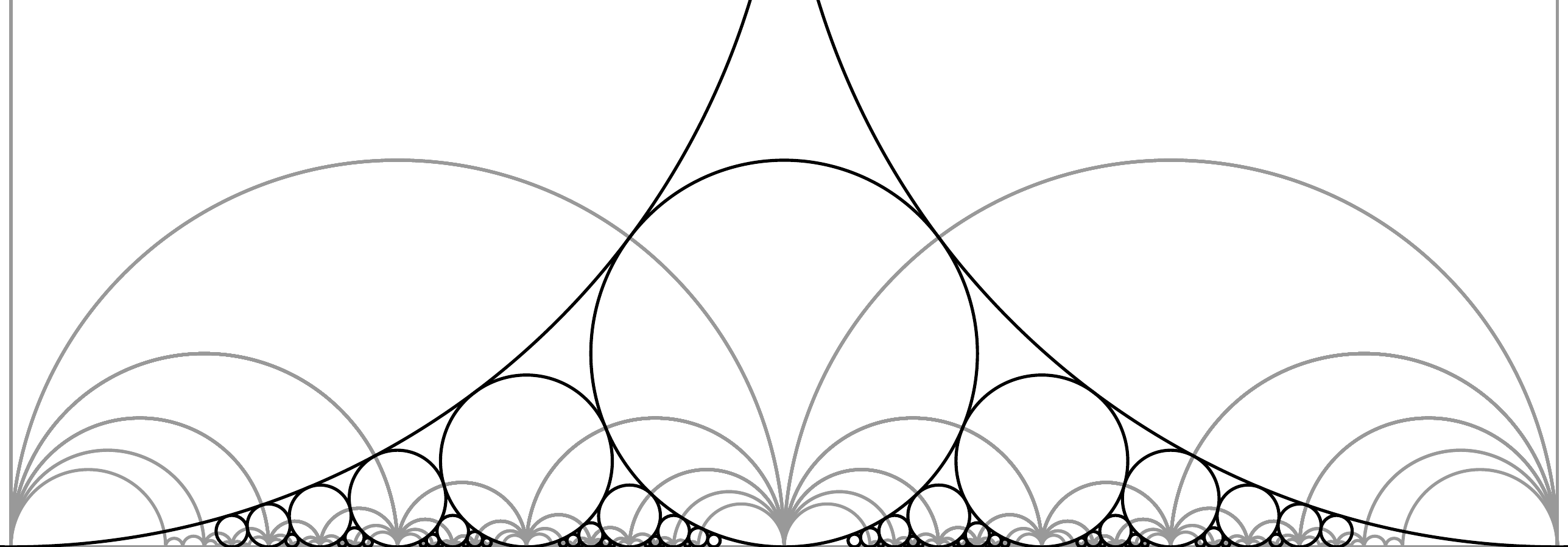} 
\caption{Interplay between the Farey graph of Figure~\ref{fig:Farey7wall} and Ford circles. 
\label{fig:Farey7wall2}}
\end{figure}

The algorithm $F_\ell$ is related to the Farey graph associated with the Hecke group $H_{\ell}$ in the following way.  By the Euclidean algorithm, we see that any element in $H_{\ell}(\infty) \cap \mathbb I_{\ell}$ can be expressed by a finite continued fraction expansion given by the algorithm $F_{\ell}$. Even more, from \eqref{eq:Fconv} one can derive that each edge of the Farey graph associated with the Hecke group $H_{\ell}$ corresponds to a geodesic of the form $\big( \frac{p_{\ell, n-1}(x)}{q_{\ell, n-1}(x)},  \frac{p_{\ell, n}(x)}{q_{\ell, n}(x)}\big)$, where $x\in H_{\ell}(\infty) \cap \mathbb I_{\ell}$. Thus this Farey graph is completely determined by the convergents of~$F_\ell$.

The following theorem shows that the restriction of the Farey graph $\mathcal{E}_d$ to certain walls is the Farey graph associated with the Hecke group $H_4$ for $d\in \{2,7\}$ and the Farey graph associated with the Hecke group $H_6$ for $d=11$.

%%%%%%%%%%%%%%%%% 
%%%%%%%%%%%%%%%%%
\begin{theorem} 
\label{th:Hecke}
\mbox{}
\begin{itemize}
\item[(i)]\; The subgraphs of $\mathcal E_{2}$ restricted to vertices on $\{y\omega_{2} : 0 \le y <1\}$ and restricted to the vertices on $\{1+y\omega_{2} : 0 \le y <1\}$ each are isomorphic to the Farey graph defined on~$\mathbb I_{4}$ associated with $H_{4}$.  

\item[(ii)]\;The subgraphs of $\mathcal E_{7}$ restricted to vertices on $\{y\omega_{7} : 0 \le y <1\}$ and restricted to the vertices on $\{1+y(\omega_{7}-1) : 0 \le y <1\}$ each are isomorphic to the Farey graph defined on $\mathbb I_{4}$ associated with $H_{4}$.  

\item[(iii)]\; The subgraphs of $\mathcal E_{11}$ restricted to vertices on $\{y \omega_{11} : 0 \le y <1\}$ and restricted to the vertices on $\{1+y(\omega_{11}-1) : 0 \le y <1\}$ each are isomorphic to the  Farey graph defined on $\mathbb I_{6}$ associated with $H_{6}$.

\end{itemize}
\end{theorem}
%%%%%%%%%%%%%%
%%%%%%%%%%%%%%
\begin{remark}
In the case of $d=1$, the Ford spheres on the diagonal set $\{x+ xi : 0 \le x < 1\}$ 
has also the same structure as Ford circles associated with $ H_{4}$.  
For $d = 2$, the diagonal set $\{ x + x \omega_2 : 0 \le x < 1 \}$ has 
the same structure as $ H_{6}$ the same sense.  \\
In the case of $d = 7$, the diagonal set $\{ x + x(\omega_{7} + 1) 
: 0 \le x < 1 \}$ has the same structure with the subgraph of the original Farey graph which is induced from the even continued fractions.
\end{remark}    
%%%%%%%%%%%%%%
%%%%%%%%%%%%%%

In the proof of the theorem, we restrict ourselves to the case $\{y\omega_{d} : 0 \le y <1\}$, the other cases can be treated analogously. For $d\in\{2,7,11\}$, we consider the restricted complex continued fraction expansions defined on 
\[
\mathbb J_{d} = \{ z= y \omega_{d} \text{ or } z= y \overline{\omega}_{d} : 0 \le y < 1\}.
\]
They are defined by 
\[
S_{d}(z) = \left\{ \begin{array}{lll}
\left\lceil\tfrac{1}{z}\right\rceil_{d} - \tfrac{1}{z} & \mbox{for} & z \in \mathbb J_{d} \setminus 
\{0\}, \\
0 & \mbox{for} & z=0,  \end{array} \right.    
\]
with $y\in \mathbb{R}$
\[
\lceil z \rceil_{d} =
\begin{cases}
\lceil y \rceil \overline \omega_{d} & \text{if }z = y \omega_d, \\
\lceil y \rceil  \omega_{d} & \text{if }z = y \overline \omega_d. \\ 
\end{cases}
\]
It is easy to see that we have $S_{d}(\{y \omega_{d} : 0 \le y <1\}) = \{y \overline{\omega}_{d} : 0 \le y < 1\}$, and $S_{d}(\{y \overline{\omega}_{d} : 0 \le y <1\}) = \{y \omega_{d} : 0 \le y < 1\}$.  
We put
\[
B_{d, n}(z) = \left\{\begin{array}{lll} \left\lceil \tfrac{1}{S_{d}^{n-1}(z)}\right\rceil_d
&\mbox{if} & S_{d}^{n-1}(z) \in \mathbb J_{d} \setminus \{0\}, \\
\infty         & \mbox{if} &   S_{d}^{n-1}(z) = 0 \end{array} \right.    
\]
and have 
\[
z=\displaystyle \frac1{B_{d, 1}(z)- 
 \displaystyle  \frac1{B_{d, 2}(z) - 
 \displaystyle \frac1{\ddots}}}.
\]
for $ z \in \mathbb J_{d}$.  In particular, for $z \in \{y\omega_d  : 0 < y < 1\}$, we have 
$B_{d, 2k+1}(z), k \ge 0$, are of the form $m \overline{\omega}$, $m \in \mathbb N$ and 
$B_{d, 2k}(z), k \ge 1$, of the form $m \omega$, $m \in \mathbb N$.  As before, we put 
\[
\begin{pmatrix} -P_{d, n-1}(z) & P_{d, n}(z) \\  -Q_{d, n-1}(z) &Q_{d, n}(z) 
\end{pmatrix} 
=
\begin{pmatrix} 0 & 1 \\ -1 & B_{d, 1}(z) \end{pmatrix} 
\begin{pmatrix} 0 & 1 \\ -1 & B_{d, 2}(z) \end{pmatrix}  \,\,\cdots \,\, 
\begin{pmatrix} 0 & 1 \\ -1 & B_{d, n}(z) \end{pmatrix} \qquad (n\ge1), 
\] 
with $P_{d, 0}(z) = 0$ and $Q_{d, 0}(z) = 1$.

The Euclidean algorithm implies that the algorithm $S_{d}$ yields a finite continued fraction expansion for each $z \in \mathbb J_{d} \cap \mathbb Q(\sqrt{-d})$.  Thus $\mathbb J_{d} \cap \mathbb Q(\sqrt{-d}) = \left\{\tfrac{P_{d, n}(z)}{Q_{d, n}(z)} : z \in \mathbb J_{d}, n\ge 0 \right\}$. 

Set
\begin{equation}\label{eq:lld}
\ell=\ell(d)= \begin{cases} 4 &\text{if } d\in\{2,7\}, \\ 6 &\text{if } d=11. \end{cases}
\end{equation}
By induction, we have the following result.

%%%%%%%%%%%%%%%%%%%%%%%%%
\begin{proposition}\label{prop:pqPQ}
Let $p_{\ell,k}$ $q_{\ell,k}$, $P_{d,k}$, and $Q_{d,k}$ and $\ell=\ell(d)$ be defined as above. Then the following identities hold.
\begin{itemize}
\item[(i)]\; For $k \ge 0$, $d\in\{2, 7,11\}$ and $y \omega_d \in \mathbb J_{d}$ we have
\[
\begin{pmatrix} P_{d, 2k+1}(y \omega_d) \\  Q_{d, 2k+1}(y \omega_d) 
\end{pmatrix} 
= 
\begin{pmatrix}
p_{\ell, 2k+1}\Big(y \sqrt{\tfrac\ell2}\Big)  \\ q_{\ell, 2k+1}\Big(y \sqrt{\tfrac\ell2}\Big) \cdot
\overline\omega_d \sqrt{\tfrac2\ell}  \end{pmatrix}.
\]
\item[(ii)] \; For $k \ge 1$, $d\in\{2, 7,11\}$ and $y \omega_d \in \mathbb J_{d}$ we have
\[
\begin{pmatrix} P_{d, 2k}(y \omega_d) \\  Q_{d, 2k}(y \omega_d) 
\end{pmatrix} 
=  
\begin{pmatrix}
p_{\ell, 2k}\Big(y \sqrt{\tfrac\ell2}\Big) \cdot
\omega_d\sqrt{\tfrac2\ell}  \\ q_{\ell, 2k}\Big(y \sqrt{\tfrac\ell2}\Big)\end{pmatrix}.
\]
\end{itemize}
\end{proposition} 

Note that $\Big|\overline\omega_d \sqrt{\tfrac2\ell}  \Big| = \Big| \omega_d\sqrt{\tfrac2\ell} \Big|=1$, hence, associated numerators denominators have the same modulus. This will entail that the associated Ford circles or Ford spheres have the same radius.

We can now finish the proof of Theorem~\ref{th:Hecke}.

\begin{proof}[Proof of Theorem~\ref{th:Hecke}]
Let $\ell=\ell(d)$ be as in \eqref{eq:lld}. For any $R \in \mathbb J_{d} \setminus \{0\}$, there exist $0 < a < b \in \mathbb N$, $a, b$, coprime, such that $R = \tfrac{a}{b} \omega_d$ (or $R = \tfrac{a}{b}\overline\omega_d$, which is treated in the same way).  Thus there exists $m\ge 1$ such that $S_{d}^{m-1}(R) \ne 0$ and $S_{d}^{m}(R) = 0$.  Proposition~\ref{prop:pqPQ} implies that 
\[
\frac{a}{b} \omega_d  = \frac{P_{d,m}\big(\tfrac{a}{b} \omega_d\big)}{Q_{d,m}\big(\tfrac{a}{b} \omega_d\big)} \quad \hbox{and}\quad 
\frac{a}{b}\sqrt{\frac\ell2} = \frac{p_{\ell,m}\Big(\tfrac{a}{b} \sqrt{\tfrac\ell2}\Big) }{q_{\ell,m}\Big(\tfrac{a}{b} \sqrt{\tfrac\ell2}\Big) }. 
\]
%$F_{4}^{m-1}\Big(\tfrac{a}{b}\sqrt{\tfrac\ell2}\Big) \ne 0$ and $F_{4}^{m-1}\Big(\tfrac{a}{b}\sqrt{\tfrac\ell2}\Big) = 0$. Furthermore, $b_{\ell, k}$ can be obtained from $B_{d, k}$  by replacing $\omega_{d}$, $\overline{\omega}_{d}$ with $\sqrt{2}$ or $\sqrt{3}$. 
By  Proposition~\ref{prop:pqPQ} we have 
\[
\big|P_{d,m}\big(\tfrac{a}{b} \omega_d\big)\big|=\Big|p_{\ell,m}\Big(\tfrac{a}{b} \sqrt{\tfrac\ell2}\Big)\Big| \quad\text{and}\quad \big|Q_{d,m}\big(\tfrac{a}{b} \omega_d\big)\big|=\Big|q_{\ell,m}\Big(\tfrac{a}{b} \sqrt{\tfrac\ell2}\Big)\Big|.
\]
Therefore, the Ford sphere $\mathcal{F}_{d}\left(\tfrac{a}{b} \omega_d\right)$ has the same radius as the Ford circle $\mathcal{F}\Big(\tfrac{a}{b}\sqrt{\tfrac\ell2}\Big)$  and $\mathcal{F}_{d}\left(\tfrac{a}{b} \omega_d\right) \cap \mathcal{F}_{d}\left(\tfrac{c}{d} \omega_d\right)\not=0$ if and only of $\mathcal{F}\Big(\tfrac{a}{b}\sqrt{\tfrac\ell2}\Big) \cap \mathcal{F}\Big(\tfrac{c}{d}\sqrt{\tfrac\ell2}\Big) \not=0$. Thus via the mapping $\mathbb I_{\ell} \to \mathbb J_\ell;  \; x \sqrt{\tfrac{\ell}{2}} \mapsto x \omega_{d}$, we get a one-to-one correspondence between the edges in the Farey graph associated with $H_\ell$ and the edges in the subgraph of $\mathcal E_{d}$ restricted to vertices on $\{y\omega_{d} : 0 \le y <1\}$.  
\end{proof}

In a similar way, we can connect some other types of continued fraction 
expansions associated with Hecke groups $H_{4}$ and $H_{6}$, {\em e.g.}, the ones discussed in 
\cite{D-K-S,K-S-S,R}, to continued fractions associated with $\mathbb Q(\sqrt{-d})$ for $d\in\{1, 2, 3, 7, 11\}$.

The following consequence of Theorem~\ref{th:Hecke} will be needed in the next section.

\begin{proposition}%[{\em cf}.~\cite{S-W}]
\label{prop:wallTessellation}
Let $d\in\{1,2,3,7,11\}$. For each wall $X$ and each $A\in \mathrm{PSL}(2,\mathbb{Z}[\omega_d])$ the set $A\cdot X$ is tessellated by Farey polygons. 
\end{proposition}

\begin{proof}
As mentioned at the beginning of this section, for $d\in\{1,3\}$ the Farey tessellation of a wall $X$ is the classical one. For $d\in\{2,7,11\}$ the Farey tessellation of a wall $X$ is either the classical one or, in view of Theorem~\ref{th:Hecke}, the tessellation induced by a Farey graph associated with one of the Hecke groups $H_4$ or $H_6$. The classical Farey graph induces a Farey  tessellation of $X$ by Farey triangles (see Figure~\ref{fig:ClassicalFarey}). Moreover, according to \cite{S-W}, $H_4$ induces a Farey tessellation by Farey quadrangles (see Figure~\ref{fig:Farey7wall}), and $H_6$ induces a Farey tessellation by Farey hexagons. This proves the proposition. 
\end{proof}

\section{The Farey tessellation of $\mathbb H^{3}$}\label{sec:polyhedron}

For $d\in\{1,2,3,7,11\}$ we will now use the standard cell $\mathcal{C}_d$ from Definition~\ref{def:standard} to provide analogs of the classical Farey tessellation of~$\mathbb{H}^2$ in $\mathbb{H}^3$.

\begin{lemma}\label{lem:boundFarey}
For each $d\in\{1,2,3,7,11\}$, the boundary of $\mathcal{C}_d$ is a union of Farey $n$-gons ($n\ge 3$).
\end{lemma}

\begin{proof}
Let $B\in \mathrm{PSL}(2,\mathbb{Z}[\omega_d])$ be arbitrary and note that $\mathcal{C}_d \cap B\cdot X$ is a subset of  $B\cdot X$ whose boundary $\partial(\mathcal{C}_d \cap B\cdot X)$ (relative to $\mathcal{C}_d$) is contained in a union of topologically $1$-dimensional sets of the form  $A\cdot X \cap B\cdot X$, where $A\in \mathrm{PSL}(2,\mathbb{Z}[\omega_d])$. Because, by Proposition~\ref{prop:wallTessellation},
$A\cdot X$ and $B\cdot X$ are tessellated by Farey polygons, Proposition~\ref {lem:GeodesicIntersectGeneral} implies that $A\cdot X \cap B\cdot X$ is a union of boundaries of Farey polygons (and cusps). Thus $\mathcal{C}_d \cap B\cdot X$ is a union of Farey polygons. Since 
\[
\partial \mathcal{C}_d = \bigcup_{B\in \mathrm{PSL}(2,\mathbb{Z}[\omega_d]),\, X \text{ wall}} (\mathcal{C}_d \cap B\cdot X), 
\]
the lemma follows (note that Farey digons can be neglected because they are covered by Farey $n$-gons with $n\ge 3$).
\end{proof}

\begin{corollary}\label{cor:finitecusp}
Suppose that the set $S$ of cusps of the standard cell $\mathcal{C}_d$  is finite. Then $\mathcal{C}_d$ is bounded by finitely many Farey polygons each of whose cusps are contained in $S$.
\end{corollary}

\begin{proof}
By Lemma~\ref{lem:boundFarey} the cell $\mathcal{C}_d$ is bounded by Farey polygons. Each Farey polygon bounding $\mathcal{C}_d$ has to have edges passing from one cusp of $\mathcal{C}_d$ to another. Since there are only finitely many such cusps, there are only finitely many possible Farey polygons connecting these cusps.
\end{proof}

\begin{figure}[h]
\begin{tikzpicture}[scale=0.3]
\node at (12.5,7.5) {$\frac{\omega_1}{1}$};
\node at (27.5,7.5) {$\frac{\omega_1+1}{1}$};
\node at (12.5,-7.5) {$\frac{0}{1}$};
\node at (27.5,-7.5) {$\frac{1}{1}$};
\node at (20,1.7) {$\frac{1}{1-\omega_1}$};

\node[draw=none,fill=none] at (0,-1){\includegraphics[width=.43\textwidth]{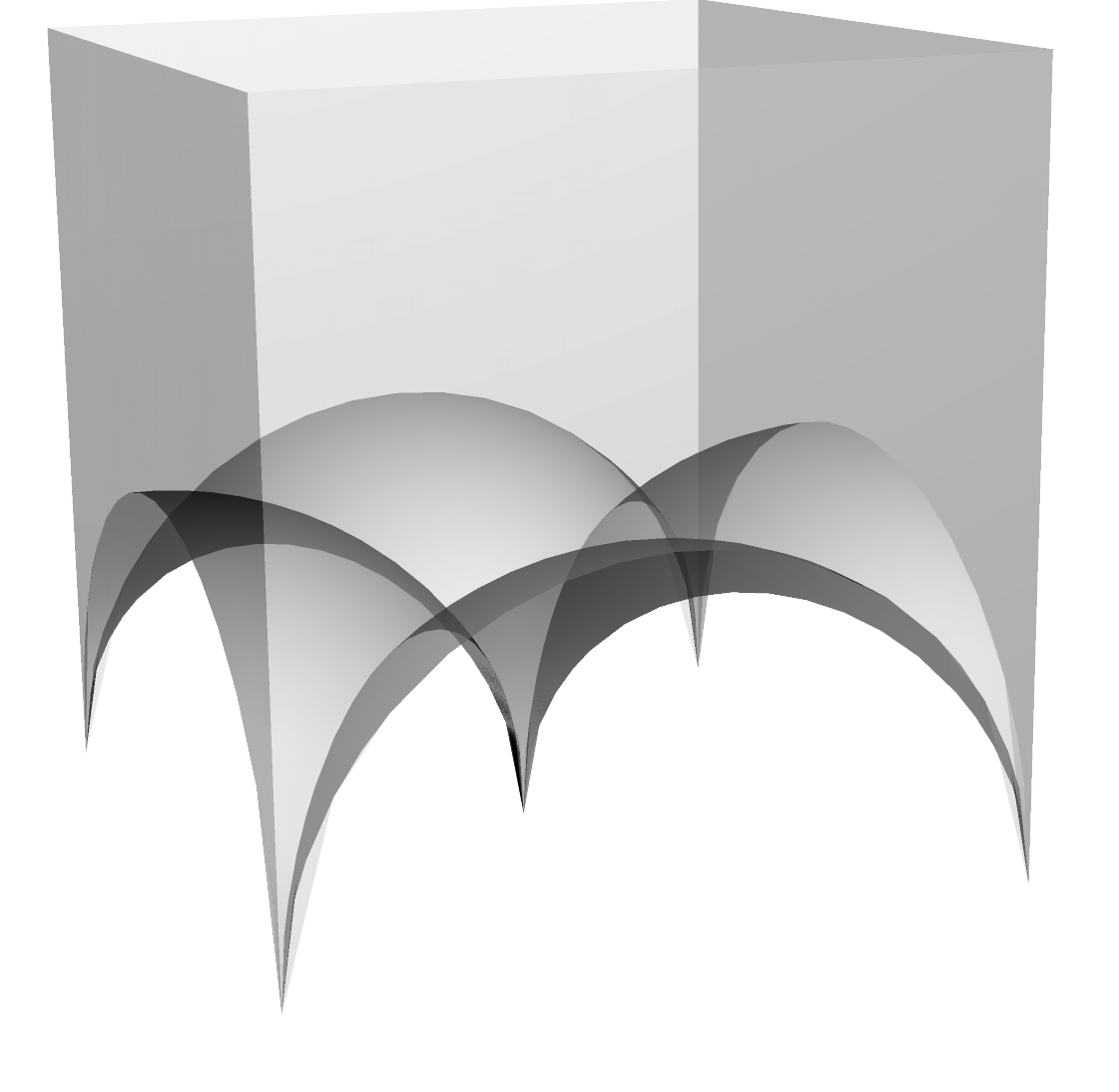}};
\node[draw=none,fill=none] at (20,0){\includegraphics[width=.35\textwidth]{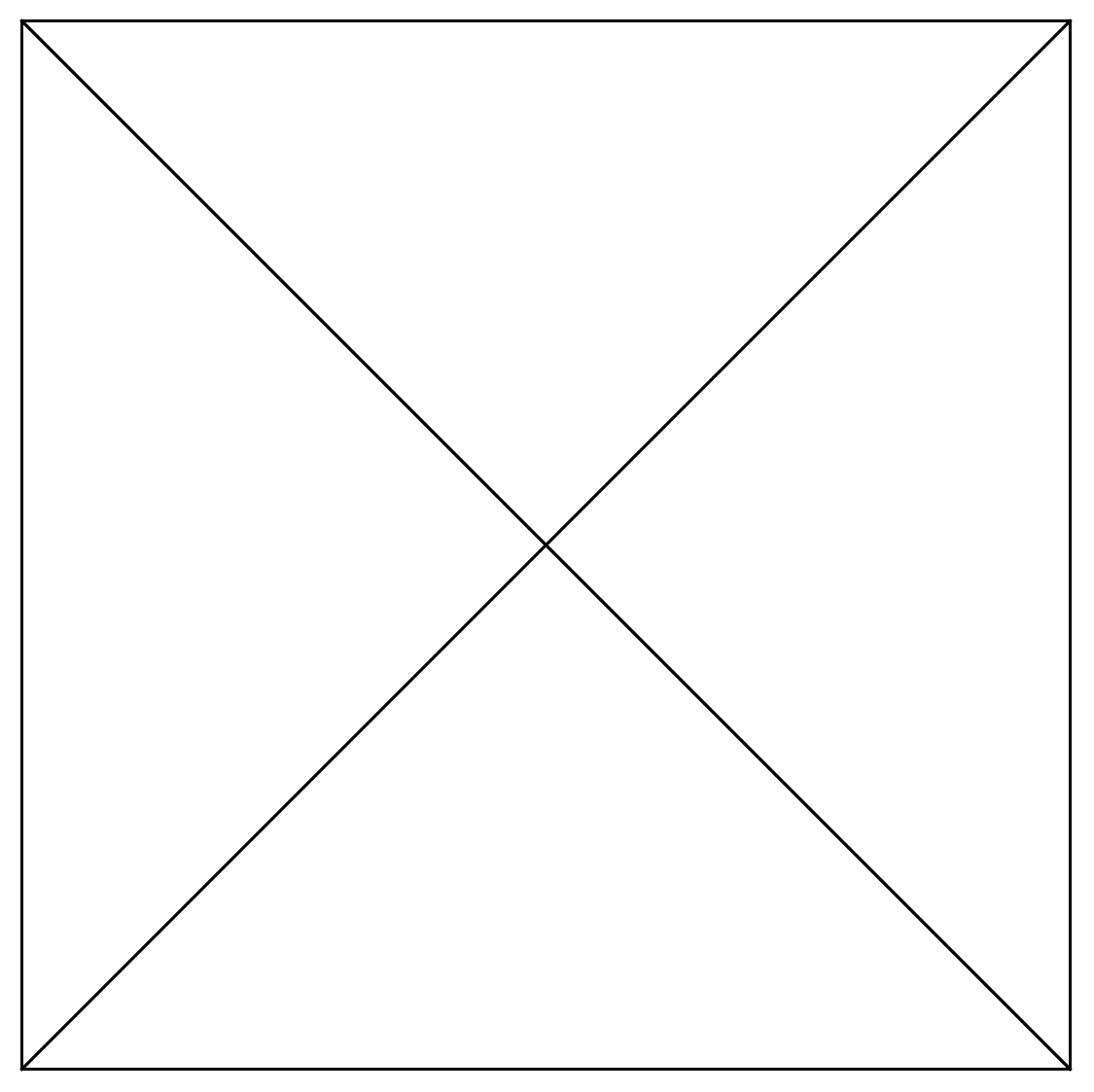}};
\end{tikzpicture}
\caption{A view of the lower part of the octagonal standard cell $\mathcal{C}_1$ (left), and $\mathcal{C}_1$ seen from the cusp at~$\infty$ (right; compare \cite[Figure~2]{SchmidtAsmus:69}). In the right figure the Farey triangles containing $\infty$ correspond to the boundary lines of the square. \label{fund1}}
\end{figure}
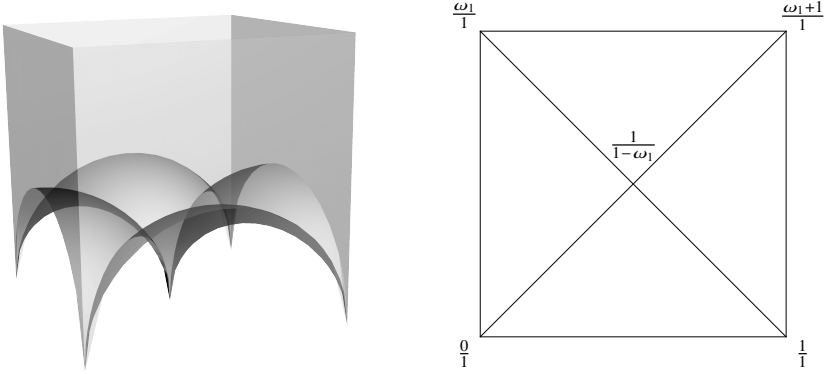

We call two rational cusps $\frac pq,\frac rs$ \emph{neighboring cusps} if $|ps-qr|=1$.

\begin{proposition}
The standard cells $\mathcal{C}_d$ with $d\in\{1,2,3,7,11\}$ are characterized as follows.
\begin{itemize}
\item[(i)]\, The standard cell $\mathcal{C}_1$ is an octahedron whose boundary consists of all Farey triangles whose edges connect neighboring cusps taken from the set $\big\{\frac01,\frac10,\frac11,\frac {\omega_1}1,\frac{\omega_1+1}1,\frac{1}{1-\omega_1}\big\}$. These $8$ Farey triangles are depicted in Figure~\ref{fund1}.

\item[(ii)]\, The standard cell $\mathcal{C}_2$ is a tetradecahedron whose boundary consists of all Farey triangles and Farey quadrangles whose edges connect neighboring cusps taken from 
$
\big\{\frac01,\frac10,\frac11,\frac{\omega_2}1,\frac{\omega_2+1}1,\frac{1}{-\omega_2},\frac{1-\omega_2}{-\omega_2},\frac{\omega_2+1}{2},\frac{\omega_2-1}{\omega_2+1},\frac{2}{1-\omega_2},\frac{1}{1-\omega_2},\frac{\omega_2}{\omega_2+1}\big\}.
$
These $14$ Farey polygons are depicted in Figure~\ref{fund2}.

\item[(iii)]~The standard cell $\mathcal{C}_3$ is the tetrahedron whose boundary consists of all Farey triangles whose edges connect neighboring cusps taken from $\big\{\frac01,\frac10,\frac11,\frac{\omega_3}{1}\big\}$. These $4$ Farey triangles are depicted in Figure~\ref{fund3}.

\item[(iv)]~The standard cell $\mathcal{C}_7$ is the pentahedron whose boundary consists of all Farey triangles and Farey quadrangles whose edges connect neighboring cusps taken from $\big\{\frac01,\frac10,\frac11,\frac{\omega_7}{1},\frac{1}{1-\omega_7},\frac{\omega_7-1}{\omega_7}\big\}$. 
These $5$ Farey polygons are depicted in Figure~\ref{fund3}.

\item[(v)]\,The standard cell $\mathcal{C}_{11}$ is the octahedron whose boundary consists of all Farey triangles and Farey hexagons whose edges connect neighboring cusps taken from
{
$
\big\{\frac01,\frac10,\frac11,\frac{\omega_{11}}{1},\frac{1}{1-\omega_{11}},\frac{\omega_{11}-1}{\omega_{11}},\frac{\omega_{11}}{2},\frac{\omega_{11}+1}{2},\frac{\omega_{11}-2}{\omega_{11}}, \frac{2}{1-\omega_{11}},\frac{\omega_{11}-1}{\omega_{11}+1},\frac{2}{2-\omega_{11}}\big\}. 
$
}
These $8$ Farey polygons are depicted in Figure~\ref{fund3}.
\end{itemize}
\end{proposition}

\begin{proof}
We consider the case $d=1$. The octagon $\mathcal{O}$ described in the statement has six cusps (one of them being $\infty=\frac10$; see Figure~\ref{fund1}). Since $\mathcal{C}_1 \subseteq \mathcal{O}$, the set of cusps of the fundamental cell $\mathcal{C}_1$ is a subset of the set formed by the six cusps of $\mathcal{O}$. However, each pair of neighboring cusps is connected by a side of a Farey triangle in $\mathcal{O}$. Thus Corollary~\ref{cor:finitecusp} implies that $\mathcal{C}_1$ cannot be bounded by any other Farey polygon.

By analogous reasoning we get the result for the other values of $d$. Just observe that the polygons depicted in Figure~\ref{fund2} and Figure~\ref{fund3} are all Farey polygons. 
\end{proof}

\begin{figure}[h]
\begin{center}
\begin{tikzpicture}[scale=0.3]
\node at (12.5,10) {$\frac{\omega_2}{1}$};
\node at (28,10) {$\frac{\omega_2+1}{1}$};
\node at (18.6,4.1) {$\frac{\omega_2-1}{\omega_2+1}$};
\node at (21.4,4.1) {$\frac{2}{1-\omega_2}$};
\node at (18,0) {$\frac{\omega_2+1}{2}$};
\node at (12,0) {$\frac{1}{-\omega_2}$};
\node at (28.5,0) {$\frac{1-\omega_2}{-\omega_2}$};
\node at (18.6,-4.1) {$\frac{1}{1-\omega_2}$};
\node at (21.4,-4.1) {$\frac{\omega_2}{\omega_2+1}$};
\node at (12.5,-10) {$\frac{0}{1}$};
\node at (27.5,-10) {$\frac{1}{1}$};
\node[draw=none,fill=none] at (20,0){\includegraphics[width=.35\textwidth]{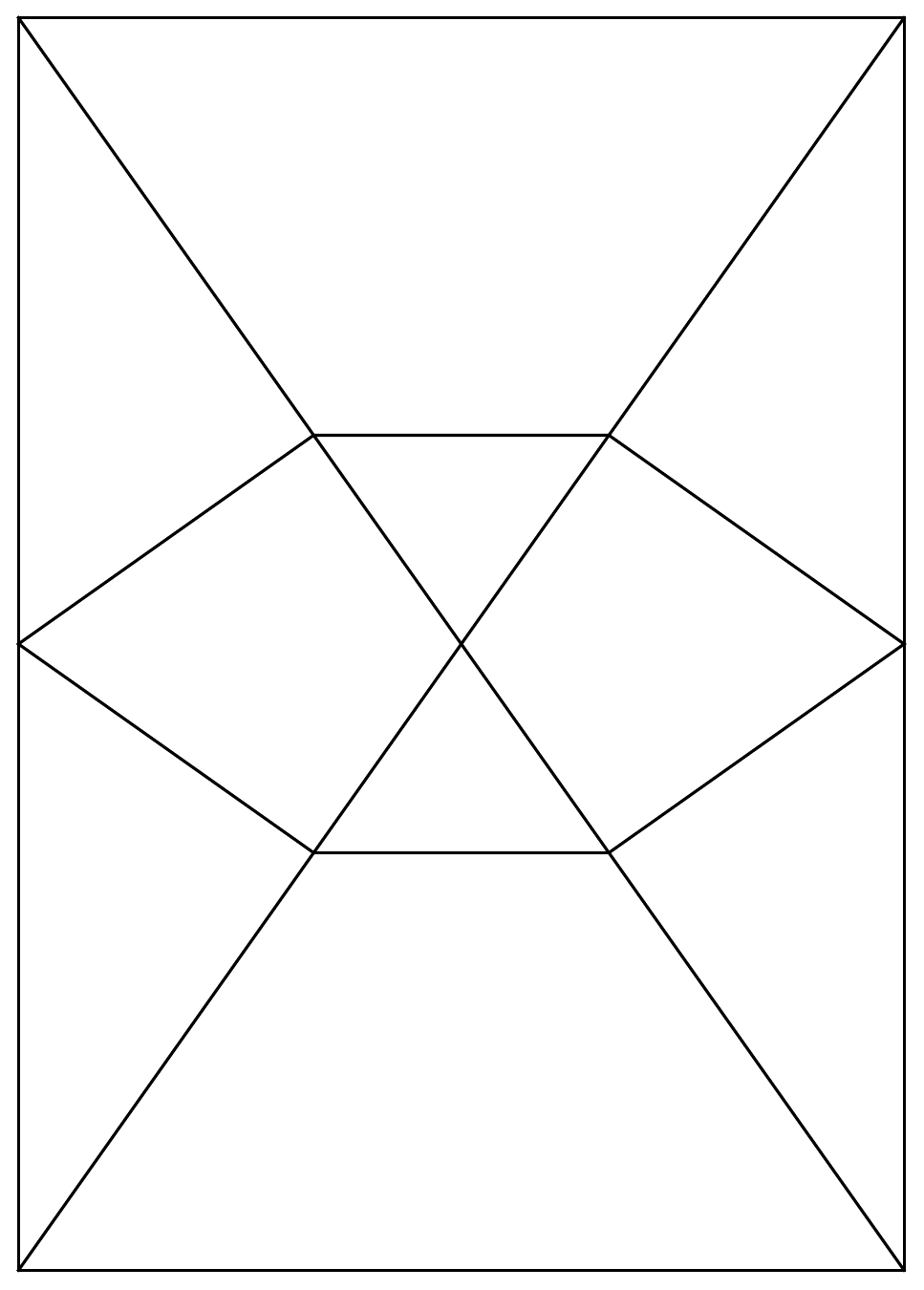}};
\end{tikzpicture}
\end{center}
\caption{The standard cell $\mathcal{C}_2$ seen from the cusp at $\infty$ (compare Figure~\ref{fig:sqrt2Ed} and \cite[Figure~2]{SchmidtAsmus:69}). The figure shows the Farey polygons in $\partial \mathcal{C}_2$ not containing $\infty$. The ones containing $\infty$ are contained in the walls (which are seen as the boundary lines of the rectangle). These are $\FT\big(\frac{0}{1},\frac{1}{0},\frac{1}{1}\big)$, $\FT\big(\frac{{\omega_{2}}}{1},\frac{1}{0},\frac{{\omega_{2}}+1}{1}\big)$, $\FQ\big(\frac{0}{1},\frac{1}{0},\frac{{\omega_{2}}}{1},\frac{1}{-{\omega_{2}}}\big)$, $\FQ\big(\frac{1}{1},\frac{1}{0},\frac{{\omega_{2}}+1}{1},\frac{1-{\omega_{2}}}{-{\omega_{2}}}\big)$. 
\label{fund2}}
\end{figure}
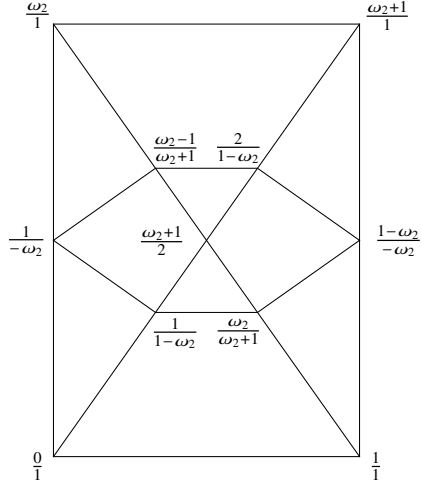

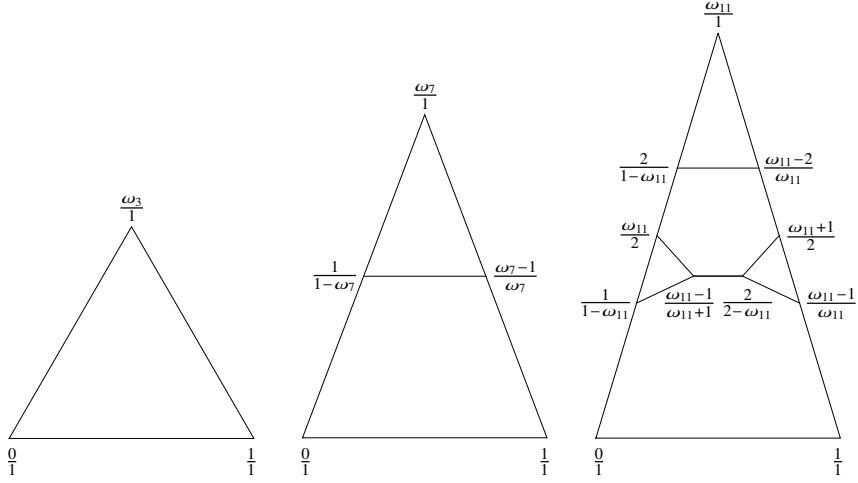
\begin{figure}[h]
\begin{tikzpicture}[scale=0.3]
\node at (-5.3,-10) {$\frac{0}{1}$};
\node at (5.3,-10) {$\frac{1}{1}$};
\node at (0,1.2) {$\frac{\omega_3}{1}$};

\node at (13,6.2) {$\frac{\omega_7}{1}$};
\node at (13-5.3,-10) {$\frac{0}{1}$};
\node at (13+5.3,-10) {$\frac{1}{1}$};
\node at (13+4,-1.8) {$\frac{\omega_7-1}{\omega_7}$};
\node at (13-4,-1.8) {$\frac{1}{1-\omega_7}$};

\node at (26-5.3,-10) {$\frac{0}{1}$};
\node at (26+5.3,-10) {$\frac{1}{1}$};
\node at (26,9.8) {$\frac{\omega_{11}}{1}$};
\node at (26+3.1,3) {$\frac{\omega_{11}-2}{\omega_{11}}$};
\node at (26-3.3,3) {$\frac{2}{1-\omega_{11}}$};
\node at (26+4.1,0) {$\frac{\omega_{11}+1}{2}$};
\node at (26-3.7,0) {$\frac{\omega_{11}}{2}$};
\node at (26+5,-3) {$\frac{\omega_{11}-1}{\omega_{11}}$};
\node at (26-5,-3) {$\frac{1}{1-\omega_{11}}$};

\node at (26-1.3,-3) {$\frac{\omega_{11}-1}{\omega_{11}+1}$};
\node at (26+1.3,-3) {$\frac{2}{2-\omega_{11}}$};

\node[draw=none,fill=none] at (0,-4.3){\includegraphics[width=.28\textwidth]{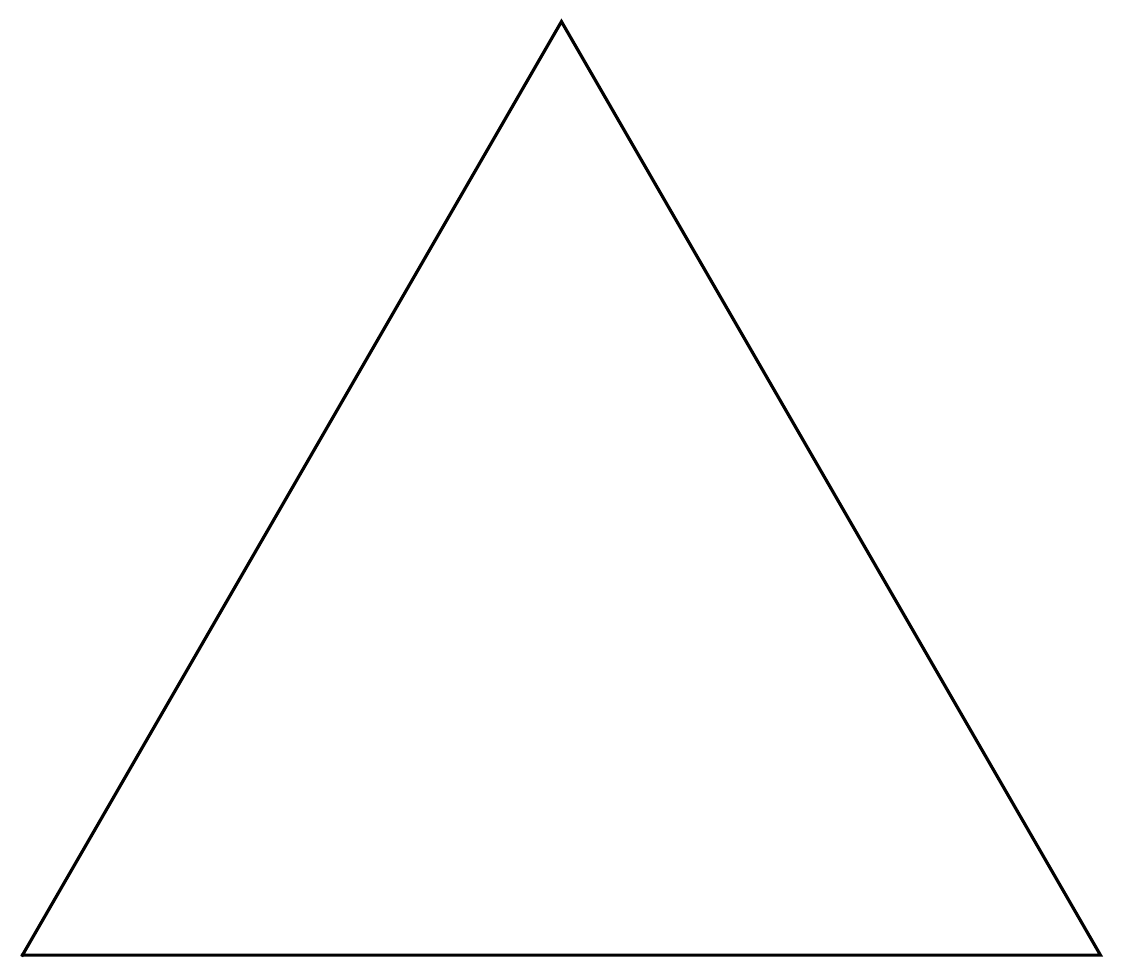}};
\node[draw=none,fill=none] at (13,-1.8){\includegraphics[width=.28\textwidth]{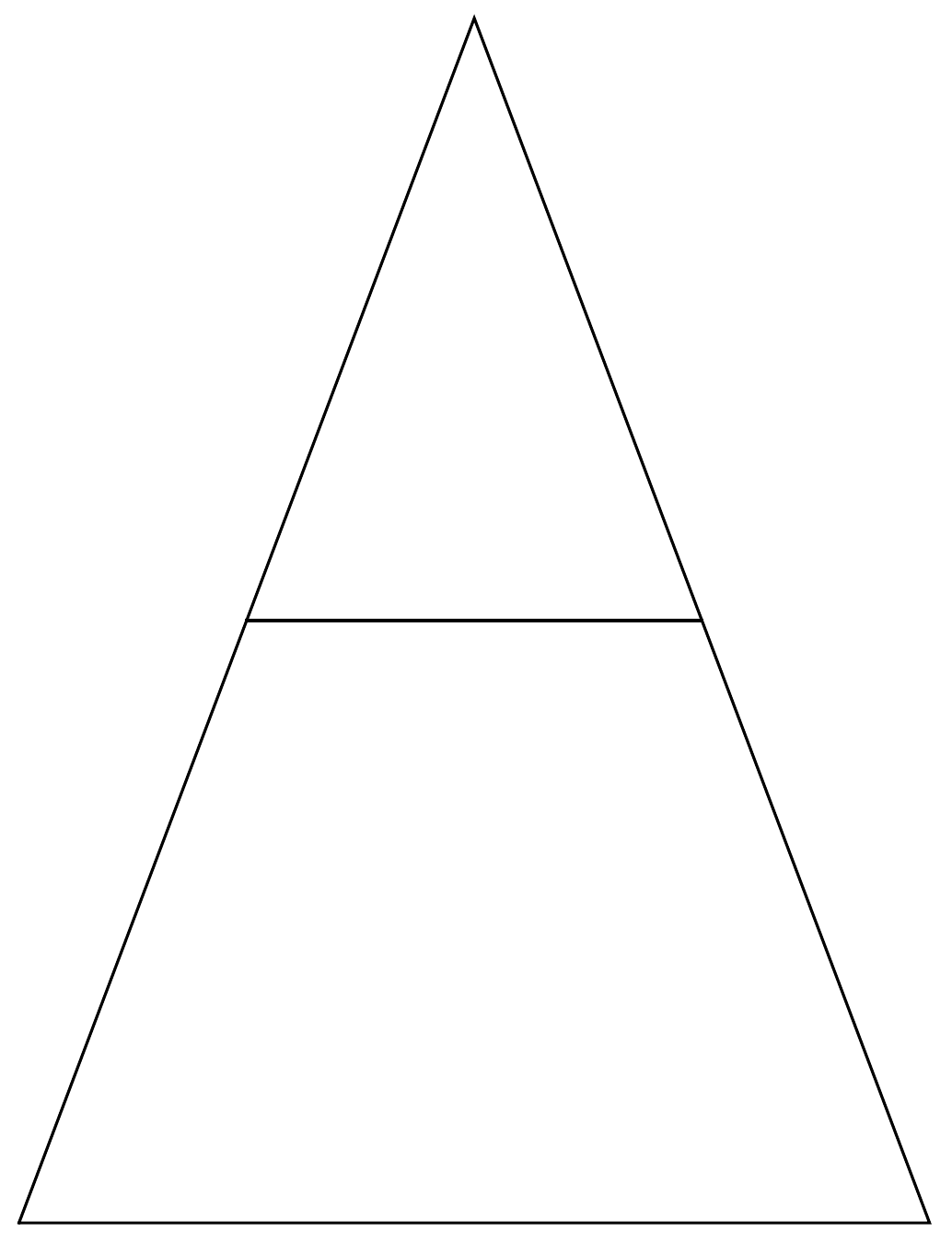}};
\node[draw=none,fill=none] at (26,0){\includegraphics[width=.28\textwidth]{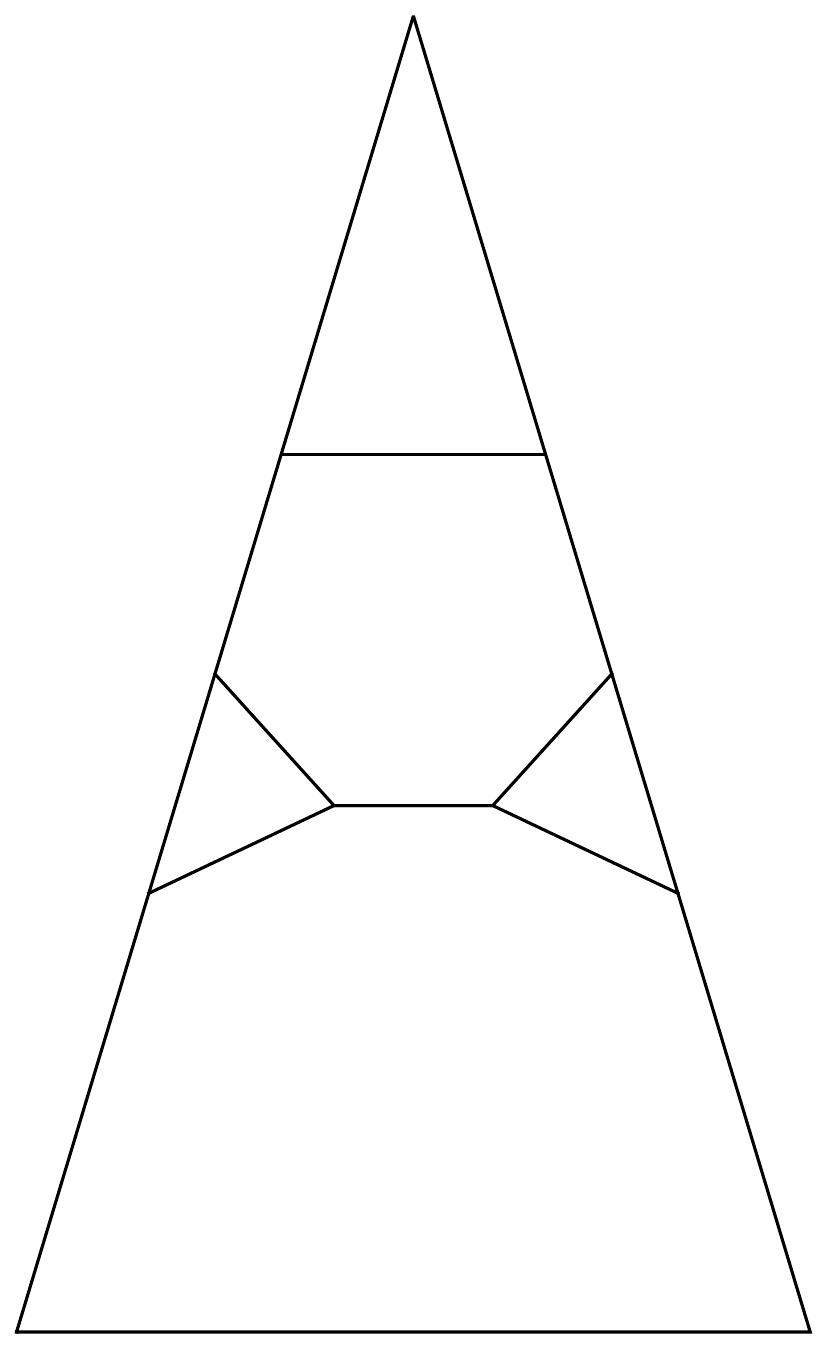}};
\end{tikzpicture}
\caption{The standard cells $\mathcal{C}_3$, $\mathcal{C}_7$, and $\mathcal{C}_{11}$ (from left to right) seen from the cusp at $\infty$. Compare \cite[Figures~1 and 4]{SchmidtAsmus:69} as well as \cite[Figure~1]{SchmidtAsmus:78}. Note that each edge of each triangle corresponds to a Farey polygon one of whose cusps is $\infty$.
\label{fund3}}
\end{figure}

Let 
$
\mathcal{C}_d'=\{(\overline{z}, y) \,:\, (z,y)\in  \mathcal{C}_d\}.
$
We can now give the following generalization of the classical Farey tessellation using the sets $\mathcal{C}_d$ and $\mathcal{C}'_d$ as fundamental domains.

\begin{theorem}\mbox{}
\begin{itemize}
\item[(i)]\, If $d\in\{1,2\}$ then the collection $\mathcal{T}_d=\{ A\cdot \mathcal{C}_d\,:\, A\in \mathrm{PSL}(2,\mathbb{Z}[\omega_d]) \}$ forms a tessellaton of $\mathbb{H}^3$. 
\item[(ii)]\, If $d\in\{3,7,11\}$ then the collection $$\mathcal{T}_d=\{ A\cdot \mathcal{C}_d\,:\, A\in \mathrm{PSL}(2,\mathbb{Z}[\omega_d]) \} \cup \{ A\cdot \mathcal{C}'_d\,:\, A\in \mathrm{PSL}(2,\mathbb{Z}[\omega_d]) \}$$ forms a tessellaton of $\mathbb{H}^3$. 
\end{itemize}
The tessellation $\mathcal{T}_d$ is called the \emph{Farey tessellation} of $\mathbb{H}^3$ induced by $\mathrm{PSL}(2,\mathbb{Z}[\omega_d])$.
\end{theorem}

\begin{proof}
Using the generators of $\mathrm{PSL}(2,\mathbb{Z}[\omega_d])$ provided in \cite{Swan:71}, it is easy to see that for each face $f$ of the polyhedron $\mathcal{C}_d$ there is an element $A_f\in \mathrm{PSL}(2,\mathbb{Z}[\omega_d])$ for which $\mathcal{C}_d \cap (A_f\cdot\mathcal{C}_d) = f$ or, if $d\in\{3,7,11\}$, $\mathcal{C}_d \cap (A_f\cdot\mathcal{C}'_d) = f$. For $d\in\{3,7,11\}$ the same is true if the roles of $\mathcal{C}_d$ and $\mathcal{C}'_d$ are interchanged. 
Thus, each face $f$ of each polyhedron in $\mathcal{T}_d$ is equal to the intersection of two elements of $\mathcal{T}_d$.

By the definition of $\mathcal{C}_d$, for each $A \in \mathrm{PSL}(2,\mathbb{Z}[\omega_d])$ either $\mathcal{C}_d^\circ \cap (A\cdot\mathcal{C}_d)^\circ=\emptyset$ or $\mathcal{C}_d^\circ \cap (A\cdot\mathcal{C}_d)^\circ=\mathcal{C}_d^\circ$ holds. For $d\in \{3,7,11\}$ it follows that $\mathcal{C}_d^\circ \cap (A\cdot\mathcal{C}'_d)^\circ=\emptyset$ because $(\mathcal{C}'_d)^\circ$ is obviously disjoint from each set in $\{A\cdot \mathcal{C}_d\,:\,A \in \mathrm{PSL}(2,\mathbb{Z}[\omega_d])\}$.

It remains to show that $\mathcal{T}_d$ covers $\mathbb{H}^3$. Suppose this is not true. Then $K=\bigcup_{ \mathcal{D} \in  \mathcal{T}_d} \mathcal{D}$ is a strict subset of $\mathbb{H}^3$ and, hence, the boundary $\partial K$ of $K$ relative to $\mathbb{H}^3$ is nonempty. Choose $(z,x)\in \partial K$. Since $x>0$, there is a ball $B$ (w.r.t.\ the hyperbolic metric) 
centered at $(z,x)$. 
Because only finitely many elements of $\mathcal{T}_d$ intersect $B$, each of the sets $B \cap K$ and $B \cap (\mathbb{H}^3 \setminus K)$ is topologically  $3$-dimensional. Thus these two sets cannot be separated by a  $1$-dimensional set and, hence, $B\cap \partial K$ is $2$-dimensional. Thus there is $\mathcal{D}\in\mathcal{T}_d$ such that the ball $B$ contains an inner point (w.r.t.\ $\partial \mathcal{D}$) of a face $f$ of $\mathcal{D}$ which is contained in only one element of $\mathcal{T}_d$. This contradicts the first paragraph of this proof and, hence, $\mathcal{T}_d$ covers $\mathbb{H}^3$.
Summing up we proved that $\mathcal{T}_d$ tessellates~$\mathbb{H}^3$.
\end{proof}

The following lemma is an immediate consequence of the definitions.

\begin{lemma}
Let $d\in\{1,2,3,7,11\}$. The edges of the {\em Farey graph} $\mathcal{E}_{d}$ coincide with the edges of the polyhedra in the tessellation $\mathcal{T}_d$.
\end{lemma}

\begin{example}[The case $d=7$]
For $d=7$ the standard cell $\mathcal{C}_7$ is given by the $5$-hedron depicted in the middle of Figure~\ref{fund3}. Note that $\infty$ is a vertex of this $5$-hedron, three of its faces are walls. The lower wall is a Farey triangle, while the two other walls are Farey quadrangles. The reflection matrices for the five faces of the $5$-hedron $\mathcal{C}_7$ are given as follows. 
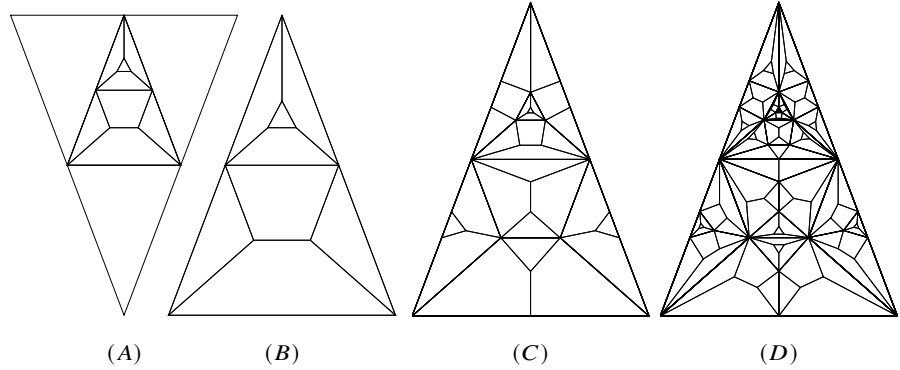
\begin{figure}
\begin{tikzpicture}[scale=0.3]
\node at (0,-7) {$(A)$};
\node at (7,-7) {$(B)$};
\node at (18,-7) {$(C)$};
\node at (29,-7) {$(D)$};

\node[draw=none,fill=none] at (0,0){\includegraphics[trim=0 -10 0 120, width=.26\textwidth]{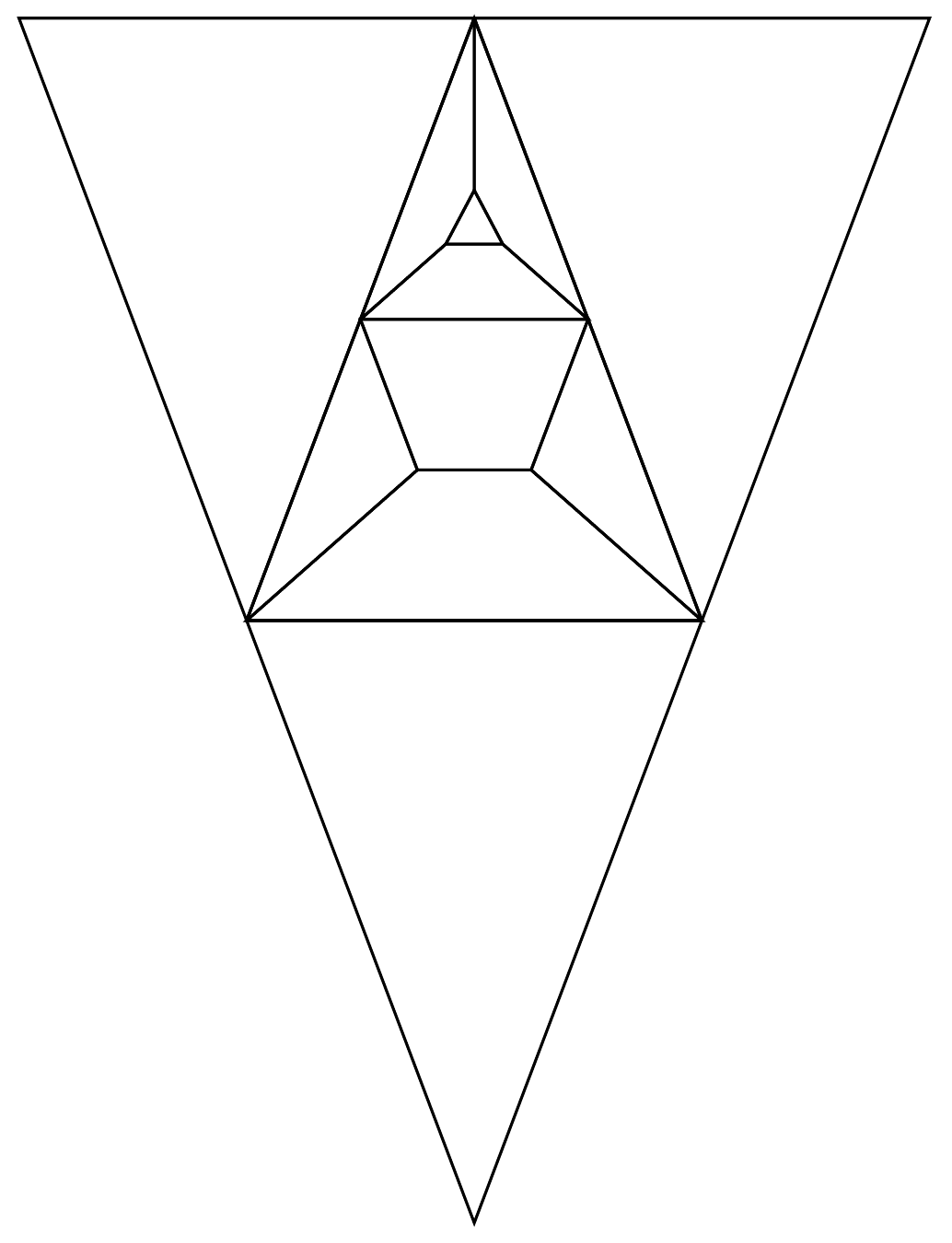}};
\node[draw=none,fill=none] at (7,0){\includegraphics[trim=0 -10 0 120, width=.26\textwidth]{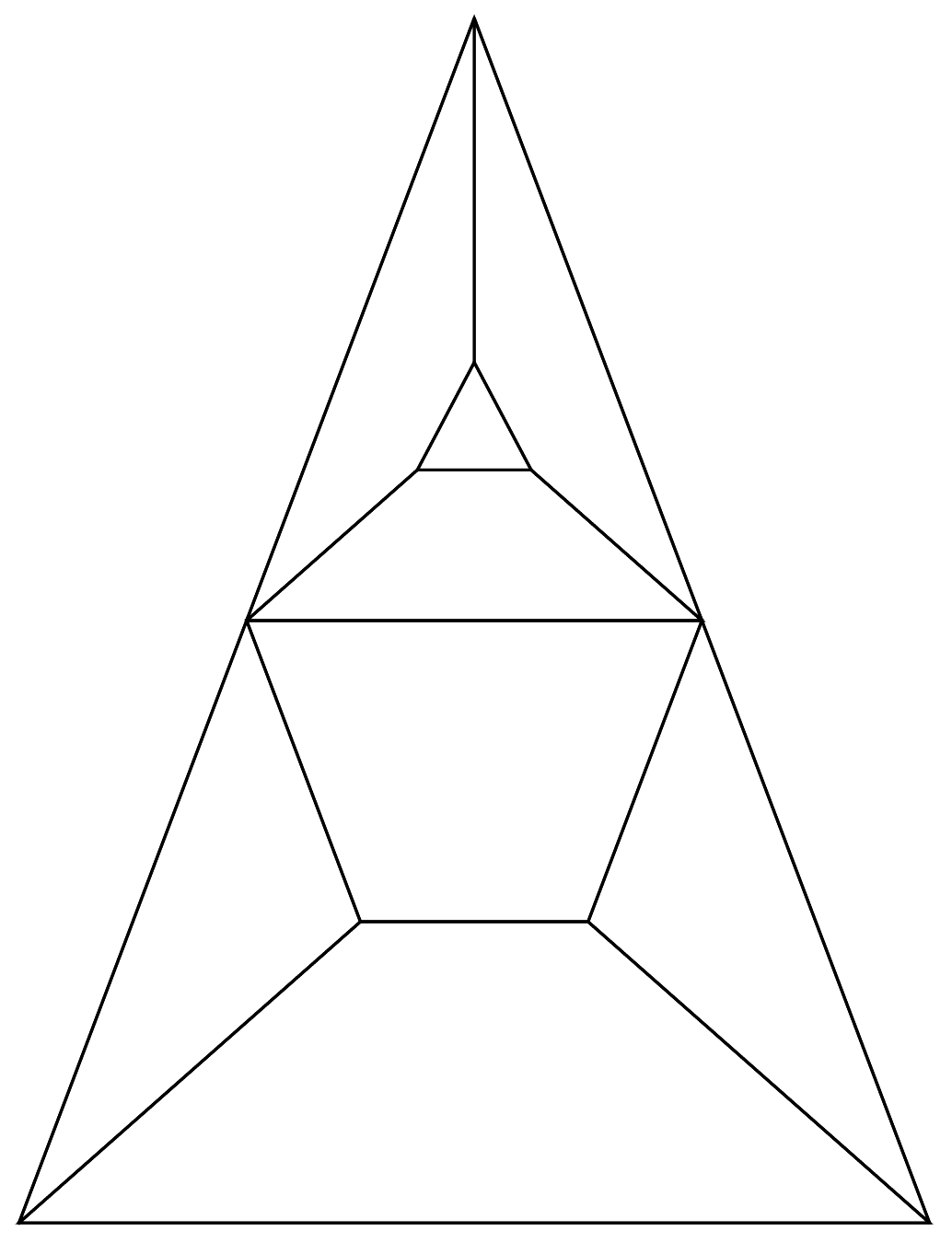}};
\node[draw=none,fill=none] at (18,0.935){\includegraphics[trim=0 -10 0 120, width=.26\textwidth]{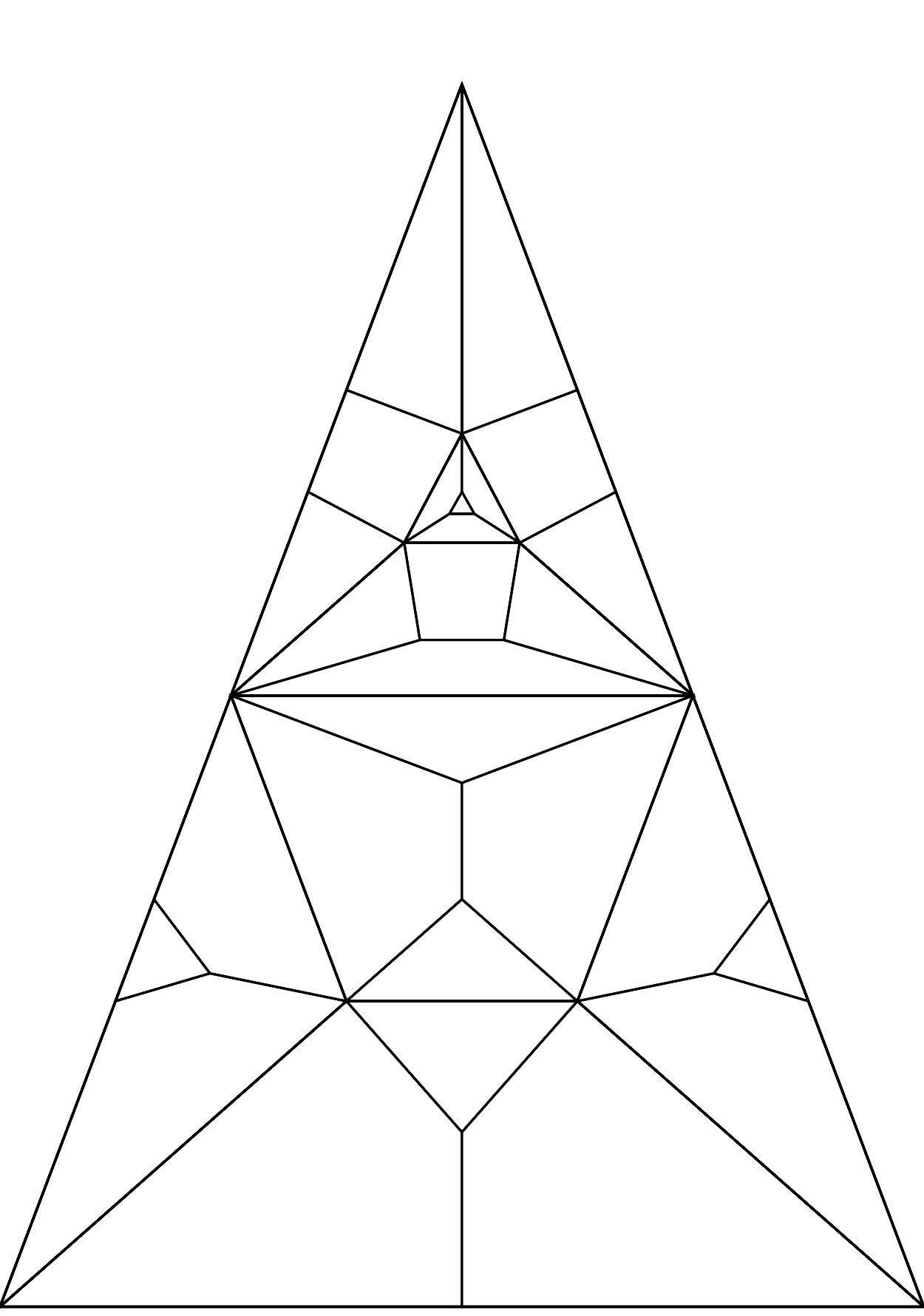}};
\node[draw=none,fill=none] at (29,0.935){\includegraphics[trim=0 -10 0 120, width=.26\textwidth]{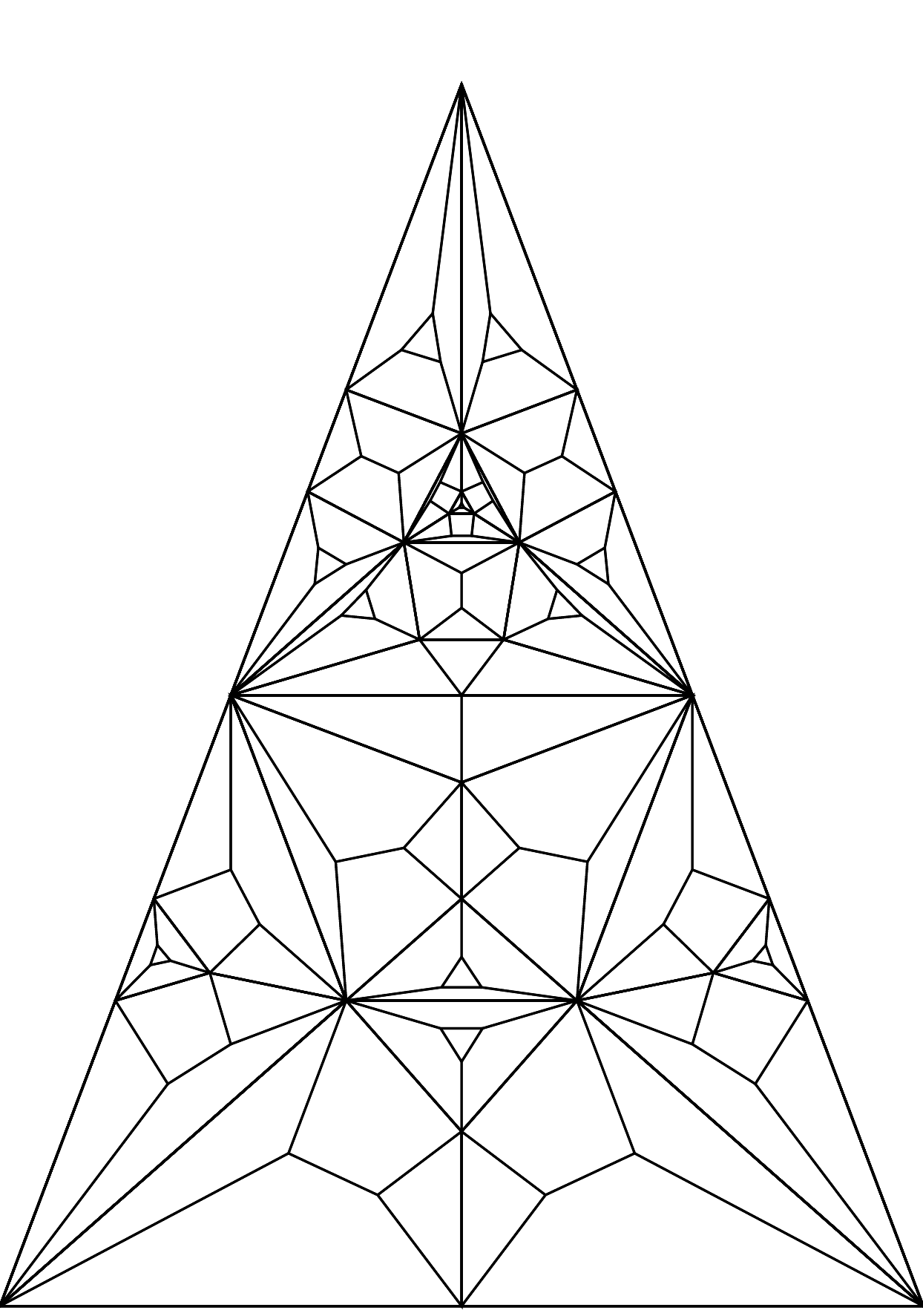}};
\end{tikzpicture}
\caption{(A) The reflections of $\mathcal{C}_7$ along its $5$ faces. (B) Only the reflections across Farey polygons ``below'' $\mathcal{C}_7$, drawn larger. (C) and~(D): Reflecting across the new faces.
\label{fig:7reflection}}
\end{figure}
\begin{itemize}
\item Reflection across lower wall: $R_1=\begin{pmatrix} -1&1\\0&1 \end{pmatrix}$.
\item Reflection across right wall: $R_2=\begin{pmatrix} -1&\omega_7+1\\0&1 \end{pmatrix}$.
\item Reflection across left wall: $R_3=\begin{pmatrix} -1&\omega_7\\0&1 \end{pmatrix}$.
\item Reflection across upper triangle: $R_4=\begin{pmatrix} 1&-\omega_7\\1&\omega_7-1 \end{pmatrix}$.
\item Reflection across lower quadrangle: $R_5=\begin{pmatrix} \omega_7&-1\\1&\omega_7-1 \end{pmatrix}$.
\end{itemize}
\begin{figure}[h]
\includegraphics[width=0.85\textwidth]{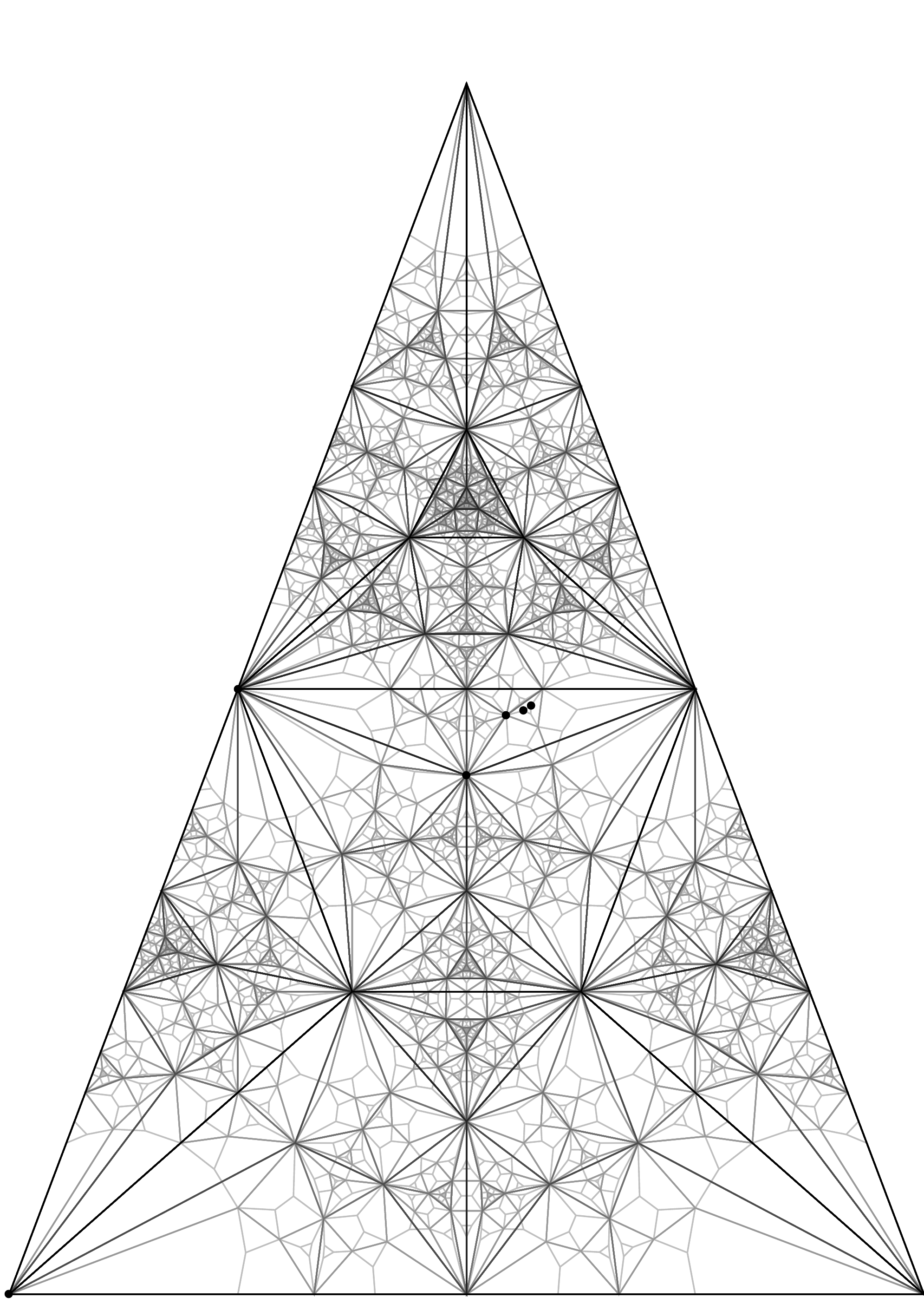} \qquad
\caption{Reflecting across the new faces for six times.
\label{fig:7reflection6}}
\end{figure}
Figure~\ref{fig:7reflection}~(A) shows $\mathcal{C}_7$ together with the reflections $R_k\mathcal{C}_7$ across its five faces ($k\in\{1,2,3,4,5\}$). Of course, each of these reflections is a $5$-hedron in $\mathbb{H}^3$. Observe that the reflections $R_4\mathcal{C}_7$ and $R_5\mathcal{C}_7$ in Figure~\ref{fig:7reflection}~(A) look like subdivisions of the upper triangle and the lower quadrangle of $\mathcal{C}_7$. This is what Schmidt~\cite{SchmidtAsmus:69,SchmidtAsmus:78} calls the {\it subdivision of Farey polygons}. We think that it is misleading to speak about such subdivisions, even though Schmidt mentions that these subdivisions are not to be understood geometrically. It seems to be better to keep in mind that these are just edges of reflected $5$-hedra (see also the work of Nakada~\cite{N-Monatshefte,N}, which shows that Farey polygons and Schmidt's complex continued fraction algorithms become more natural when one considers them in~$\mathbb{H}^3$). 
We concentrate on the part of the tessellation $\mathcal{T}_7$ ``below'' $\mathcal{C}_7$ (the remaing part of the tessellation can be obtained by translations and reflection across the origin).  This is shown in Figure~\ref{fig:7reflection}~(B). We note that these reflections yield two new polyhedra, and eight of their faces are new. We now reflect about the new faces. This yields the union of polyhedra depicted in Figure~\ref{fig:7reflection}~(C). We can go on and reflect along all the new faces to obtain Figure~\ref{fig:7reflection}~(D). Doing this for six times we obtain Figure~\ref{fig:7reflection6}. The ``generations'' of the reflections are indicated in different gray scales in this figure.\footnote{The meaning of the six black dots in the figure will become clear in Example~\ref{ex:exd7}
; they define a path that will be of interest to us.} Note that these ``generations'' are the generalization of the levels of the Farey tree in the classical case mentioned in the introduction.
The edges visible in Figures~\ref{fig:7reflection} and~\ref{fig:7reflection6} are edges of the Farey graph $\mathcal{E}_7$. Thus we can see walks of $\mathcal{E}_7$ in these images.
\end{example}

\section{Geodesic complex continued fraction algorithms}  

The aim of this section is to prove that the nearest integer continued fraction algorithms $T_1$, $T_2$, and $T_3$ from Definition~\ref{def:nicf} always produce geodesic continued fraction expansions. Recall that $T_d$ is defined in the region $U_d$ (see again Definition~\ref{def:nicf}) ($d\in \{1,2,3\}$).
For each $d \in \{1, 2, 3\}$, $\partial U_{d}$ consists of $4$ or $6$ line segments.  Suppose that $L_{d}$ is a line which contains one of these segments of $\partial U_{d}$.  We need  the following preparatory lemma.  

\begin{lemma}\label{lem:prepgeoalg}
For $d\in \{1,2,3\}$ let $z \in \mathbb Q(\sqrt{-d})$.  Then the reflection $\hat{z}$ of $z$ across 
$L_{d}$ is contained in $\mathbb Q(\sqrt{-d})$ and the Ford spheres $\mathcal F_d(z)$ 
and $\mathcal F_d(\hat{z})$ have the same radius.  
\end{lemma} 
\begin{proof} 
Suppose that $L_{d}$ is the line $\Re(w) = \tfrac{1}{2}$.  Then $\hat{z} = 
-(\overline{z}-1)$. If we write $z = \tfrac{P}{Q}$ with $\gcd(P,Q)=1$, then $\hat{z} = \tfrac{\overline{Q} - 
\overline{P}}{\overline{Q}}$ and, hence, $|Q| = |\overline{Q}|$. This implies the assertion of the 
lemma.  For $d=1$ the other possible lines are $\omega_1 L_{d}$, $k \in \{1, 2, 3\}$ and the assertion follows analogously for these lines. In the case $d=2$, if $L_d$ is one of the lines satisfying $\Im(w)=\pm \frac{\omega_2}{2}$, then $\hat z= \overline{z}\pm \omega_2$. For $d=3$, the other possible lines are given by $\omega_{3}^{k} L_{d}$, $k\in\{1, 2, 3, 4, 5\}$. For each of these lines the lemma follows by similar reasoning.   
\end{proof}

We can now prove the following main result of this section.

\begin{theorem}
Let $d \in \{1, 2, 3\}$. Then the nearest integer complex continued fraction map~$T_d$ always produces geodesic continued fraction expansions. 
\end{theorem}

\begin{proof}
Let $\frac{P}{Q} \in U_{d}$ be given. Suppose that
$\frac{P}{Q}= [a_1,\ldots, a_n]$ is the nearest integer continued fraction expansion of 
$\frac{P}{Q}$. If $\frac{P_k}{Q_k}$ is the $k$-th convergent of $\frac PQ$, in view of Theorem~\ref{th:correspondencePathCF} this expansion corresponds to the walk
\begin{equation}\label{eq:HurwitzPath}
\infty=\frac{P_{-1}}{Q_{-1}} \to 0=\frac{P_0}{Q_0} \to \frac{P_1}{Q_1} \to \frac{P_2}{Q_2} \to \cdots \to \frac{P_n}{Q_n}=\frac{P}{Q}
\end{equation}
in the Farey graph $\mathcal{E}_{d}$. Now suppose that
$\infty \to r_0 \to r_1 \to \cdots \to r_{\ell}=\frac{P}{Q}$ is a geodesic path in $\mathcal{E}_{d}$. We need to show that \eqref{eq:HurwitzPath} is also a geodesic path in $\mathcal{E}_{d}$. If these two walks are equal there is nothing to do. If not, we choose $m$ such that $\frac{P_{k}}{Q_{k}}=r_{k}$ holds for $0 \leq k <m$, and $\frac{P_m}{Q_m} \neq r_m$. Thus $\infty \to \frac{P_0}{Q_0}\to\cdots\to \frac{P_{m-1}}{Q_{m-1}}$ is geodesic. Consider the linear fractional map 
$$
S=\begin{pmatrix}
0 & 1 \\
1 & a_{m-1}
\end{pmatrix}^{-1}
\begin{pmatrix}
0 & 1 \\
1 & a_{m-2}
\end{pmatrix}^{-1} 
\cdots 
\begin{pmatrix}
0 & 1 \\
1 & a_1
\end{pmatrix}^{-1}
$$
($S$ is the identity if $m=0$).  Then $S\big(\frac{P_{m-2}}{Q_{m-2}}\big)=S(r_{m-2})=\infty$, $S\big(\frac{P_{m-1}}{Q_{m-1}}\big)=S(r_{m-1})=0$, and $S\big(\frac{P_m}{Q_m}\big)=\frac{1}{a_m}$. 
Since M\"obius transformations preserve adjacency in $\mathcal{E}_d$ this implies that $0 \to S\left(r_m\right)$ and, hence, $\infty \to S\left(r_m\right)^{-1}$. Thus
$S\left(r_m\right)^{-1}\in\mathbb Z[\omega_{d}]$. By the choice of $m$ we have  $S\left(r_m\right)^{-1}\not=a_m$ and, hence, $S\left(r_m\right)^{-1}-a_m \notin U_d$. 
Writing $r_k^{(1)}=S\left(r_k\right)^{-1}-a_m$ for $k \geq m$ we see that  
\begin{equation}\label{eq:rprimepath}
\infty \to r_m^{(1)} \to \cdots \to r_\ell^{(1)}
\end{equation}
is a geodesic path in $\mathcal{E}_d$. We have $r_m^{(1)}\not\in{U_{d}}^{\mathrm{cl}}$ and  
\[
r_{\ell}^{(1)}=S(r_\ell)^{-1}-a_m= S\big(\frac PQ\big)^{-1}-a_m= T_{d}^{m-1}\big(\frac PQ\big)^{-1}-a_m = T_{d}^m\big(\frac PQ\big) \in U_{d}.
\]
Thus there exists $t\in\{m+1,\ldots,\ell\}$ such that $r_s^{(1)}\not\in {U_{d}}^{\mathrm{cl}}$ for $m \leq s<t$ and $r_t^{(1)}\in {U_{d}}^{\mathrm{cl}}$.  The following three cases can occur.
\begin{itemize}
\item[(a)]\,  $r_t^{(1)}$ is symmetric to $r_{t-1}^{(1)}$ w.r.t. some line $L_d$ containing a line segment of $\partial U_{d}$.
\item[(b)] \, $r_t^{(1)}$ not on $\partial U_{d}$  and not symmetric to $r_{t-1}^{(1)}$ w.r.t.\ any line $L_d$ containing a line segment of $\partial U_{d}$.
\item[(c)]\, $r_t^{(1)}$ is on some line $L_d$ containing a line segment of $\partial U_{d}$.
\end{itemize}

Case (a): For each $s\in\{m,\ldots, t-1\}$ let $r_s^{(2)}$ be symmetric to $r_s^{(1)}$ w.r.t.\ $L_d$. Then $r_{t-1}^{(2)}=r_t^{(1)}$. By Lemma~\ref{lem:prepgeoalg}, this yields the walk 
$\infty\to r_m^{(2)}\to \cdots \to r_{t-1}^{(2)}=r_t^{(1)}\to r_{t+1}^{(1)} \to \cdots \to r_\ell^{(1)}$ that is shorter than the geodesic path in \eqref{eq:rprimepath}, a contradiction.

Case (b): We consider the projections of the Ford spheres $G_{t-1}$ and $G_t$ associated with $r_{t-1}^{(1)}$ and $r_t^{(1)}$ to $\mathbb{C}$. These projections are circles. These circles cannot cross a line $L_d\subset \partial U_{d}$. Indeed, if the Ford sphere $G$ of some point $p\not\in L_d$ would cross $L_d$, it would intersect the Ford sphere $G'$ of the reflection of $p$ across $L_d$ (because, by Lemma~\ref{lem:prepgeoalg}, $G$ and $G'$ have the same radius), a contradiction.  So the Ford spheres $G_{t-1}$ and $G_t$ cannot be tangent to each other. Thus $r_{t-1}^{(1)}$ and $r_t^{(1)}$ are not adjacent in $\mathcal{E}_d$, a contradiction. 

Thus only case (c) can occur. In this case by Lemma~\ref{lem:prepgeoalg} we can construct a walk $\infty\to r_1^{(2)}\to \cdots\to r_t^{(2)}=r_t^{(1)}$, where $r_s^{(2)}$ is symmetric to $r_s^{(1)}$ w.r.t.\ $L_d$ for $s\in \{1,\ldots, t\}$. This walk has the same length as  the geodesic path $\infty\to r_1^{(1)}\to \cdots\to r_t^{(1)}$. Setting $r_k^{(2)}=r_k^{(1)}$ for $t \le k \le \ell$ yields a geodesic path $\infty \to r_m^{(2)}\to \cdots \to r_\ell^{(2)}$ and $t_2<t$ such that $r_s^{(2)} \not\in {U_{d}}^{\mathrm{cl}}$ for $m \leq s<t_2$ and $r_{t_2}^{(2)}\in {U_{d}}^{\mathrm{cl}}$. Iterating this procedure yields a geodesic path $\infty \to r_m^{(n)} \to \cdots \to r_\ell^{(n)}$ such that  $r_{m}^{(n)}\in{U_d}^{\mathrm{cl}}$, {\em i.e.} $r_m^{(n)}=0$. Thus
\[
\begin{split}
&\infty \to r_0 \to r_1 \to \cdots \to r_{m-1} \to S^{-1}\bigg(\frac1{a_m + r_m^{(n)}}\bigg)  \to  \cdots \to S^{-1}\bigg(\frac1{a_m + r_\ell^{(n)}}\bigg)
\end{split}
\]
is a  geodesic path from $\infty$ to $S^{-1}\bigg(\frac1{a_m + r_\ell^{(n)}}\bigg)=r_\ell=\frac PQ$ with $r_0=\frac {P_0}{Q_0},\ldots, r_{m-1}=\frac {P_{m-1}}{Q_{m-1}}$ and $S^{-1}\bigg(\frac1{a_m + r_m^{(n)}}\bigg)=\frac{P_m}{Q_m}$. Thus, in particular, $\infty \to \frac{P_0}{Q_0}\to\cdots\to \frac{P_{m}}{Q_{m}}$ is geodesic. The fact that \eqref{eq:HurwitzPath} is geodesic now follows by induction on $m$.
\end{proof}

\begin{remark} 
For the continued fraction algorithms $\mathrm{CF}(7, \mathrm{H})$ and $\mathrm{CF}(11, \mathrm{H})$ defined in Lakein~\cite{lakein1973approximation} for $d=7$ and $d=11$, respectively, the point $z \in U_{d}$ and its symmetric point $\hat z$ (see Lemma~\ref{lem:prepgeoalg} above) may have Ford spheres of different size, which causes problems.
\end{remark}

The following necessary conditions for a continued fraction expansion to be geodesic correspond to the conditions for the Hecke group case studied by Short and Walker \cite{S-W}. 

\begin{proposition}  \label{any-d}
Let $d \in\{ 1, 2, 3, 7, 11\}$ be given and suppose that $z = [a_{0}; a_{1}, a_{2}, \ldots ]$, $a_{n} \in \mathbb Z[\omega_{d}]$, $n \ge 0$, is geodesic. Then the following conditions hold for $n\ge 1$.
\begin{itemize}
\item[(i)]\;$|a_{n}| \ne 1$.
\item[(ii)]\;$|a_{n}| = \sqrt{2}$ implies $a_{n+1} \ne - \overline{a_{n}}$.
\item[(iii)]\; If $\zeta$ is a unit, then $a_{n} = 2 \zeta$ implies $a_{n+1} \ne - 2 \overline{\zeta}$. 
\end{itemize}
\end{proposition}
\begin{proof}
Let $n\ge 1$. We have  $[a_0;a_{1}, a_{2}, \ldots, a_{n-1}, a_{n}] = [a_0;a_{1}, a_{2}, \ldots, a_{n-1} + \overline{a_{n}}]$ 
if $|a_{n}|=1$ and, hence, $[a_0;a_{1}, a_{2}, \ldots, a_{n}]$ is not geodesic. Also, if $|a_{n}| = \sqrt{2}$ we have
$[a_0;a_{1}, a_{2}, \ldots, a_{n-1}, a_{n}, -\overline{a_{n}}] = 
[a_0;a_{1}, a_{2}, \ldots, a_{n-1} + \overline{a_{n}}]$, and,  if $\zeta$ is a unit we see that
$[a_0;a_{1}, a_{2}, \ldots, a_{n-1}, 2 \zeta, -2\overline{\zeta}] = [a_0;a_{1}, a_{2}, \ldots, a_{n-1} +  \overline{\zeta}, - 3 \zeta]$.  In any of these cases we conclude that $[a_0;a_{1}, a_{2}, \ldots, a_{n},a_{n+1}]$ is not geodesic.   
\end{proof}

Proposition~\ref{any-d} is applicable to the partial quotients $\pm 1$, $\pm i$, $\pm(1+i)$, and $\pm(1-i)$ for $d=1$, to the partial quotient $\pm\sqrt{-2}$ for $d=2$, and to $\pm \omega_{7}$ and $\pm \overline{\omega}_{7}$ for $d=7$. 
%%%%%%%%%%%%%%%%%%%%%%%%%%%%%%%%%%%%
%%%%%%%%%%%%%%%%%%%%%%%%%%%%%%%%%%%%%%

\begin{proposition}\label{prop:some-d}
For $d \in\{2, 3, 11\}$ suppose that $z = [a_{0}; a_{1}, a_{2}, \ldots ]$, $a_{n} \in \mathbb Z[\omega_{d}]$ for $n \ge 0$, is geodesic. Then 
$|a_{n}| = \sqrt{3}$, $a_{n+1} = - \overline{a_{n}}$, $a_{n+2} = a_{n}$ implies that
$a_{n+3} \ne -\overline{a_{n}}$. 
\end{proposition}
\begin{proof} 
Because $[a_0;a_{1}, a_{2}, \ldots, a_{n-1}, \eta, -\overline{\eta}, \eta, -\overline{\eta}] = 
[a_0;a_{1}, a_{2}, \ldots, a_{n-1} + \overline{\eta}] $ holds for $|\eta| = \sqrt{3}$, the result follows. 
\end{proof}

Proposition~\ref{prop:some-d} can be applied to the partial quotients $\pm(1 + \sqrt{-2})$ and $\pm(1 - \sqrt{-2})$ for $d=2$,  to $\pm(\omega_{3}+1)$, $\pm \sqrt{-3}$, and $\pm( \overline{\omega}_{3} + 1)$ for 
$d=3$, and to $\pm \omega_{11}$ and $\pm\overline{\omega}_{11}$ for $d =11$.  
%%%%%%%%%%%%%%%%%%%%%%%%%%%%%%%%%%
In \cite{H-2} it is claimed that for $d=1$ the restrictions in Propositions~\ref{any-d} and~\ref{prop:some-d} are also sufficient for an expansion to be geodesic.  This might be true also for $d= 2, 3$, however, by the following example, for $d=7$ these conditions are not sufficient.
%%%%%%%%%%%%%%%%%%%%%%
%%%%%%%%%%%%%%%%%%%%%%
\begin{example}[A subtle nongeodesic path for the case $d=7$]\label{ex:exd7}
\, Consider the six elements $v_1=0$, $v_2=\frac{1 +  \sqrt{-7}}4 $, $v_3=\frac{7 + 3 \sqrt{-7}}{14} $, $v_4=\frac{25 + 11 \sqrt{-7}}{46} $, $v_5=\frac{9(7 + 3 \sqrt{-7})}{112}$, $v_6=\frac{61 + 26 \sqrt{-7}}{107}$ of $\mathbb{Q}(\sqrt{-7})$. These elements correspond to the six black dots in Figure~\ref{fig:7reflection6}. It can be seen from this figure that these points form the vertices of the path $P_1: \infty\to v_1\to v_2\to v_3\to v_4\to v_5\to v_6$ in $\mathcal{E}_7$. According to Theorem~\ref{th:correspondencePathCF}, the path $P_1$ corresponds to a continued fraction expansion, namely, we get
\begin{equation}\label{eq:nonGeoEx}
v_6 = [0;-\omega_7+1,-\omega_7+2, \omega_7-2,2\omega_7-1,\omega_7].
\end{equation}
The point $v_6$ corresponds to the rightmost of the six points in Figure~\ref{fig:7reflection6}. It can be seen from Figure~\ref{fig:7reflection6} that $v_6$ can be reached faster by the path  $P_2:\infty\to 1\to \frac{3 + \sqrt{-7}}4\to \frac{7 + 3 \sqrt{-7}}{12}\to v_6$ in $\mathcal{E}_7$. Thus the continued fraction expansion in \eqref{eq:nonGeoEx} is nongeodesic. More precisely, we get by applying  Theorem~\ref{th:correspondencePathCF} to $P_2$ that
\[
v_6=[1;-\omega_7,-2 \omega_7 - 1,1 - 2 \omega_7]=[0;-\omega_7+1,-\omega_7+2, \omega_7-2,2\omega_7-1,\omega_7].
\]

Thus the string $(-\omega_7+1,-\omega_7+2, \omega_7-2,2\omega_7-1,\omega_7)$ is forbidden (non-geodesic).  
\end{example}
%%%%%%%%%%%%%%%%%%%%%% 

We end this section with a result on geodesicness of reversed expansions.

\begin{proposition}\label{prop:reverse}
If $R = [a_{1},a_{2},\ldots,a_{n}]$ is a  geodesic continued fraction expansion, then also $R^{*} = [a_{n}, a_{n-1},\ldots,a_{1}]$
is a geodesic continued fraction expansion.
\end{proposition}

\begin{proof}
By Proposition~\ref{any-d} we may assume that $|a_{1}| \ne 1$. We consider the geodesic path 
\begin{equation}\label{eq:path1}
\infty \to 0 \to [a_1] \to [a_{1},a_2] \to \cdots \to[a_{1},a_{2},\ldots 
a_{n-1},a_{n}]
\end{equation}
in $\mathcal{E}_d$ which corresponds to the continued fraction expansion 
$R = [a_{1},a_{2},\ldots,a_{n}]$ by Theorem~\ref{th:correspondencePathCF}.
Now consider $R^{-1}-a_1= [a_{2},\ldots,a_{n}]$. If we apply inversion and subtraction of $a_1$ to the walk in \eqref{eq:path1}, then we obtain the walk
\[
-a_{1} \to \infty \to 0 \to [a_2] \to [a_{2},a_{3}] \to \cdots \to[a_{2},a_{3},\ldots 
a_{n-1},a_{n}]
\]
in $\mathcal{E}_d$. Continuing in the same way, we have 
\[
-\left( a_{2}  + \tfrac{1}{a_{1}}\right) \to -a_{2} \to \infty \to 0  \to [a_3] \to [a_{3},a_{4}] \to \cdots \to[a_{3},a_{4},\ldots 
a_{n-1},a_{n}],
\] 
and repeating this for $n$ times we end up with 
\[
-[a_n;a_{n-1},\ldots, a_1] \to -[a_n;a_{n-1},\ldots, a_2] \to -[a_n;a_{n-1}] \to -a_n \to \infty \to 0.
\]
Taking inverses, this yields 
\[
-[a_n,a_{n-1},\ldots, a_1] \to -[a_n,a_{n-1},\ldots, a_2] \to -[a_n,a_{n-1}] \to -[a_n] \to 0 \to \infty.
\]
Suppose that there exists a shorter walk from $R^{*}$ to $\infty$, or, equivalently a shorter walk from $-R^{*}$ to $\infty$, {\it i.e.}, a walk $-R^{*} \to t_{1} \to t_{2} \to \cdots \to t_{\ell-1} \to \infty$ with $\ell < n$.  There exists a M\"obius transformation $L$ such that $L(-R^{*})= \infty$ and $L(\infty)=R$. Applying this map to $-R^{*} \to t_{1} \to \cdots \to t_{\ell-1} \to \infty$ yields the walk $\infty \to L(t_{1}) \to \cdots \to L(t_{\ell-1}) \to R$. This contradicts to the assumption of \eqref{eq:path1} being a geodesic path.
\end{proof}

Let $[a_1,a_2,a_3,\ldots]$ be a continued fraction expansion. Proposition~\ref{prop:reverse} is of interest if one studies the pairs 
$([a_{n}, a_{n-1}, \ldots, a_{1}], [a_{n+1}, a_{n+2}, \ldots])$, $n\ge 1$,
which play an important role in various contexts in the theory of continued fractions, both for real numbers and for complex numbers.  In particular, in the construction of the natural extension of a continued fraction map, Proposition~\ref{prop:reverse} shows that a geodesic expansion remains geodesic in the natural extension. Here, the natural extension is defined on the closure of $\{([a_{n}(z), a_{n-1}(z), \ldots, a_{1}(z)], [a_{n+1}(z), a_{n+2}(z), \ldots]) : n \ge 1, z \in U\}$, where $U$ is the domain of the underlying continued fraction map. For example, we refer to  Ei {\em et al.} \cite{E-N-N} in the case of the nearest integer complex continued fraction maps $T_d$.

\section{Large partial quotients imply geodesicness}

In the present section we show that large partial quotients always give rise to geodesic continued fraction expansions. 
Let $\pi:\mathbb{H}^3\to\mathbb{C}:\, (z,x)\mapsto z$ be the projection from $\mathbb{H}^3$ to the complex plane~$\mathbb{C}$.

\begin{theorem}\label{th:largea_kgeodesic}
Let $d\in\{1,2,3,7,11\}$ and let $x=[a_1,\ldots, a_n]$ be a continued fraction expansion of $x$ satisfying $a_\ell\in\mathbb{Z}[\omega_d]$ with $|a_\ell| \ge  4.17209$ for each $\ell\in\{1,\ldots,n\}$. Then $[a_1,\ldots, a_n]$ is the unique geodesic continued fraction expansion of $x$.
\end{theorem}

The theorem is proved by the following two lemmas.

\begin{lemma}\label{lem:circleGeometry}
Let $G_1=\mathcal{F}_d(a_1)$ and $G_2=\mathcal{F}_d(a_2)$ be two (not necessarily tangent) Ford spheres of radius $\rho_1$ and $\rho_2$, respectively. Suppose that $\rho_1 > \rho_2$ and assume that 
\begin{equation}\label{eq:distanceBasePointEstimate}
|a_1-a_2| \le (2-\sqrt{2}) \rho_1. 
\end{equation}
Let $G=\mathcal{F}_d(a)$ be a Ford sphere tangent  to $G_2$. Then $\pi(G) \subset \pi(G_1)$.
\end{lemma}

\begin{figure}[h]
\begin{tikzpicture}[scale=0.3]
\node[draw=none,fill=none] at (0,0){\includegraphics[width=.42\textwidth]{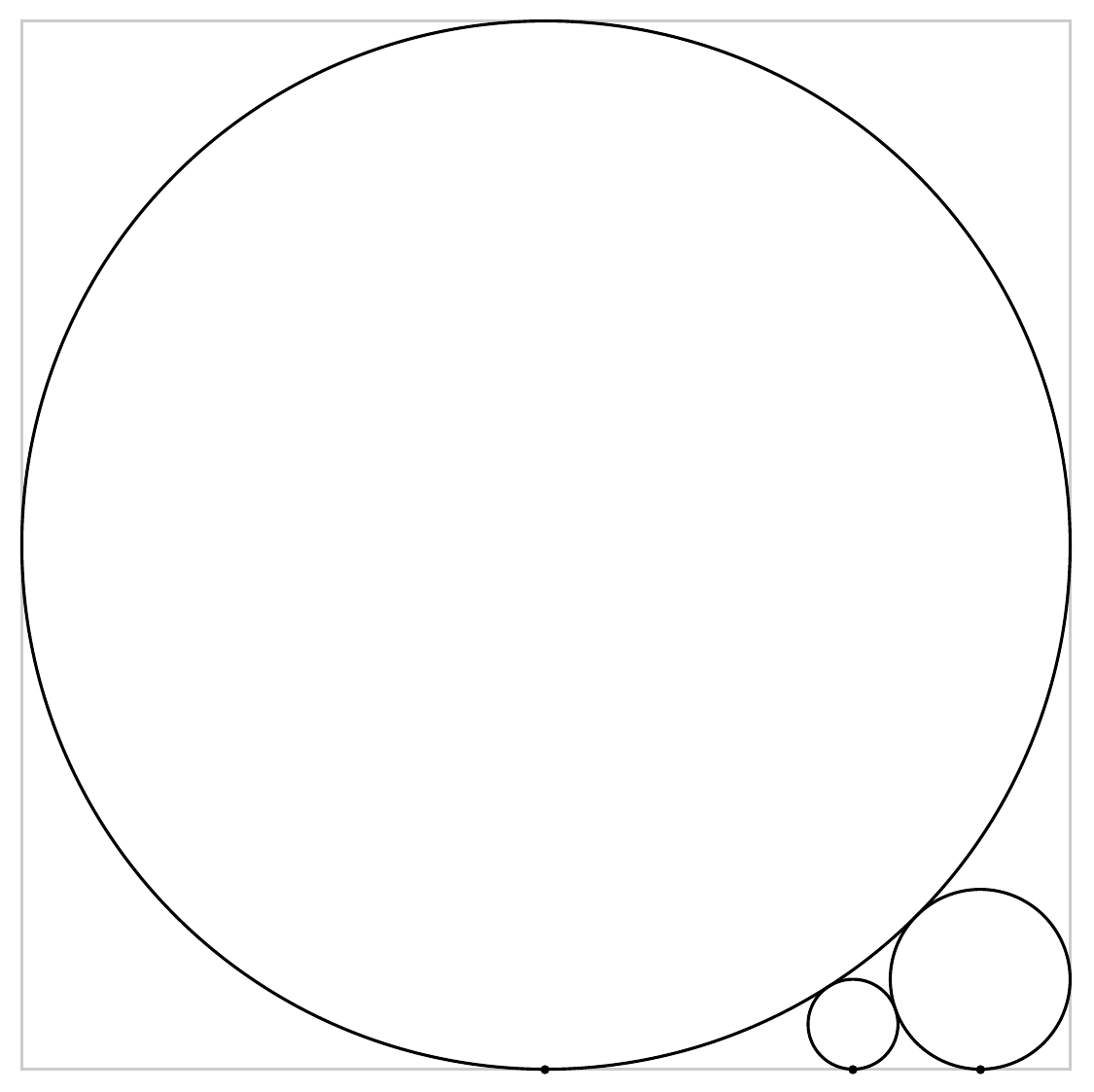}};
\node[draw=none,fill=none] at (20,-0.5){\includegraphics[width=.5\textwidth]{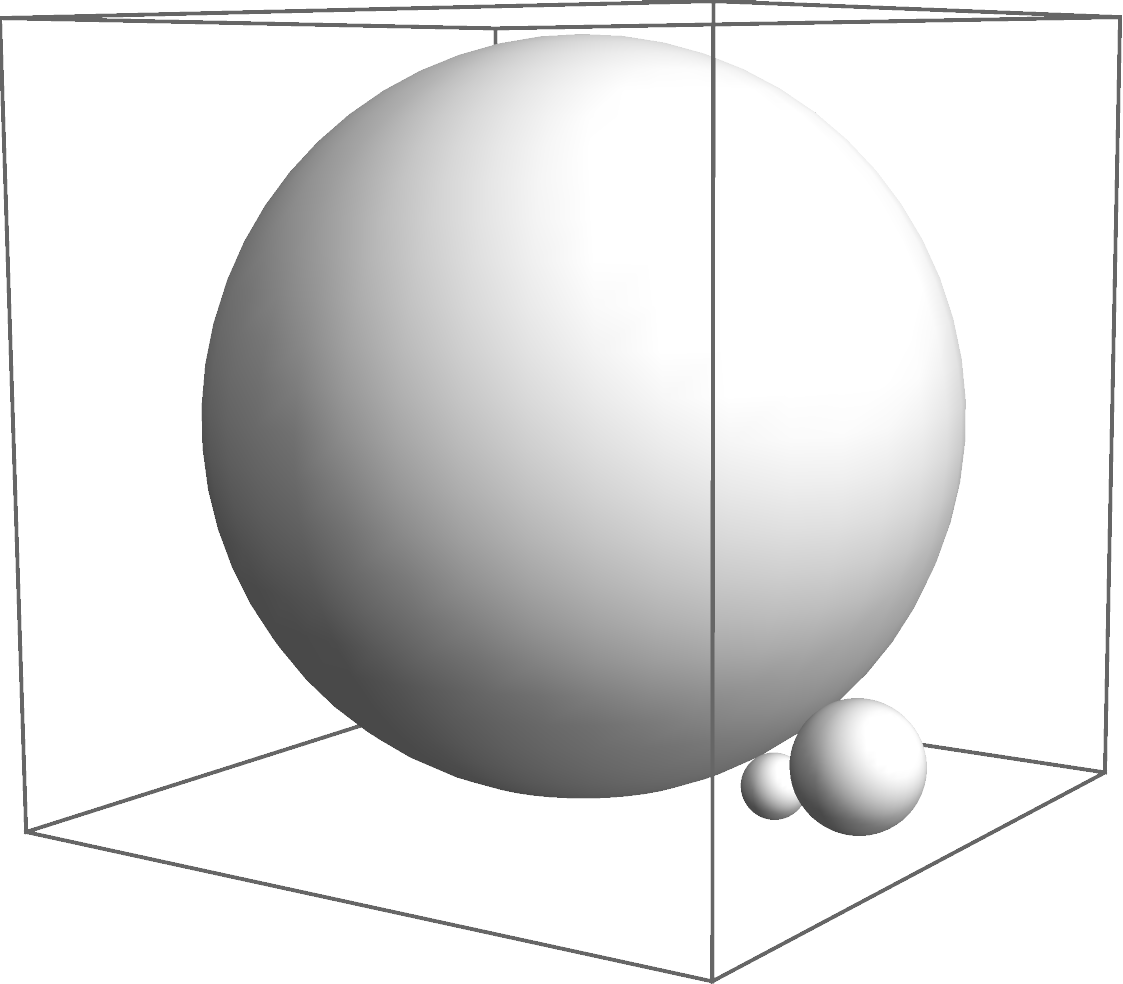}};

\node at (0,0) {$G_1$};
\node at (4.73,-7.4) {$G_2$};
\node at (6.7,-6.8) {$G$};
\node at (0,-9) {$a_1$};
\node at (4.73,-9) {$a_2$};
\node at (6.7,-9) {$a$};

\node at (20.5,1) {$G_1$};
\node at (23.5,-7) {$G_2$};
\node at (25.3,-5.4) {$G$};
\end{tikzpicture}
\caption{An extreme situation in Lemma~\ref{lem:circleGeometry}.
\label{fig:largestC}}
\end{figure}
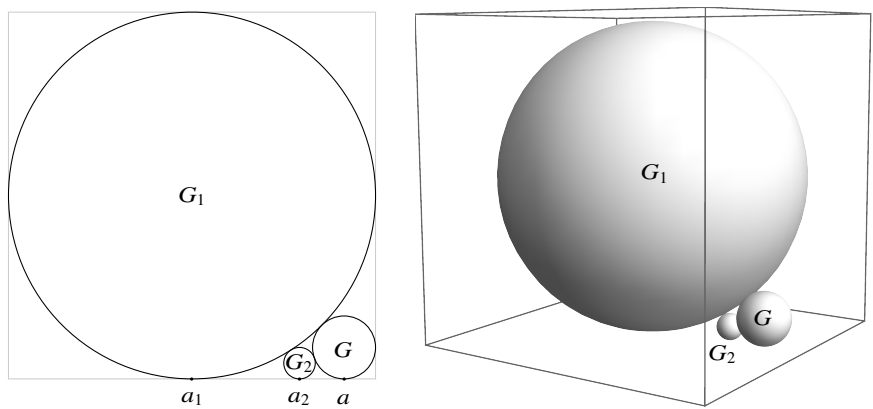

In Figure~\ref{fig:largestC} the constellation with the largest possible radius of $G$ is shown (provided that $G\not=G_1$). The left hand side shows a 2-dimensional cut, the right hand side gives an impression of the 3-dimensional situation.

\begin{proof}
Set $c=2-\sqrt{2}$ for convenience and let $\rho$ be the radius of $G$. We may assume that $G\not=G_1$ because otherwise the lemma is trivially true. Applying Pythagoras' theorem to the triangle indicated in Figure~\ref{fig:Pythagoras}, we get $|a_1-a_2| \ge 2\sqrt{\rho_1\rho_2}$ (note that $G_1$ and $G_2$ are disjoint or tangent; equality only holds if $G_1$ and $G_2$ are tangent). Therefore $\rho_2\le \frac{|a_1-a_2|^2}{4\rho_1}$, and, hence,  \eqref{eq:distanceBasePointEstimate} yields that $\rho_2 \le \frac{c^2}4 \rho_1$. 
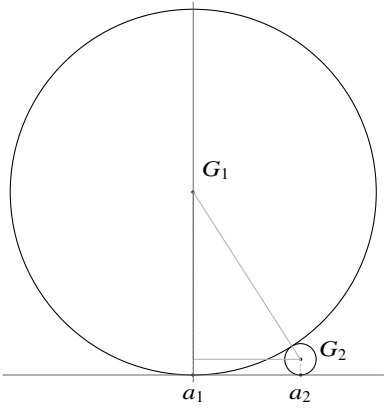
\begin{figure}[h] 
\begin{center}
\begin{tikzpicture}[scale=0.3]
\node[draw=none,fill=none] at (0,0){\includegraphics[width=.42\textwidth]{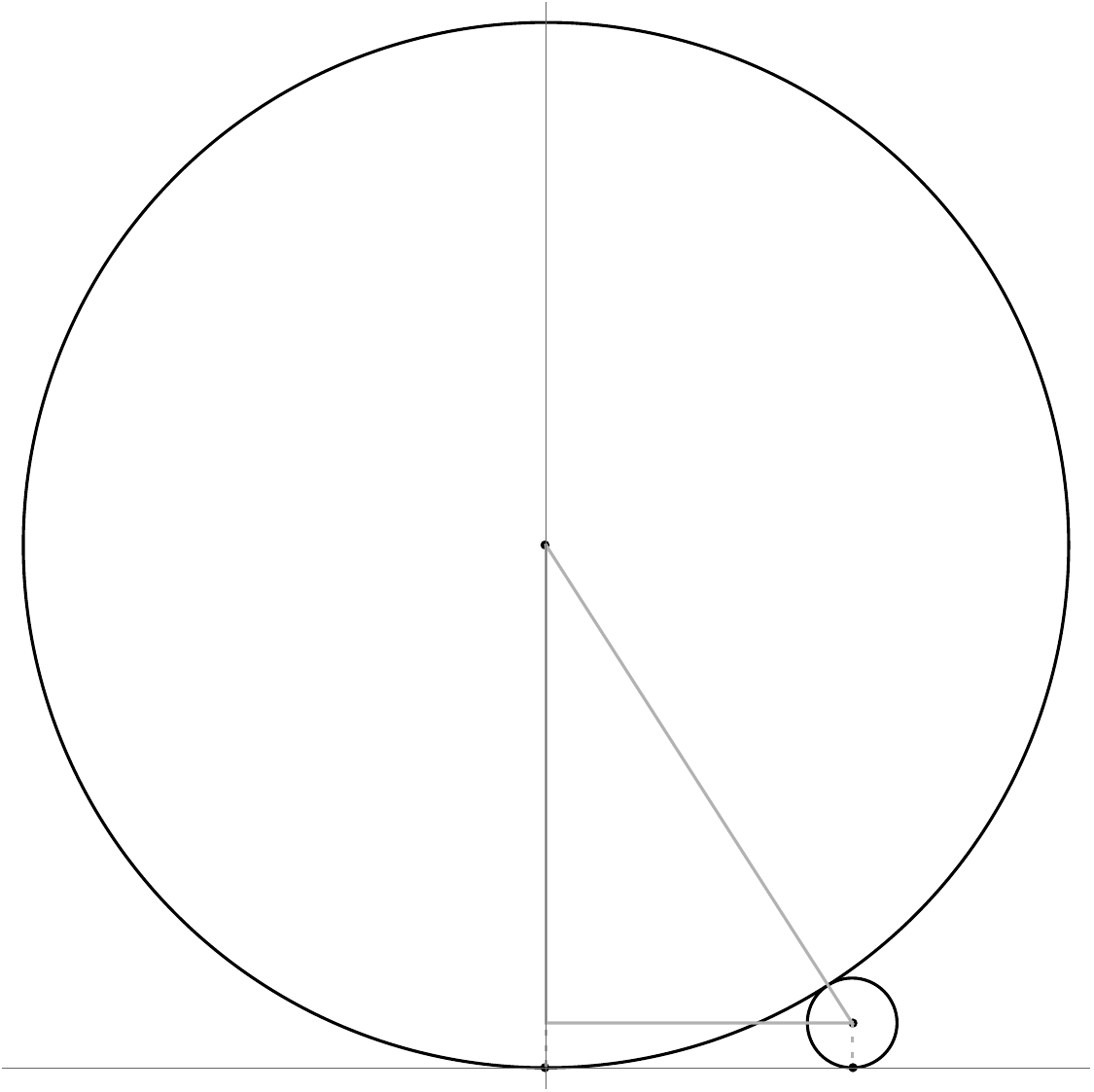}};

\node at (1,1) {$G_1$};
\node at (6.2,-7) {$G_2$};
\node at (0,-9) {$a_1$};
\node at (4.73,-9) {$a_2$};
\end{tikzpicture}
\end{center}
\caption{An application of Pythagoras' theorem.
\label{fig:Pythagoras}}
\end{figure}
Applying Pythagoras' theorem again and using this estimate, we see that, because $G_2$ and $G$ are tangent, the distance between $a_2$ and $a$ satisfies
\[
|a_2 - a| = 2 \sqrt{\rho_2\rho} \le  c \sqrt{\rho_1\rho}.
\]
Thus 
\[
|a_1- a| \le  |a_1-a_2| + |a_2-a| = c(\rho_1 + \sqrt{\rho_1\rho})=c\bigg(1+\sqrt{\frac{\rho}{\rho_1}}\bigg)\rho_1.
\]
 Applying Pythagoras' theorem for a third time, this implies that (the first inequality becomes an equality if $G$ and $G_1$ are tangent)
\[
\rho \le \frac{|a_1-a|^2}{4\rho_1} \le \frac{c^2\bigg(1+\sqrt{\frac{\rho}{\rho_1}}\bigg)^2}4 \rho_1 \quad\text{and, therefore, }\quad
\rho \le \frac{c^2}{(2-c)^2}\rho_1.
\]
By the choice of $c$ we get 
\[
|a-a_1| +\rho \le c\bigg(1+\sqrt{\frac{\rho}{\rho_1}}\bigg)\rho_1 + \frac{c^2}{(2-c)^2}\rho_1
\le c(1+\frac{c}{2-c})\rho_1+ \frac{c^2}{(2-c)^2}\rho_1 = \rho_1.
\]
Therefore, $G$ lies ``below'' $G_1$, or, to be more precise, $\pi(G) \subset \pi(G_1)$, and the proof is finished.
\end{proof}

\begin{lemma}\label{lem:keyStepGeo}
Let $x=[a_1,\ldots, a_n]$ be a continued fraction expansion of $x$ satisfying $a_\ell\in\mathbb{Z}[\omega_d]$ with $|a_\ell| \ge4.17209$ for each $\ell\in\{1,\ldots,n\}$. Let $x=[b_1,\ldots, b_m]$ be another continued fraction expansion of $x$. 
Then $m>n$, {\em i.e.}, $[a_1,\ldots, a_n]$ is the unique geodesic expansion of $x$.
\end{lemma}

\begin{proof}
For convenience, we set $C=4.17209$. Let $\frac{p_k}{q_k}$ be the convergents of $[a_1,\ldots, a_n]$ and let $\frac{p_k'}{q_k'}$ be the convergents of $[b_1,\ldots, b_m]$.  Define the Ford spheres $F_k=\mathcal{F}_d\big(\frac{p_k}{q_k}\big)$ of radius $\frac1{2|q_k|^2}$ and $F_k'=\mathcal{F}_d\big(\frac{p_k'}{q_k'}\big)$ of radius $\frac1{2|q'_k|^2}$. We claim that $\pi(F_{m-k}') \subset \pi(F_{n-k})$ holds for each $k\le n$. We use induction to prove this claim. The induction start $k=0$ is trivial because $\frac{p_m'}{q_m'}=x=\frac{p_n}{q_n}$ yields  $F_m'=F_n$. Our aim is to use Lemma~\ref{lem:circleGeometry} for the induction step. Suppose that $\pi(F_{m-k}') \subset \pi(F_{n-k})$ holds for some $k$. Then, applying Pythagoras' theorem, we gain  
\begin{equation}\label{eq:diffpkqkest}
\begin{split}
\bigg|\frac{p_{n-1-k}}{q_{n-1-k}} - \frac{p_{m-k}'}{q_{m-k}'}\bigg|
&\le 
\bigg|\frac{p_{n-1-k}}{q_{n-1-k}} - \frac{p_{n-k}}{q_{n-k}}\bigg|
+
\bigg|\frac{p_{n-k}}{q_{n-k}} - \frac{p_{m-k}'}{q_{m-k}'}\bigg| \\
&\le 
2\sqrt{\frac{1}{2|q_{{n-1}-k}|^2} \frac{1}{2|q_{n-k}|^2}}
+
\frac{1}{2|q_{n-k}|^2} - \frac{1}{2|q_{m-k}'|^2}\\
&=
\frac{1}{|q_{{n-1}-k}|} \frac{1}{|q_{n-k}|}
+
\frac{1}{2|q_{n-k}|^2} - \frac{1}{2|q_{m-k}'|^2} \\
&\le \frac{1}{|q_{{n-1}-k}|} \frac{1}{|q_{n-k}|} +
\frac{1}{2|q_{n-k}|^2}.
\end{split}
\end{equation}
Note that $q_{\ell+1}=a_\ell q_\ell + q_{\ell-1}$ and, because $|a_\ell| \ge C$, the sequence $(|q_\ell|)_{\ell \ge 0}$ is monotone and, hence, $|q_{\ell+1}| \ge (C-1) |q_\ell|$. Therefore, $|q_{\ell+1}| \ge C|q_{\ell}| - \frac{|q_\ell|}{C-1} = \big(C-\frac1{C-1}\big)|q_\ell|$.
By the choice of $C$ we now get from \eqref{eq:diffpkqkest} that
\begin{equation*}
\begin{split}
\bigg|\frac{p_{n-1-k}}{q_{n-1-k}} - \frac{p_{m-k}'}{q_{m-k}'}\bigg|
&\le 
\frac{1}{\big(C-\frac{1}{C-1}\big)|q_{{n-1}-k}|^2} + \frac{1}{\big(C-\frac{1}{C-1}\big)^22|q_{{n-1}-k}|^2}
\\&=
\bigg( \frac{2}{C-\frac{1}{C-1}} + \frac{1}{\big(C-\frac{1}{C-1}\big)^2}
\bigg)\frac{1}{2|q_{{n-1}-k}|^2}
 < \frac{2-\sqrt{2}}{2|q_{{n-1}-k}|^2}.
\end{split}
\end{equation*}
Thus we may apply Lemma~\ref{lem:circleGeometry} with $G_1=F_{n-1-k}$, $G_2=F'_{m-k}$ and $G=F'_{m-1-k}$ to obtain $\pi(F_{m-1-k}') \subset \pi(F_{n-1-k})$ and the induction proof is finished.

We proved that $\pi(F'_{m-k}) \subset \pi(F_{n-k})$ holds for each $k \le n$. Thus 
$|q'_{m-k}| \ge |q_{n-k}| > |q_0|=1$ holds for $k\in\{0,\ldots,n-1\}$. Therefore, we need to go $k\ge n$ steps back to reach a Ford circle $F'_{m-k}$ of radius $\frac12$. Hence, $m\ge n$, and thus $[a_1,\ldots, a_n]$ is a geodesic expansion for $x$.

To show uniqueness observe that, if, for some $k$, the inclusion $\pi(F'_{m-k}) \subsetneq \pi(F_{n-k})$ is strict, then $\pi(F'_{m-\ell}) \subsetneq \pi(F_{n-\ell})$ for each $\ell \ge k$ and, hence, $m>n$. This implies that $[a_1,\ldots, a_n]$ is the unique geodesic expansion for $x$.
\end{proof}

Theorem~\ref{th:largea_kgeodesic} follows immediately from Lemma~\ref{lem:keyStepGeo}.
\subsubsection*{Acknowledgments}
The first and second authors were partially supported by JSPS Grants 24K03611 and 25K07035, respectively. The third author was supported by the ANR-FWF grant ``SymDynAr'' (FWF I 6750).

We thank Andrei Zabolotskii for bringing the references \cite{hatcher1983hyperbolic,S-vS-Z} to our attention.

%%%%%%%%%%%%%%%%%%%%%%%% referenc.tex %%%%%%%%%%%%%%%%%%%%%%%%%%%%%%
% sample references
% %
% Use this file as a template for your own input.
%
%%%%%%%%%%%%%%%%%%%%%%%% Springer-Verlag %%%%%%%%%%%%%%%%%%%%%%%%%%
%
% BibTeX users please use
% \bibliographystyle{}
% \bibliography{}
%

\end{document}